\documentclass[12pt, a4paper]{amsart}
\usepackage{amsmath}
\usepackage{latexsym,amsfonts,amssymb,amsmath,amsxtra,bbm,color,mathrsfs}
\usepackage{amscd}
\usepackage{mathdots}
\allowdisplaybreaks[4]

\usepackage{braket}

\usepackage{hyperref}
\usepackage{pdfsync}
\usepackage[all]{xypic}

\usepackage{verbatim}

\newcommand{\BC}{{\mathbb {C}}}

\newcommand{\BN}{{\mathbb {N}}}

\newcommand{\BR}{{\mathbb {R}}}

\newcommand{\BZ}{{\mathbb {Z}}}

\newcommand{\CB}{{\mathcal {B}}}
\newcommand{\CC}{{\mathcal {C}}}

\newcommand{\CE}{{\mathcal {E}}}
\newcommand{\CF}{{\mathcal {F}}}

\newcommand{\CI}{{\mathcal {I}}}

\newcommand{\CK}{{\mathcal {K}}}

\newcommand{\CM}{{\mathcal {M}}}

\newcommand{\CS}{{\mathcal {S}}}

\newcommand{\CX}{{\mathcal {X}}}

\newcommand{\Cd}{{\mathcal {D}}}

\newcommand{\RB}{{\mathrm {B}}}
\newcommand{\RC}{{\mathrm {C}}}

\newcommand{\RH}{{\mathrm {H}}}

\newcommand{\RK}{{\mathrm {K}}}
\newcommand{\RL}{{\mathrm {L}}}

\newcommand{\RN}{{\mathrm {N}}}
\newcommand{\RO}{{\mathrm {O}}}
\newcommand{\RP}{{\mathrm {P}}}

\newcommand{\RR}{{\mathrm {R}}}

\newcommand{\RT}{{\mathrm {T}}}
\newcommand{\RU}{{\mathrm {U}}}

\newcommand{\RZ}{{\mathrm {Z}}}

\newcommand{\GL}{{\mathrm{GL}}}

\newcommand{\Hom}{{\mathrm{Hom}}}

\newcommand{\Ind}{{\mathrm{Ind}}}

\newcommand{\SL}{{\mathrm{SL}}}

\newcommand{\Sym}{{\mathrm{Sym}}}

\newcommand{\wt}{\widetilde}
\newcommand{\wh}{\widehat}

\newcommand{\bs}{\backslash}

\newcommand{\sgn}{\operatorname{sgn}}

\newcommand{\od}{\operatorname{d}}

\newcommand{\oL}{\operatorname{L}}

\newcommand{\oH}{\operatorname{H}}

\newcommand{\oZ}{\operatorname{Z}}

\renewcommand{\k}{\mathfrak k}

\newcommand{\Z}{\mathbb{Z}}
\newcommand{\C}{\mathbb{C}}
\newcommand{\R}{\mathbb R}

\newcommand{\K}{\mathbb{K}}

\newcommand{\abs}[1]{\lvert#1\rvert}

\newcommand{\be}{\begin {equation}}
\newcommand{\ee}{\end {equation}}
\newcommand{\bee}{\begin {equation*}}
\newcommand{\eee}{\end {equation*}}

\theoremstyle{Theorem}

\theoremstyle{Theorem}

\theoremstyle{Theorem}

\theoremstyle{Theorem}
\newtheorem{prp}{Proposition}[section]
\newtheorem{corp}[prp]{Corollary}
\newtheorem{lemp}[prp]{Lemma}
\newtheorem{thmp}[prp]{Theorem}

\theoremstyle{Plain}

\newtheorem{remarkp}[prp]{Remark}

\theoremstyle{Definition}

\newtheorem{dfnp}[prp]{Definition}

\begin{document}

	\title[Archimedean Period Relations]{Archimedean period relations for Rankin-Selberg convolutions}
	
	\author[Y. Jin]{Yubo Jin}
	\address{Institute for Advanced Study in Mathematics, Zhejiang University\\
		Hangzhou, 310058, China}\email{yubo.jin@zju.edu.cn}

	\author[D. Liu]{Dongwen Liu}
	\address{School of Mathematical Sciences,  Zhejiang University\\
		Hangzhou, 310058, China}\email{maliu@zju.edu.cn}

	\author[B. Sun]{Binyong Sun}
	\address{Institute for Advanced Study in Mathematics and  New Cornerstone Science Laboratory, Zhejiang University,  Hangzhou, 310058, China}\email{sunbinyong@zju.edu.cn}

	\date{\today}

	\maketitle
	
	\begin{abstract}
		We formulate and prove the archimedean period relations for Rankin-Selberg convolutions of $\GL(n)\times\GL(n)$ and $\GL(n)\times\GL(n-1)$, for all generic cohomological representations. As a consequence, we prove the non-vanishing of the archimedean modular symbols. This extends the earlier results in \cite{LLS24}  for essentially tempered representations of  $\GL(n)\times\GL(n-1)$.  
	\end{abstract}
	
	\tableofcontents
	\section{Introduction and main results}
	
	For the study of special values of L-functions by modular symbols, there are three pivotal local ingredients from the representation theory, namely the non-vanishing hypothesis at infinity, the archimedean period relation, and the non-archimedean period relation. Among these three ingredients, the archimedean period relation is the most difficult to prove. For more information, see \cite{LLS24} for example. 
	
	Let $\K$ be an archimedean local field, which is topologically isomorphic to $\R$ or $\C$. Let $n$ and $n'$ be a pair of positive integers such that $n'=n$ or $n-1$. 
	In the case of $n'=n-1$, the archimedean period relation is formulated and proved  in \cite{LLS24} for essentially tempered cohomological representations of $\GL_n(\K)\times \GL_{n'}(\K)$.  In this article, we will extend the result to all generic cohomological  representations of  $\GL_n(\K)\times\GL_{n'}(\K)$. As a consequence of the archimedean period relations, we prove the non-vanishing of the archimedean modular symbols. When $n'=n-1$, the non-vanishing was conjectured by Kazhdan-Mazur in the 1970s. It was proved in \cite{Sun} when the representation of $\GL_n(\K)\times\GL_{n'}(\K)$ is essentially tempered, and in \cite{Chen} when $\K\cong\R$ and the representation of $\GL_n(\K)$ is essentially tempered.

    Michael Harris suggested that the authors also provide a proof of the non-vanishing in the case when $n'=n-1$ and $\K\cong \C$, for the representations that are not necessarily essentially tempered. Harris' suggestion  is one of the  main motivations of the present article. In a forthcoming  paper,  the authors will use the results of this paper to study rationality of special values of Rankin-Selberg L-functions for $\GL(n)\times \GL(n)$, extending the main results of \cite{LLS24}.

	We are going to introduce various notions in order to state our main results in the subsequent subsections. Throughout the article, for $k\in\BN:=\{0,1,2,\ldots\}$, $\GL_k$ is the general linear group over $\Z$. Let $\RB_k=\RT_k\RN_k$ be the Borel subgroup of $\GL_k$ consisting of upper triangular matrices with $\RT_k$ the diagonal torus and $\RN_k$ the unipotent radical. Likewise, let $\overline{\RB}_k=\RT_k\overline{\RN}_k$ be the opposite Borel subgroup of lower triangular matrices with unipotent radical $\overline{\RN}_k$. The centre of $\GL_k$ is denoted by $\RZ_k$. Denote by $1_k$ the identity matrix of size $k$. For each Lie group $G$ we denote by $G^0$ its identity connected component and $\pi_0(G)$ its component group. A superscript `$^\vee$' over a finite-dimensional representation or Casselman-Wallach representation will indicate its contragredient representation.

	\subsection{Generic cohomological representation}
	\label{sec1.1}
	Let $\mathcal{E}_{\K}$ denote the set of all continuous field embeddings $\iota:\K\to\C$. Let $k$ be a positive integer. As usual, we say that a weight
	\[
	\mu:=(\mu^\iota)_{\iota\in \mathcal{E}_{\K}}:=(\mu_1^{\iota},\mu_2^{\iota}, \dots, \mu_k^{\iota})_{\iota\in \mathcal{E}_{\K}}
	\in (\Z^k)^{\mathcal{E}_{\K}}
	\]
	is dominant if 
	\[
	\mu_1^{\iota}\geq\mu_2^{\iota}\geq\dots\geq\mu_k^{\iota}\quad\textrm{for all } \iota\in \mathcal{E}_{\K}. 
	\]
	Let $\mu$ be such a dominant weight and denote by $F_\mu$ the (unique up to isomorphism) irreducible holomorphic finite-dimensional representation of $\GL_k(\K\otimes_\R\C)$ of highest weight $\mu$. We say a Casselman-Wallach representation $\pi_{\mu}$ of $\GL_k(\K)$ is $F_{\mu}$-cohomological if the total continuous cohomology
    \[
\RH_{\mathrm{ct}}^{\ast}\left(\R_+^{\times}\backslash\GL_k(\K)^0;F_{\mu}^{\vee}\otimes\pi_{\mu}\right)\neq\{0\}.
    \]
    Denote by $\Omega(\mu)$ the set of isomorphism classes of generic irreducible $F_{\mu}$-cohomological Casselman-Wallach representations of $\GL_k(\K)$. See Section \ref{sec:TP} for explicit descriptions of these representations.


  As a consequence of Delorme's lemma (\cite[Theorem III.3.3]{BW}), we have that
    \[
	\#(\Omega(\mu))=\begin{cases}
		2,  & \text{if }\K\cong\R\text{ and }k\text{ is odd};\\
		1,  & \text{otherwise},
	\end{cases}
	\]
	and for every $\pi_{\mu}\in\Omega(\mu)$, $i\in\Z$,
	\[
	\RH^i_{\mathrm{ct}}(\R_+^{\times}\backslash\GL_k(\K)^0;F_{\mu}^{\vee}\otimes\pi_{\mu})\neq \{0\}\ \Longleftrightarrow \ b_{k,\K}\leq i\leq t_{k,\K},    
	\]
where 
\[
	(b_{k,\K},t_{k,\K}):=\begin{cases}
		\left(\lfloor\frac{k^2}{4}\rfloor, \lfloor\frac{(k+1)^2}{4}\rfloor-1\right), & \text{if $\K\cong\R$};\\
		\left(\frac{k(k-1)}{2},\frac{k(k+1)}{2}-1\right), & \text{if $\K\cong\C$}.
	\end{cases}
	\]
    Moreover,
	\[
	\begin{aligned}
		& \RH^{b_{k,\K}}_{\mathrm{ct}}(\R_+^{\times}\backslash\GL_k(\K)^0;F_{\mu}^{\vee}\otimes\pi_{\mu})\\
		\cong \, &\RH^{t_{k,\K}}_{\mathrm{ct}}(\R_+^{\times}\backslash\GL_k(\K)^0;F_{\mu}^{\vee}\otimes\pi_{\mu})\\
		\cong \, &  \begin{cases}
			\mathbf{1}_{\K^{\times}}\oplus\mathrm{sgn}_{\K^{\times}},  & \textrm{if}\ \K\cong\R \text{ and $k$ is even};\\
			\varepsilon_{\pi_{\mu}},  & \text{otherwise},
		\end{cases}
	\end{aligned}
	\]
	as representations of $\pi_0(\K^{\times})$, where $\mathbf{1}_{\K^{\times}}$ is the trivial character, $\mathrm{sgn}_{\K^{\times}}$ is the sign character, and $\varepsilon_{\pi_{\mu}}$ is the central character of $F_{\mu}^{\vee}\otimes\pi_{\mu}$.

	We realize $F_{\mu}$, $F_{\mu}^{\vee}$ as algebraic inductions
	\begin{equation}\label{Fmu}
		F_{\mu}={^{\mathrm{alg}}\mathrm{Ind}}^{\GL_k(\K\otimes_{\R}\C)}_{\overline{\RB}_k(\K\otimes_{\R}\C)}\chi_{\mu},\qquad F^{\vee}_{\mu}={^{\mathrm{alg}}\mathrm{Ind}}^{\GL_k(\K\otimes_{\R}\C)}_{\RB_k(\K\otimes_{\R}\C)}\chi_{-\mu}.
	\end{equation}
	Here $\chi_{\mu}$ (resp. $\chi_{-\mu}$) is the algebraic character of $\RT_k(\K\otimes_{\R}\C)$ corresponding to the weight $\mu$ (resp. $-\mu$). We fix a highest weight vector $v_{\mu}\in (F_{\mu})^{\RN_k(\K\otimes_{\R}\C)}$ and a lowest weight vector $v_{\mu}^{\vee}\in (F^{\vee}_{\mu})^{\overline{\RN}_k(\K\otimes_{\R}\C)}$ such that $v_{\mu}(1_k)=v_{\mu}^{\vee}(1_k)=1$. 
    The invariant pairing 
    \[
\langle\cdot,\cdot\rangle:F_{\mu}\times F_{\mu}^{\vee}\to\C
    \]
is normalized such that $\langle v_{\mu},v_{\mu}^{\vee}\rangle=1$. Then viewed as a linear functional on $F_{\mu}$ (resp. $F_{\mu}^{\vee}$), $v_{\mu}^{\vee}$ (resp. $v_{\mu}$) equals the evaluation map at $1_k$.

	Fix a  unitary character $\psi_\R:\R\rightarrow \C^\times$ whose kernel equals $\Z$. Clearly $\psi_{\R}$ equals either $x\mapsto e^{2\pi\mathrm{i}x}$ or $x\mapsto e^{-2\pi\mathrm{i}x}$  ($\mathrm{i}:=\
    \sqrt{-1}$). Put
	\[
	\varepsilon_{\psi_{\R}}=\begin{cases}
		1, & \text{if  }\ \psi_{\R}(x)=e^{2\pi\mathrm{i}x};\\
		-1, &\text{if  }\ \psi_{\R}(x)=e^{-2\pi\mathrm{i}x}.
	\end{cases}
	\]
	Define a character 
	\[
	\psi_{\K}:\K\to\C^{\times},\qquad x\mapsto\psi_{\R}\left(\sum_{\iota\in\mathcal{E}_{\K}}\iota(x)\right),
	\]
	which induces a unitary character
	\[
	\psi_{k,\K}:\RN_k(\K)\to\C^{\times},\qquad [x_{i,j}]_{1\leq i,j\leq k}\mapsto\psi_{\K}\left(\sum_{i=1}^{k-1}x_{i,i+1}\right).
	\]
	
	Take $\pi_{\mu}\in \Omega(\mu)$. Write $0_k$ for the zero element of the abelian group $(\Z^k)^{\mathcal{E}_{\K}}$.
	Let $\pi_{0_k}$ be the unique representation in $\Omega(0_k)$ with the same central character as that of $F_\mu^\vee\otimes \pi_\mu$.
	We fix  generators 
	\[
	\lambda_{\mu}\in \mathrm{Hom}_{\RN_k(\K)}(\pi_{\mu},\psi_{k,\K}), \quad \lambda_{0_{k}} \in\mathrm{Hom}_{\RN_{k}(\K)}(\pi_{0_{k}},{\psi_{k,\K}}),
	\]
	for the one dimensional spaces of Whittaker functionals.
	
	The following proposition is proved in \cite[Proposition 2.1]{LLS24} when $\pi_\mu$ is essentially tempered. The proof in the general case is similar and will be reviewed in Section \ref{sec:TP}.
	
	\begin{prp}
		\label{prop:translationpi} 
		There is a unique map 
		\be \label{jmu00}
		\jmath_{\mu}\in \Hom_{\GL_k(\K)}(\pi_{0_{k}},  F_\mu^\vee\otimes \pi_{\mu})
		\ee 
		such that the diagram
		\be \label{diagjmu}
		\begin{CD}
			\pi_{0_{k}}
			@>\jmath_{\mu} >> F_\mu^\vee\otimes \pi_{\mu} \\
			@V \lambda_{0_{k}}VV            @VV v_\mu\otimes \lambda_{\mu} V\\
			\C@=\C \\
		\end{CD}
		\ee 
		commutes. Moreover, for each $i\in\Z$, $\jmath_{\mu}$ induces a linear isomorphism
		\begin{equation} \label{jmucohomology}
			\jmath_\mu :\RH^i_{\mathrm{ct}}(\R_+^{\times}\backslash\GL_k(\K)^0;\pi_{0_k}) \xrightarrow{\cong} \RH^i_{\mathrm{ct}}(\R_+^{\times}\backslash\GL_k(\K)^0;F_{\mu}^{\vee}\otimes\pi_{\mu})
		\end{equation}
		of $\pi_0(\K^{\times})$-representations.
	\end{prp}

	\subsection{Degenerate principal series representations}\label{sec:1.2}
	
	Denote by $\RP_n(\K)$  the parabolic subgroup of $\GL_n(\K)$ consisting of all the matrices whose last row equals $[0\ \cdots \ 0 \ a]$ for some $a\in\K^{\times}$. For two characters $\eta,\chi:\K^{\times}\to\C^{\times}$, we consider the degenerate principal series representations
	\begin{equation}\label{Ietachi}
		\begin{aligned}
			I_{\eta,\chi}:=&\left(\mathrm{Ind}^{\GL_n(\K)}_{\RP_n(\K)}\left(\mathbf{1}\otimes|\cdot|_{\K}^{\frac{n}{2}}\eta^{-1}\chi^{-n}\right)\right)\otimes\left((|\cdot|_{\K}^{-\frac{1}{2}}\cdot\chi)\circ\det\right)\\
			=&\left({^{\mathrm{u}}\mathrm{Ind}}^{\GL_n(\K)}_{\RP_n(\K)}\left(\mathbf{1}\otimes\eta^{-1}\chi^{-n}\right)\right)\otimes(\chi\circ\det),
		\end{aligned}
	\end{equation}
    where $\mathbf{1}$ is the trivial representation of $\GL_{n-1}(\K)$. When $n=1$, we understand $I_{\eta,\chi}=\eta^{-1}$ as a character. Here and henceforth, $\mathrm{Ind}$ and ${^{\mathrm{u}}\mathrm{Ind}}$ stand for the normalized and unnormalized smooth parabolic inductions respectively. As usual, $|\cdot|_{\K}$ is the normalized absolute value on $\K$.
Let
	\begin{equation}
		\mathrm{d}\eta:=\{\eta_{\iota}\}_{\iota\in\mathcal{E}_{\K}}\in\C^{\mathcal{E}_{\K}},\qquad	\mathrm{d}\chi:=\{\chi_{\iota}\}_{\iota\in\mathcal{E}_{\K}}\in\C^{\mathcal{E}_{\K}}
	\end{equation}
	denote the complexified differentiations of $\eta$ and $\chi$ respectively. We assume that both $\eta$ and $\chi$ are algebraic in the sense that $\mathrm{d}\eta,\mathrm{d}\chi\in\Z^{\mathcal{E}_{\K}}$. When $\K\cong\R$, we impose the assumption that 
    \be \label{nonzerohc}
   \eta(-1)\chi(-1)^n=\begin{cases}
       (-1)^{\eta_{\iota}+n\chi_{\iota}}, & \textrm{Case ($-$)},\\
       (-1)^{\eta_{\iota}+n\chi_{\iota}-n}, & \textrm{Case ($+$)}.
   \end{cases}
    \ee
    
It is easy to see that the infinitesimal character of $I_{\eta,\chi}$ is regular if and only if
	\begin{equation}\label{regular}
		(\eta_{\iota}+n\chi_{\iota})(\eta_{\iota}+n \chi_{\iota}-n)\geq 0 
	\end{equation}
    for all $\iota\in\mathcal{E}_{\K}$.
	Assume this is the case and write $F_{\eta,\chi}$ for the irreducible holomorphic finite-dimensional representation of $\GL_n(\K\otimes_{\R}\C)$ whose infinitesimal character equals that of $I_{\eta,\chi}$.
	
Let $\iota\in\mathcal{E}_{\K}$ be an element such that
	\begin{equation}\label{iota}
		\eta_{\iota}+n\chi_{\iota}\leq \eta_{\overline{\iota}}+n\chi_{\overline{\iota}},
	\end{equation}
	where $\overline{\iota}$ denotes the composition of $\iota$ with the complex conjugation. 
    Note that $\iota = \bar\iota$ when $\K\cong\R$. When $n'=n$, we distinguish three cases according to the pair $(\mathrm{d}\eta,\mathrm{d}\chi)$:   
	\begin{eqnarray}
	\nonumber   & \textrm{Case ($-$)} \quad & \eta_{\bar\iota}+n\chi_{\bar\iota}\leq 0;\\
\label{cases}	&\textrm{Case ($+$)}\quad  & \eta_{\iota}+n\chi_{\iota} \geq n;\\
	\nonumber &	\textrm{Case ($\pm$)} \quad &\eta_{\iota}+n\chi_{\iota}\leq 0 \text{ and }  \eta_{\overline{\iota}}+n\chi_{\overline{\iota}}\geq n.
	\end{eqnarray}
    In particular $\K\cong\C$ in Case ($\pm$).
	Set
	\[
	c_{n,\K}:=\begin{cases}
		0, & \text{Case ($-$)};\\
		[\K:\R](n-1), & \text{Case ($+$)};\\
		n-1, & \text{Case ($\pm$)}.
	\end{cases}
	\]
	As a consequence of Delorme's lemma (\cite[Theorem III.3.3]{BW}), we have that
	$\RH_{\mathrm{ct}}^i(\R_+^{\times}\backslash\GL_n(\K)^0;F_{\eta,\chi}^{\vee}\otimes I_{\eta,\chi})=\{0\}$ if $i<c_{n,\K}$ and
    \[
	\mathrm{dim}\,\RH_{\mathrm{ct}}^{c_{n,\K}}(\R_+^{\times}\backslash\GL_n(\K)^0;F_{\eta,\chi}^{\vee}\otimes I_{\eta,\chi})=1.
	\] 
We remark that the above cohomology group vanishes without the assumption \eqref{nonzerohc} for $n\geq 2$.


We introduce some notations for later use. For a ring $R$ and $k, l\in \BN$, write $R^{k\times l}$ for the set of $k\times l$ matrices with entries in $R$. Write $\C_{\iota'}:=\C$ for each $\iota'\in\mathcal{E}_{\K}$,  viewed as a $\K$-algebra via $\iota'$. When no confusion arises, we view an element $g\in \K^{k\times l}$ as its diagonal image in $(\K\otimes_\BR\BC)^{k\times l}=\prod_{\iota'\in\CE_\K}\BC_{\iota'}^{k\times l}$, and  write $g^{\iota'}:=\iota'(g)$ for its image in $\BC_{\iota'}^{k\times l}$. Similar notation will be used without further explanation.

Realize $F_{\eta,\chi}$ and $F_{\eta,\chi}^\vee$ by algebraic inductions as in \eqref{Fmu}, and denote by $v_{\eta,\chi}\in(F_{\eta,\chi})^{\RN_n(\K\otimes_{\R}\C)}$  the evaluation map at $1_n$ (when viewed as a linear functional on $F_{\eta,\chi}^\vee$).  Define a continuous linear map
\[
\begin{aligned}
\ell_{\eta,\chi}:F_{\eta,\chi}^{\vee}\otimes I_{\eta,\chi}& \to\C, \\
v\otimes \varphi & \mapsto \begin{cases} \langle v,v_{\eta,\chi}\rangle\cdot \varphi(w_n),
        & \text{Case ($-$)}; \\
		\langle v,v_{\eta,\chi}\rangle\cdot \varphi(1_n), & \text{Case ($+$)};\\
		 \langle v,(w_n^\iota, 1_n^{\bar\iota}).v_{\eta,\chi}\rangle\cdot \varphi(1_n), &\text{Case ($\pm$)},
         \end{cases}
    \end{aligned}
    \]
where $w_n$ is the anti-diagonal permutation matrix in $\GL_n(\Z)$.
    Put
    \begin{equation}
    \begin{aligned}
    \eta_0: &=\eta\cdot\prod_{\iota'\in\mathcal{E}_{\K}}(\iota'|_{\K^{\times}})^{-\eta_{\iota'}}, \\
  \chi_0:& =
  \begin{cases}
  \chi\cdot\prod_{\iota'\in\mathcal{E}_{\K}}(\iota'|_{\K^{\times}})^{-\chi_{\iota'}}, &\textrm{Case ($-$)};\\
  \chi\cdot\prod_{\iota'\in\mathcal{E}_{\K}}(\iota'|_{\K^{\times}})^{1-\chi_{\iota}'}, &\textrm{Case ($+$)};\\
  \chi\cdot\overline{\iota}|_{\K^{\times}}\cdot\prod_{\iota'\in\mathcal{E}_{\K}}(\iota'|_{\K^{\times}})^{-\chi_{\iota'}}, &\textrm{Case ($\pm$)}.
   \end{cases}
   \end{aligned}
 \end{equation}   
    Then the representation $I_{\eta_0, \chi_0}$  also has regular infinitesimal character, and  equalities hold  in  \eqref{cases} when $(\eta,\chi)$ is replaced by $(\eta_0,\chi_0)$. We apply the above discussions  and notations for $(\eta,\chi)$ to $(\eta_0,\chi_0)$. 
	As an analogue of Proposition \ref{prop:translationpi}, we will prove the following proposition in Section \ref{sec:Translation}.
	
	\begin{prp}\label{prop:translation}
		There is a unique map
		\begin{equation} \label{jetachi}
			\jmath_{\eta,\chi}\in\mathrm{Hom}_{\GL_n(\K)}\left(F_{\eta_0,\chi_0}^{\vee}\otimes I_{\eta_0,\chi_0},F_{\eta,\chi}^{\vee}\otimes I_{\eta,\chi}\right)
		\end{equation}
		such that the diagram
		\begin{equation}\label{jetachidiagram}
			\begin{CD}
				F_{\eta_0,\chi_0}^{\vee}\otimes I_{\eta_0,\chi_0} @>\jmath_{\eta,\chi}>> F_{\eta,\chi}^{\vee}\otimes I_{\eta,\chi}\\
				@V\ell_{\eta_0,\chi_0}VV @VV\ell_{\eta,\chi}V\\
				\C @= \C
			\end{CD}
		\end{equation}
		commutes. Moreover, $\jmath_{\eta,\chi}$ induces a linear isomorphism 
		\begin{equation}\label{jetachicohomology}
			\jmath_{\eta,\chi}:\RH_{\mathrm{ct}}^{c_{n,\K}}(\R_+^{\times}\backslash\GL_n(\K)^0;F_{\eta_0,\chi_0}^{\vee}\otimes I_{\eta_0,\chi_0})\xrightarrow{\cong}\RH_{\mathrm{ct}}^{c_{n,\K}}(\R_+^{\times}\backslash\GL_n(\K)^0;F_{\eta,\chi}^{\vee}\otimes I_{\eta,\chi}).
		\end{equation}
	\end{prp}

\subsection{The balanced coefficient systems}

Recall that $n'=n$ or $n'=n-1$. Suppose that $\mu\in(\Z^n)^{\mathcal{E}_{\K}}$, $\nu\in(\Z^{n'})^{\mathcal{E}_{\K}}$ are dominant weights and $F_{\mu}$, $F_{\nu}$ are associated irreducible holomorphic finite-dimensional representations. Write $F_{\chi}:=\C$ for the holomorphic one-dimensional representation of $\GL_{n'}(\K\otimes_{\R}\C)$ whose differential equals the complexified differential of the character $\chi\circ\det:\GL_{n'}(\K)\to\C^{\times}$. Denote $\xi:=(\mu,\nu)$, and in the case of $n'=n$ we assume that the complexified differential of $\eta$ equals the differential of the product of the central characters of $F_\mu$ and $F_\nu$.
	
	\begin{dfnp}\label{def:balance}
		An algebraic character $\chi$  of $\K^\times$ is said to be $\xi$-balanced if
		\[
		\begin{cases}
			\mathrm{Hom}_{\GL_{n'}(\K\otimes_{\R}\C)}\left(F_{\mu}^{\vee}\otimes F_{\nu}^{\vee}\otimes F_{\chi}^{\vee},\C\right)\neq \{0\},& \text{when }n'=n-1;\\
			\mathrm{Hom}_{\GL_{n'}(\K\otimes_{\R}\C)}\left(F_{\mu}^{\vee}\otimes F_{\nu}^{\vee}\otimes F_{\eta,\chi}^{\vee},\C\right)\neq \{0\},&\text{when }n'=n.
		\end{cases}
		\]
		Denote the set of all $\xi$-balanced characters by $\RB(\xi)$.
	\end{dfnp}
	
	We assume $\chi\in\RB(\xi)$, which is the necessary archimedean condition for studying L-functions via modular symbols. The conditions for $\chi$ to be $\xi$-balanced are given by Lemma \ref{lem-balanced}. We remark that no critical condition is required in this article (but see Lemma \ref{lem-bc}).
	
	Clearly by \eqref{Fmu}, when $n'=n-1$, $F_{\mu}\otimes F_{\nu}\otimes F_{\chi}$ is  realized as a space of algebraic functions on 
    \[
    \CX_{n,n',\K}:=\GL_n(\K\otimes_{\R}\C)\times\GL_{n-1}(\K\otimes_{\R}\C).
    \]
    When $n'=n$,  we realize $F_{\eta,\chi}$ as a space of polynomial coefficient differential operators on $(\K\otimes_{\R}\C)^{1\times n}=\prod_{\iota'\in\CE_\K}\C_{\iota'}^{1\times n}$ in Section \ref{sec:Translation}. Then using \eqref{Fmu} and the Fourier transform \eqref{partial-F}, $F_{\mu}\otimes F_{\nu}\otimes F_{\eta,\chi}$ is also  realized as a space of algebraic functions on the complex variety $\CX_{n,n,\K}$ given by 
	\[
    \begin{cases}
		  \GL_n(\K\otimes_\R\C)\times\GL_n(\K\otimes_\R\C)\times(\K\otimes_\R\C)^{1\times n},  & \text{Case ($-$)}; \\
		\GL_n(\K\otimes_\R\C)\times\GL_n(\K\otimes_\R\C)\times(\K\otimes_\R\C)^{n\times 1}, & \text{Case ($+$)}; \\
		\GL_n(\C_{\iota})\times \GL_n(\C_{\iota})\times \BC_{\iota}^{1\times n}\times \GL_n(\C_{\bar\iota})\times \GL_n(\C_{\bar\iota})\times \BC_{\bar\iota}^{n\times 1}, &
        \text{Case ($\pm$)}.
        \end{cases}
	\]

	Following \cite[Section 1.3]{LLSS}, define a family  $\{z_k\in \GL_k(\BZ)\}_{k\in \BN}$ of matrices
	inductively by
	\[
	z_0:=\emptyset\ \  (\textrm{the unique element of $\GL_0(\BZ)$}), \quad  z_1:=[1]\quad \text{and}
	\]
	\begin{equation}\label{zn}
		z_k :=
		\begin{bmatrix}
			w_{k-1}& 0 \\
			0 & 1 
		\end{bmatrix}
		\begin{bmatrix}
			{^tz}_{k-2}^{-1}& 0 \\
			0 & 1_2 
		\end{bmatrix}
		\begin{bmatrix}
			{}^tz_{k-1}  w_{k-1} z_{k-1}& {}^t e_{k-1}\\
			0 & 1 \\
		\end{bmatrix}, \quad k\geq 2,
	\end{equation}
	where	$e_{k-1}:=[0 \ \cdots \  0 \ 1]\in \BZ^{1\times(k-1)}$ is a row vector.  For $g\in \GL_n(\BC)$, write $g^\tau:={}^tg^{-1}$ for the transpose inverse of $g$. Put 
	\be \label{elet:z}
	z:=\left(z_n, \begin{bmatrix}  z_{n-1} \\ & 1\end{bmatrix}, e_n\right),\quad 
	\check z:=\left(w_n z_n^\tau, w_n \begin{bmatrix} z_{n-1}^\tau \\ & 1\end{bmatrix}, {}^te_n\right),
	\ee
viewed as elements of $\GL_n(\Z)\times\GL_n(\Z)\times\Z^{1\times n}$ and $\GL_n(\Z)\times\GL_n(\Z)\times\Z^{n\times 1}$ respectively. 

Define an element $\tilde z\in \CX_{n, n', \K}$ by
	\[
	\tilde z:=\begin{cases}
    (z_n, z_{n-1}), &  \text{Case $n' =n-1$}; \\
		z, & \text{Case ($-$)};\\
		\check{z}, & \text{Case ($+$)};\\
		(z^\iota, \check{z}^{\bar\iota}), & \text{Case ($\pm$)}.
	\end{cases}
	\]
We have the following proposition which is proved in \cite[Proposition 3.1]{LLS24} when $n'=n-1$. When $n'=n$, it follows from the fact that the diagonal right action of $\GL_n(\C)$ on $\mathcal{B}_n(\C)\times\mathcal{B}_n(\C)\times\C^{1\times n}$ (resp. $\mathcal{B}_n(\C)\times\mathcal{B}_n(\C)\times\C^{n\times 1}$) has a unique open orbit represented by $z$ (resp. $\check{z}$), where $\mathcal{B}_n(\C):=\overline{\RB}_n(\C)\backslash\GL_n(\C)$.

	\begin{prp}\label{prop:phixichi}
		There is a unique element 
		\begin{equation}\label{phixichi}
				\phi_{\xi,\chi}\in \begin{cases}
					(F_{\mu}\otimes F_{\nu}\otimes F_{\chi})^{\GL_{n'}(\K\otimes_{\R}\C)},  & \text{when }n'=n-1;\\
					(F_{\mu}\otimes F_{\nu}\otimes F_{\eta,\chi})^{\GL_{n'}(\K\otimes_{\R}\C)}, & \text{when }n'=n,
				\end{cases}
		\end{equation}
		such that $\phi_{\xi,\chi}(\tilde z)=1$.
	\end{prp}

	\subsection{Rankin-Selberg integrals}\label{sec:1.3}
	
	 Take $\pi_{\mu}\in\Omega(\mu)$ and $\pi_{\nu}\in\Omega(\nu)$ and fix generators
	\[
	\lambda_{\mu}\in\mathrm{Hom}_{\RN_n(\K)}(\pi_{\mu},\psi_{n,\K}),\qquad\lambda'_{\nu}\in\mathrm{Hom}_{\RN_{n'}(\K)}(\pi_{\nu},\overline{\psi_{n',\K}}).
	\]
	Set $\xi:=(\mu,\nu)$ and write $\pi_{\xi}:=\pi_{\mu}\widehat{\otimes}\pi_{\nu}$ for the completed projective tensor product. Write $\widehat{\K^{\times}}:=\mathrm{Hom}(\K^{\times},\C^{\times})$ for the set of characters of $\K^{\times}$, which is naturally a commutative complex Lie group. For every $\chi'\in\widehat{\K^{\times}}$, write $\mathrm{Re}(\chi')$ for the real number such that 
    \[
    \textrm{$|\chi'(a)|=|a|_{\K}^{\mathrm{Re}(\chi')}\quad $ for all $a\in\K^{\times}$.}
	\]

	Fix the Haar measure  $\mathrm{d}x$ on $\K$ to be the self-dual measure with respect to $\psi_{\K}$. The Haar measure on $\RN_{n'}(\K)$ is fixed to be the product of the Haar measures on $\K$,  and the measure on $\RZ_n(\K)=\K^{\times}$ is fixed to be the standard multiplicative measure $\frac{\mathrm{d}x}{|x|_{\K}}$. Denote by $\mathfrak{M}_{n',\K}$ the one-dimensional complex vector space of left (or right) invariant measures on $\GL_{n'}(\K)$. Any invariant measure $\mathrm{d}g$ on $\GL_{n'}(\K)$ induces quotient measures on $\RN_{n'}(\K)\backslash\GL_{n'}(\K)$ and $\RZ_{n'}(\K)\RN_{n'}(\K)\backslash\GL_{n'}(\K)$.  We will denote both of these quotient measures   by $\overline{\mathrm{d}}g$.
	
	\subsubsection{ }
	Suppose that $n'=n-1$. We view $\GL_{n-1}$ as an algebraic subgroup of $\GL_n$ via the embedding 
	\[
	g\mapsto\begin{bmatrix}
		g & 0\\
		0 & 1
	\end{bmatrix}.
	\]
	We have the Rankin-Selberg integral
\[
\RZ_{\xi,\chi'}:\pi_{\xi}\otimes(\chi'\circ\det)\to\mathfrak{M}_{n',\K}^{\ast}\quad (\phantom{\,}^\ast\textrm{ indicates the dual space})
\]
defined by 
	\begin{equation}
		\RZ_{\xi,\chi'}(f,f';\mathrm{d}g):=\int_{\RN_{n'}(\K)\backslash\GL_{n'}(\K)}\left\langle\lambda_{\mu},g.f\right\rangle\left\langle\lambda_{\nu}',g.f'\right\rangle\chi'(\det g)\overline{\mathrm{d}}g,
	\end{equation}
	where $f\in\pi_{\mu}$, $f'\in\pi_{\nu}$, $\chi'\in\widehat{\K^{\times}}$, and $\mathrm{d}g\in\mathfrak{M}_{n',\K}$. It converges absolutely when $\mathrm{Re}(\chi')$ is sufficiently large and has a meromorphic continuation to the complex Lie group $\widehat{\K^{\times}}$. 
    By \cite{J09}, the normalized Rankin-Selberg integral
	\[
	\RZ_{\xi,\chi'}^{\circ}(f,f';\mathrm{d}g):=\frac{1}{\RL(\frac{1}{2},\pi_{\mu}\times\pi_{\nu}\times\chi')}\RZ_{\xi,\chi'}(f,f';\mathrm{d}g)
	\]
    is holomorphic and yields a nonzero linear functional
	\[
	\RZ_{\xi,\chi'}^{\circ}\in\mathrm{Hom}_{\GL_{n'}(\K)}\left(\pi_{\xi}\otimes(\chi'\circ\det),\mathfrak{M}_{n',\K}^{\ast}\right).
	\]
	
	\subsubsection{ }
	
	Suppose that  $n'=n$ and  $\eta$ is the product of the central characters of $\pi_{\mu}$ and $\pi_{\nu}$ (note that when $\K\cong\R$ and $n$ is even, the equality \eqref{nonzerohc}  automatically holds when $\eta$ is such a product). 
    
    We have the Rankin-Selberg integral
    \[
\RZ_{\xi,\chi'}:\pi_{\xi}\,\widehat{\otimes}\, I_{\eta,\chi'}\to\mathfrak{M}_{n,\K}^{\ast}
    \]
    defined by
	\begin{equation}\label{rsintegral}
		\RZ_{\xi,\chi'}(f,f',\varphi;\mathrm{d}g):=\int_{\RZ_n(\K)\RN_n(\K)\backslash\GL_{n}(\K)}\left\langle \lambda_{\mu},g.f\right\rangle\left\langle \lambda'_{\nu},g.f'\right\rangle\varphi(g)\overline{\mathrm{d}}g,
	\end{equation}
	where $f\in\pi_{\mu}$, $f'\in\pi_{\nu}$, $\varphi\in I_{\eta,\chi'}$ with $\chi'\in\widehat{\K^{\times}}$, and $\mathrm{d}g\in\mathfrak{M}_{n,\K}$. It converges absolutely when $\mathrm{Re}(\chi')$ is sufficiently large and has a meromorphic continuation to the complex Lie group $\widehat{\K^{\times}}$. 

When $\chi\in\RB(\xi)$ is balanced, by Proposition \ref{RSnn} the normalized Rankin-Selberg integral
	\begin{equation} \label{normalizedRS}
	\RZ_{\xi,\chi'}^{\circ}(f,f',\varphi;\mathrm{d}g):=\frac{\RL(0,\eta\chi'^n)}{\RL(0,\pi_{\mu}\times\pi_{\nu}\times\chi')}\cdot\RZ_{\xi,\chi'}(f,f',\varphi;\mathrm{d}g)
	\end{equation}
    is holomorphic at $\chi'=\chi$ and yields a nonzero linear functional 
    \[
	\RZ_{\xi,\chi}^{\circ}\in\mathrm{Hom}_{\GL_{n}(\K)}\left(\pi_{\xi}\,\widehat{\otimes}\,I_{\eta,\chi},\mathfrak{M}_{n,\K}^{\ast}\right).
	\]

	\subsection{Archimedean modular symbols}
	
	Let
	\[
	\RK_{n',\K}:=\begin{cases}
		\RO(n'), &\K\cong\R;\\
		\RU(n'), &\K\cong\C,
	\end{cases}
	\]
	be a maximal compact subgroup of $\GL_{n'}(\K)$. Denote by $\mathfrak{gl}_{n',\K}$ and $\mathfrak{k}_{n',\K}$ the Lie algebras of $\GL_{n'}(\K)$ and $\RK_{n',\K}$ respectively. Let
	\[
	d_{n,n',\K}: =\begin{cases}
		\mathrm{dim}_{\R}(\mathfrak{gl}_{n-1,\K}/\mathfrak{k}_{n',\K}), & n'=n-1;\\
		\mathrm{dim}_{\R}(\mathfrak{gl}_{n,\K}/(\mathfrak{k}_{n,\K}\oplus\R)), & n'=n.
	\end{cases}
	\]
Then we have the following numerical coincidence
\[
	\begin{cases}
		b_{n,\K}+b_{n-1,\K}=d_{n,n-1,\K}, & \text{Case $n'=n-1$};\\
		b_{n,\K}+t_{n,\K}+c_{n,\K}=d_{n,n,\K}, & \text{Case $(-)$};\\
		b_{n,\K}+b_{n,\K}+c_{n,\K}=d_{n,n,\K}, & \text{Case $(\pm)$}.
	\end{cases}
    \]
We define a vector space $\RH_{\xi,\chi}$ as follows, which is a representation of $\pi_0(\K^{\times})^3$.
	\begin{itemize}
		 \item  Case $n'=n-1$: 
		\[
		\begin{aligned}
			\RH_{\xi,\chi}:=&\,\RH_{\mathrm{ct}}^{b_{n,\K}}(\R_+^{\times}\backslash\GL_n(\K)^0;F_{\mu}^{\vee}\otimes\pi_{\mu})\otimes\RH_{\mathrm{ct}}^{b_{n-1,\K}}(\R_+^{\times}\backslash\GL_{n-1}(\K)^0;F_{\nu}^{\vee}\otimes\pi_{\nu})\\
			& \otimes\RH_{\mathrm{ct}}^0(\R_+^{\times}\backslash\GL_{n-1}(\K)^0;F_{\chi}^{\vee}\otimes(\chi\circ\det)).
		\end{aligned}
		\] 
		\item Case $(-)$:
		\[
		\begin{aligned}
			\RH_{\xi,\chi}:=&\,\RH_{\mathrm{ct}}^{b_{n,\K}}(\R_+^{\times}\backslash\GL_n(\K)^0;F_{\mu}^{\vee}\otimes\pi_{\mu})\otimes\RH_{\mathrm{ct}}^{t_{n,\K}}(\R_+^{\times}\backslash\GL_n(\K)^0;F_{\nu}^{\vee}\otimes\pi_{\nu})\\
			& \otimes\RH_{\mathrm{ct}}^{c_{n,\K}}(\R_+^{\times}\backslash\GL_{n}(\K)^0;F_{\eta,\chi}^{\vee}\otimes I_{\eta,\chi}).
		\end{aligned}
		\] 
		\item Case $(\pm)$:
		\[
		\begin{aligned}
			\RH_{\xi,\chi}:=&\,\RH_{\mathrm{ct}}^{b_{n,\K}}(\R_+^{\times}\backslash\GL_n(\K)^0;F_{\mu}^{\vee}\otimes\pi_{\mu})\otimes\RH_{\mathrm{ct}}^{b_{n,\K}}(\R_+^{\times}\backslash\GL_n(\K)^0;F_{\nu}^{\vee}\otimes\pi_{\nu})\\
			& \otimes\RH_{\mathrm{ct}}^{c_{n,\K}}(\R_+^{\times}\backslash\GL_{n}(\K)^0;F_{\eta,\chi}^{\vee}\otimes I_{\eta,\chi}).
		\end{aligned}
		\] 
	\end{itemize}
	
	We have identifications
	\begin{equation}\label{idmeasure}
		\begin{aligned}
			\mathfrak{M}_{n',\K}&=\{\text{invariant measures on }\R_+^{\times}\backslash\GL_{n'}(\K)/\RK_{n',\K}^0\}\\
			&=\{\text{invariant measures on }\GL_{n'}(\K)/\RK_{n',\K}^0\}
		\end{aligned}
	\end{equation}
	by fixing the Haar measure on $\RK_{n',\K}^0$ with total volume $1$, and the standard multiplicative measure $\frac{\mathrm{d}x}{x}$ on $\R_+^{\times}$. Denote by $\mathfrak{O}_{n',\K}$ the one-dimensional space of $\GL_{n'}(\K)$-invariant sections of the orientation line bundle of $\R_+^{\times}\backslash\GL_{n'}(\K)/\RK_{n',\K}^0$ or $\GL_{n'}(\K)/\RK_{n',\K}^0$ with complex coefficients, respectively when $n'=n$ or $n-1$. See \cite[Section 3.1]{LLS24} for more details.
	
	In what follows we define the archimedean modular symbol, which is a linear functional
	\[
	\wp_{\xi,\chi}:\RH_{\xi,\chi}\otimes\mathfrak{O}_{n',\K}\to\C.
	\]
    When $n'=n-1$, $\wp_{\xi,\chi}$ is defined in \cite[Section 3.2]{LLS24} as the composition of 
  \begin{eqnarray*}
    &  &   \oH_{\xi,\chi}\otimes \frak O_{n',\K} \\
& \xrightarrow{{\rm res}} & \oH^{d_{n,n',\K}}_{\rm ct}(\GL_{n'}(\K)^0; F_\mu^\vee\otimes F_\nu^\vee\otimes F_\chi^\vee\otimes 
\pi_\mu\,\widehat{\otimes}\, \pi_\nu\otimes (\chi\circ\det))\otimes \frak O_{n',\K}\\
	 & \xrightarrow{\phi_{\xi,\chi}\otimes \oZ^\circ_{\xi,\chi}}& \oH^{d_{n,n',\K}}_{\rm ct}(\GL_{n'}(\K)^0; \frak M_{n',\K}^*)\otimes \frak O_{n',\K}=\BC.
	\end{eqnarray*}
    	When $n'=n$, $\wp_{\xi,\chi}$ is defined as the composition of 
        \begin{eqnarray*}
         &&\oH_{\xi,\chi}\otimes \frak O_{n,\K}  \\
		&\xrightarrow{\textrm{res}} &\oH^{d_{n,n,\K}}_{\rm ct}(\BR^\times_+\bs \GL_n(\K)^0; F_\mu^\vee\otimes F_\nu^\vee\otimes 
        F_{\eta,\chi}^\vee\otimes \pi_\mu \,\widehat{\otimes}\, \pi_\nu\,\widehat{\otimes}\, I_{\eta,\chi})\otimes \frak{O}_{n,\K}  \\
		 & \xrightarrow{\phi_{\xi,\chi}\otimes\oZ^\circ_{\xi,\chi}} &\oH^{d_{n,n,\K}}_{\rm ct}(\BR^\times_+\bs \GL_n(\K)^0; \frak M_{n,\K}^*)\otimes \frak{O}_{n,\K}=\BC.
	\end{eqnarray*}
	In both cases, the second map is induced by $\phi_{\xi,\chi}$ defined in \eqref{phixichi} and the Rankin-Selberg integral $\RZ_{\xi,\chi}^{\circ}$. Here and henceforth, `$\mathrm{res}$' indicates the restriction map of   the cohomologies through the diagonal embedding. 
    See \cite[Section 3.1]{LLS24} for a detailed explanation of the last equality, which uses the identification \eqref{idmeasure}. 

    \begin{remarkp}
Assume $n\geq 2$.  For \emph{Case $(+)$}, we have $b_{n,\K}+b_{n,\K}+c_{n,\K}>d_{n,n,\K}$. Hence the archimedean modular symbol defined in a similar fashion as above will be identically zero. However we note that if $(\pi_{\mu},\pi_{\nu},\chi)$ is of \emph{Case $(+)$}, then $(\pi_{\mu}^{\vee},\pi_{\nu}^{\vee},\chi^{-1}|\cdot|_{\K})$ is of \emph{Case $(-)$}. Therefore, the study of special values of L-functions in    \emph{Case $(+)$} can be recovered from \emph{Case $(-)$} through the functional equation. Thus we will not consider the archimedean period relations in  \emph{Case  $(+)$}. 
\end{remarkp}

	\subsection{Archimedean period relations and non-vanishing}
	\label{sec:1.6}
	Set $\xi_0:=(0_n,0_{n'})$. Recall from Proposition \ref{prop:translationpi} the map 
	\[
	\jmath_\mu\in \Hom_{\GL_{n}(\K)}(\pi_{0_{n}}, F_\mu^\vee\otimes \pi_\mu)
	\]
	which is specified by the generators 
	\[
	v_\mu\in F_\mu^{\RN_{n}(\K\otimes_\R \C)}, \quad \lambda_\mu\in\mathrm{Hom}_{\RN_{n}(\K)}(\pi_{\mu},{\psi_{n,\K}}) \quad \textrm{and}\quad \lambda_{0_{n}} \in\mathrm{Hom}_{\RN_{n}(\K)}(\pi_{0_{n}},{\psi_{n,\K}}).
	\]
	Here  $\pi_{0_n}$ is the unique representation in $\Omega(0_n)$ with the same central character as that of $F_\mu^\vee\otimes \pi_\mu$. Likewise let $\pi_{0_{n'}}$ be the unique representation in $\Omega(0_{n'})$ with the same central character as that of $F_\nu^\vee\otimes \pi_\nu$. With fixed generators
	\[
	\lambda'_\nu\in\mathrm{Hom}_{\RN_{n'}(\K)}(\pi_{\nu},\overline{\psi_{n',\K}}), \quad \lambda'_{0_{n'}} \in\mathrm{Hom}_{\RN_{n'}(\K)}(\pi_{0_{n'}},\overline{\psi_{n',\K}}),
	\]
	and  
	\[
	v_\nu\in F_\nu^{\RN_{n'}(\K\otimes_\R \C)}\quad(\textrm{the function that has constant value $1$ on $\RN_{n'}(\K\otimes_\R \C)$}),
	\]
	as in Proposition \ref{prop:translationpi} we get a map 
	\be\label{jmathnu}
	\jmath'_\nu\in \Hom_{\GL_{n'}(\K)}(\pi_{0_{n'}}, F_\nu^\vee\otimes \pi_\nu).
	\ee
	When $n'=n$, note that $\pi_{\xi_0}$ has central character $\eta_0$ and we have the homomorphism 
	\[
	\jmath_{\eta,\chi}\in\mathrm{Hom}_{\GL_n(\K)}\left(F_{\eta_0,\chi_0}^{\vee}\otimes I_{\eta_0,\chi_0},F_{\eta,\chi}^{\vee}\otimes I_{\eta,\chi}\right)
	\]
	from Proposition \ref{prop:translation}. When $n'=n-1$, we put
    \[
    \chi_0:=\chi\cdot\prod_{\iota'\in\CE_{\K}}(\iota'|_{\K^{\times}})^{-\chi_{\iota'}}.
    \]
    Then the equality \[
    F_{\chi}^{\vee}\otimes(\chi\circ\det)=\chi_0\circ\det
    \]
    yields an isomorphism
	\[
	\jmath_\chi: \oH^0_{\rm ct}(\BR^\times_+\bs \GL_{n'}(\K)^0; \chi_0\circ\det) \xrightarrow{\cong} 
	\oH^0_{\rm ct}(\BR^\times_+\bs \GL_{n'}(\K)^0; F_\chi^\vee\otimes \chi\circ\det). 
	\]
	
	In summary, we have the isomorphism 
	\[
	\jmath_{\xi,\chi}:\RH_{\xi_0,\chi_0}\xrightarrow{\cong}\RH_{\xi,\chi}
	\]
	given by
	\begin{equation}\label{jxichi}
		\jmath_{\xi,\chi}:=\begin{cases}
			\jmath_\mu\otimes \jmath'_\nu \otimes \jmath_\chi, & n'=n-1;\\
			\jmath_{\mu}\otimes\jmath'_{\nu}\otimes\jmath_{\eta,\chi}, &n'=n.
		\end{cases}
	\end{equation}

	To state our main theorem, we introduce several constants depending on $\xi$ and $\chi$. Let
	\begin{equation}
		\begin{aligned}
			c_{\xi,\chi}&:=\prod_{\substack{i+k\leq n\\ \iota'\in\CE_{\K}}}(\varepsilon_{\psi_\R}\cdot\mathrm{i})^{\mu_i^{\iota'}+\nu_k^{\iota'}+\chi_{\iota'}},\\
          	c'_{\xi,\chi}&:=\prod_{i+k\leq n}(\varepsilon_{\psi_\R}\cdot\mathrm{i})^{\mu_i^{\iota}+\mu_i^{\overline{\iota}}+\nu_k^{\iota}+\nu_k^{\overline{\iota}}+\chi_{\iota}+\chi_{\overline{\iota}}-1},\\
			\varepsilon_{\xi,\chi}&:=\prod_{\substack{i>k,\,i+k\leq n\\\iota'\in\CE_{\K}}}(-1)^{\mu_i^{\iota'}+\nu_k^{\iota'}+\chi_{\iota'}},\\
\varepsilon'_{\xi,\chi}&:=(-1)^{\frac{n(n-1)}{2}(\chi_{\iota}+\chi_{\overline{\iota}}-1)}\cdot\prod_{i=1}^n(-1)^{(n+1)\nu_i^{\overline{\iota}}}\\
&		\qquad\qquad\qquad\cdot\prod_{i>k,\,i+k\leq n}(-1)^{\mu_i^{\iota}+\nu_k^{\iota}+\mu_{n+1-i}^{\overline{\iota}}+\nu_{n+1-k}^{\overline{\iota}}+\chi_{\iota}+\chi_{\overline{\iota}}-1},\\
\varepsilon''_{\xi,\chi}&:=(-1)^{\frac{n(n-1)}{2}\sum_{\iota'\in\CE_{\K}}\chi_{\iota'}}\cdot\prod_{\substack{i>k,\,i+k\leq n\\\iota'\in\CE_{\K}}}(-1)^{\mu_i^{\iota'}+\nu_k^{\iota'}+\chi_{\iota'}}.
		\end{aligned}
	\end{equation}
In Case $(-)$, define  $g_{\xi,\chi}(s)$ as follows: 
    \begin{itemize}
   \item When $\K\cong\R$ and $n$ is even,
   \[
   g_{\xi,\chi}(s):=\frac{\RL(s,\eta_0\chi_0^n)}{\RL(s,\eta\chi^n)}\cdot\prod_{i=1}^{\frac{n}{2}}\frac{\Gamma_\BC\left(s+\max\{\mu_i^{\iota}+\nu_{n+1-i}^{\iota}+\chi_{\iota},\mu_{n+1-i}^{\iota}+\nu_{i}^{\iota}+\chi_{\iota}\}\right)}{\Gamma_{\C}(s)}.
   \]
   \item When $\K\cong\R$ and $n$ is odd,
   \[
   \begin{aligned}
    g_{\xi,\chi}(s):=&\frac{\RL(s,\eta_0\chi_0^n)}{\RL(s,\eta\chi^n)}\cdot\frac{\Gamma_{\R}\left(s+\mu_{\frac{n+1}{2}}^{\iota}+\nu_{\frac{n+1}{2}}^{\iota}+\chi_{\iota}+\delta((\iota|_{\K^{\times}})^{\mu_{\frac{n+1}{2}}^{\iota}+\nu_{\frac{n+1}{2}}^{\iota}}\cdot\eta_0\chi)\right)}{\Gamma_{\R}(s+\delta(\eta_0\chi_0))}\\
  &\cdot \prod_{i=1}^{\frac{n-1}{2}}\frac{\Gamma_\BC\left(s+\max\{\mu_i^{\iota}+\nu_{n+1-i}^{\iota}+\chi_{\iota},\mu_{n+1-i}^{\iota}+\nu_{i}^{\iota}+\chi_{\iota}\}\right)}{\Gamma_{\C}(s)}.
  \end{aligned}
  \]
  Here for any character $\omega\in\widehat{\K^{\times}}$, we write $\delta(\omega)\in\{0,1\}$ such that $\omega(-1)=(-1)^{\delta(\omega)}$.
    \item When $\K\cong\BC$,
    \[
    g_{\xi,\chi}(s):=\frac{\RL(s,\eta_0\chi_0^n)}{\RL(s,\eta\chi^n)}\cdot\prod_{i=1}^n\frac{\Gamma_\BC\left(s+\max\{\mu_i^{\iota}+\nu_{n+1-i}^{\iota}+\chi_{\iota},\mu_{n+1-i}^{\overline{\iota}}+\nu_{i}^{\overline{\iota}}+\chi_{\overline{\iota}}\}\right)}{\Gamma_{\C}(s)}.
\]
\end{itemize}
Here
	\[
	\Gamma_{\K}(s)=\begin{cases}
		\pi^{-s/2}\Gamma(s/2), & \K\cong\R;\\
		2(2\pi)^{-s}\Gamma(s), & \K\cong\C,
	\end{cases}
	\]
	with $\Gamma(s)$ the standard gamma function.
    
    By the balanced condition (Lemma \ref{lem-balanced}), $g_{\xi,\chi}(s)$ is holomorphic and nonzero at $s=0$. 
    
Recall that $\oH_{\xi,\chi}$ is a representation of $\pi_0(\K^\times)^3$. For every character $\varepsilon$ of $\pi_0(\K^\times)^3$, denote by 
$\oH_{\xi,\chi}[\varepsilon]$ the $\varepsilon$-isotypic component of $\oH_{\xi,\chi}$, which is at most one-dimensional.


	\begin{thmp} \label{mainthm}
		Retain the notations and assumptions as above.
		\begin{itemize}
			\item[(a)] The  diagram 
			\[
			\begin{CD}
            \oH_{\xi_0,\chi_0} \otimes \frak O_{n',\K} @> \wp_{\xi_0,\chi_0}>> \C\\
				@V\jmath_{\xi,\chi} VV @| \\
				\oH_{\xi,\chi}\otimes \frak O_{n',\K} @>\Omega_{\xi, \chi}\cdot \wp_{\xi,\chi} >>  \C 
			\end{CD}
			\]
			commutes, where
			\be
			\Omega_{\xi, \chi}=
			\begin{cases}
				c_{\xi,\chi} \cdot \varepsilon_{\xi,\chi},& \emph{Case $n'=n-1$};\\
c_{\xi,\chi} \cdot \varepsilon''_{\xi,\chi}\cdot g_{\xi,\chi}(0),& \emph{Case $(-)$};\\
				c'_{\xi,\chi} \cdot \varepsilon'_{\xi,\chi},& \emph{Case $(\pm)$},
			\end{cases}
			\ee
			\item[(b)] The archimedean modular symbol $\wp_{\xi,\chi}$ is non-vanishing. More precisely, its restriction to $\oH_{\xi,\chi}[\varepsilon]\otimes
            \frak O_{n',\K}$ is non-vanishing
           for every character $\varepsilon=\varepsilon_1\otimes \varepsilon_2\otimes \varepsilon_3$ of $\pi_0(\K^\times)^3$ occurring in 
           $\oH_{\xi,\chi}$ such that 
           $\varepsilon_1\cdot\varepsilon_2\cdot\varepsilon_3 = \sgn_{\K^\times}^{n'-1}$. 
		\end{itemize}
	\end{thmp}
	
	Here $\mathrm{sgn}_{\K^{\times}}$ is the sign character when $\K\cong\R$ and is understood as the trivial character when $\K\cong\C$.

	The non-vanishing of the archimedean modular symbol $\wp_{\xi_0,\chi_0}$ is proved in \cite{Sun} for Case $n'=n-1$, in \cite{DX} for Case $(\pm)$ and will be proved in Proposition \ref{prop:nonvanishing} for Case $(-)$. The general non-vanishing hypothesis in Theorem \ref{mainthm} (b) is a consequence of the period relation in (a). For period relations of $\GL_n\times\GL_{n-1}$, the proof in \cite{LLS24} for the essentially tempered representations goes through \emph{mutatis mutandis} for the general situation here. The main body of this article is devoted to the proof of Theorem \ref{mainthm} (a) for $\GL_n\times\GL_n$.
	
	\section{Translations and coefficient systems}

	In this section we prove Proposition \ref{prop:translationpi} and Proposition \ref{prop:translation}, for which we construct the translation maps $\jmath_{\mu}$ in \eqref{jmu00} and $\jmath_{\eta,\chi}$ in \eqref{jetachi} respectively. We will also discuss the balanced coefficient systems.

	\subsection{Translation of $\pi_{\mu}$}\label{sec:TP}

As in Section \ref{sec1.1}, let $k$ be a positive integer, $\mu\in(\Z^k)^{\mathcal{E}_{\K}}$  a dominant weight, and $F_{\mu}$  the irreducible holomorphic finite-dimensional representation of $\GL_k(\K\otimes_{\R}\C)$ of highest weight $\mu$. Take $\pi_{\mu}\in\Omega(\mu)$ to be a generic irreducible $F_{\mu}$-cohomological Casselman-Wallach representation of $\GL_k(\K)$. 
	
	For every $\iota\in\mathcal{E}_{\K}$, we set
	\begin{equation}\label{lmu}
		\widetilde{\mu}^{\iota}_i :=\mu^{\iota}_i+\frac{k+1-2i}{2},\qquad 1\leq i\leq k.
	\end{equation}
	Then the representation $\pi_{\mu}$ can be written as follows.
	\begin{itemize}
		\item When $\K\cong\R$ and $k$ is even,
		\[		\pi_{\mu}\cong\mathrm{Ind}^{\GL_k(\K)}_{\overline{\RR}_k(\K)}\left(D_{\widetilde{\mu}_1^{\iota},\widetilde{\mu}_k^{\iota}}\otimes
        \cdots\otimes D_{\widetilde{\mu}_{\frac{k}{2}}^{\iota},\widetilde{\mu}_{\frac{k}{2}+1}^{\iota}}\right).
		\]
		\item When $\K\cong\R$ and $k$ is odd,
		\[
		\pi_{\mu}\cong\mathrm{Ind}^{\GL_k(\K)}_{\overline{\RR}_k(\K)}\left(D_{\widetilde{\mu}_1^{\iota},\widetilde{\mu}_k^{\iota}}\otimes\cdots\otimes D_{\widetilde{\mu}_{\frac{k-1}{2}}^{\iota},\widetilde{\mu}_{\frac{k+3}{2}}^{\iota}}\otimes(\cdot)_{\K}^{\widetilde{\mu}^{\iota}_{\frac{k+1}{2}}}\sgn_{\K^\times}^{\frac{k-1}{2}}\varepsilon_{\pi_{\mu}}\right),
		\]
		where $\varepsilon_{\pi_{\mu}}:\K^\times \to \{\pm1\}$ is the central character of $F_{\mu}^{\vee}\otimes\pi_{\mu}$.
		
		\item When $\K\cong\C$,
		\[
		\pi_{\mu}\cong\mathrm{Ind}^{\GL_k(\K)}_{\overline{\RB}_k(\K)}\left(\iota^{\widetilde{\mu}_1^{\iota}}\overline{\iota}^{\widetilde{\mu}_k^{\overline{\iota}}}\otimes\cdots\otimes\iota^{\widetilde{\mu}_k^{\iota}}\overline{\iota}^{\widetilde{\mu}_1^{\overline{\iota}}}\right).
		\]
	\end{itemize}
	Here and henceforth,
	\begin{itemize}
		\item $\overline{\RR}_k$ is the lower triangular parabolic subgroup of type $(2,\dots,2)$ if $k$ is even and of type $(2,\dots,2,1)$ if $k$ is odd; 
		\item When $\K\cong\R$, $D_{a,b}$ is the relative discrete series of $\GL_2(\K)$ with infinitesimal character $(a,b)$, for all $a,b\in\C$ with $a-b\in\Z\backslash\{0\}$; 
		\item When $\K\cong\C$, $\iota^a\overline{\iota}^b$ is the character $z\mapsto\iota(z)^{a-b}(\iota(z)\overline{\iota}(z))^b$ of $\K^\times$ for $a,b\in\C$ with $a-b\in\Z$.
	\end{itemize}

\subsubsection{The map $\jmath_{\mu}$}	\label{sec2.1.1}

The translation map $\jmath_{\mu}$ in Proposition \ref{prop:translationpi} is constructed in \cite[Section 2]{LLS24} when $\pi_{\mu}$ is essentially tempered.
We are going to make a corrigendum to {\it loc. cit.} in the case when $\K\cong \R$ and $k$ is odd, and extend the construction to the general case.
Define the principal series representation
	\[
	I_{\mu}:=\mathrm{Ind}^{\GL_k(\K)}_{\overline{\RB}_k(\K)}(\chi_{\mu}\cdot\rho_k\cdot \varepsilon_{\pi_{\mu}}')={^{\mathrm{u}}\mathrm{Ind}}^{\GL_k(\K)}_{\overline{\RB}_k(\K)}(\chi_{\mu}\cdot \varepsilon_{\pi_{\mu}}'),
	\]
	where $\rho_k$ and $\varepsilon_{\pi_\mu}'$ are characters of $\RT_k(\K)$ given by
	\[
    \begin{aligned}
	\rho_k:& =\bigotimes_{i=1}^k|\cdot|_{\K}^{\frac{k+1}{2}-i}; \\
    \varepsilon_{\pi_\mu}':&= \underbrace{{\bf 1}_{\K^\times}\otimes\cdots\otimes {\bf 1}_{\K^\times}}_{\frac{k-1}{2} \text{ times}} \otimes 
    \sgn_{\K^\times}^{\frac{k-1}{2}} \varepsilon_{\pi_\mu} \otimes \underbrace{\sgn_{\K^\times}\otimes \cdots\otimes \sgn_{\K^\times}}_{\frac{k-1}{2} \text{ times} }
    \end{aligned}
    \]
    if $\K\cong\R$ and $k$ is odd, and $\varepsilon_{\pi_\mu}'$ is trivial otherwise. 
	 As in \cite[Lemma 2.2]{LLS24}, $\pi_{\mu}$ is the unique irreducible  quotient representation of $I_{\mu}$.

	By \cite[Theorem 15.4.1]{Wa2}, 
	$
	\dim\Hom_{\mathrm{N}_k(\K)}(I_\mu, \psi_{k,\K}) =1
	$
	and there is a unique homomorphism 
	\[
	\wt{\lambda}_{\mu} \in \Hom_{\mathrm{N}_k(\K)}(I_\mu, \psi_{k,\K})
	\] such that 
	\be \label{la'} 
	\wt{\lambda}_{\mu}(f)=  \int_{\mathrm{N}_k(\K)} f(u)\overline{\psi_{k,\K}}(u)\od\! u
	\ee
	for all $f\in I_\mu$ with $f|_{\mathrm{N}_k(\K)}\in \CS(\mathrm{N}_k(\K))$. Here and henceforth, for a Nash manifold $X$, denote by $\CS(X)$ the space of Schwartz functions on $X$
	(see \cite{AG1,Du}).

    An element $u\otimes f\in F_\mu^\vee\otimes I_\mu$ is identified with the function 
    \[
    \GL_k(\K)\to F_\mu^\vee,\quad g\mapsto f(g)\cdot u.
    \]
	Then $F_\mu^\vee\otimes I_\mu$ is identified with the space of $F_\mu^\vee$-valued smooth functions $\varphi$ on $\GL_k(\K)$ satisfying that 
	\[
	\varphi(bx)= \left(\chi_{\mu}\cdot \varepsilon_{\pi_{\mu}}'\right)(b) \cdot \varphi(x),\quad b\in \bar{\mathrm B}_k(\K),\quad x\in \GL_k(\K),
	\]
	on which $\GL_k(\K)$ acts by
	\[
	(g . \varphi)(x):=g. (\varphi(xg)),\quad g, \,  x \in \GL_k(\K).
	\]
	Define a map
	\be \label{imu}
    \begin{aligned}
	&\imath_\mu\in \Hom_{\GL_k(\K)}(I_{0_{k}},  F_\mu^\vee \otimes I_\mu), \\&\imath_\mu(f)(g):= f(g)\cdot (g^{-1}. v_\mu^\vee),  \  g\in \GL_k(\K).
	\end{aligned}
    \ee
	
	As in \cite[Lemma 2.3]{LLS24}, the map $\imath_\mu$ satisfies that 
	\be \label{imu-whit}
	(v_\mu \otimes \wt{\lambda}_{\mu} )  \circ \imath_\mu =  \wt{\lambda}_{0_{k}}.
	\ee
	There is a unique homomorphism $p_\mu\in \Hom_{\GL_k(\K)}(I_\mu, \pi_\mu)$ such that 
	\[
	\lambda_{\mu} \circ p_\mu=\wt{\lambda}_{\mu},
	\]
	and as in \cite[Lemma  2.4]{LLS24} there is a homomorphism 
    \[
    \jmath_\mu\in \Hom_{\GL_k(\K)}(\pi_{0_{k}},  F_\mu^\vee\otimes \pi_\mu)
    \]
    such that the  diagram 
	\be \label{I-pi}
	\begin{CD}
		I_{0_{k}}
		@>   p_{0_{k}} >> \pi_{0_{k}}   \\
		@ V  \imath_\mu VV        @VV  \jmath_\mu  V\\
		F_\mu^\vee\otimes I_\mu     @> {\rm id} \otimes p_\mu >>  F_\mu^\vee\otimes \pi_\mu\\
	\end{CD}
	\ee
	commutes.
	Moreover, $\jmath_\mu$ satisfies that
	\[
	(v_\mu\otimes \lambda_{\mu})\circ \jmath_\mu = \lambda_{0_{k}}.
	\]
	That is, the diagram \eqref{diagjmu} commutes. The uniqueness of $\jmath_{\mu}$ and the isomorphism \eqref{jmucohomology} follow from the translation principle (\cite[Section 7.3]{Vo81}, see also \cite[Section 5]{VZ}). This completes the proof of Proposition \ref{prop:translationpi}.

\subsubsection{Reformulation}
    
	It will be more convenient to reformulate the above results via the natural isomorphisms 
	\begin{equation} \label{dual}
		\begin{aligned}
			& \Hom_{\GL_k(\K)}(I_{0_{k}}, F_\mu^\vee\otimes I_\mu) \cong \Hom_{\GL_k(\K)}(F_\mu\otimes I_{0_{k}}, I_\mu), \\
			&  \Hom_{\GL_k(\K)}(\pi_{0_{k}}, F_\mu^\vee\otimes \pi_\mu) \cong \Hom_{\GL_k(\K)}(F_\mu\otimes \pi_{0_{k}}, \pi_\mu).
		\end{aligned}
	\end{equation} 
	Denote by $\check\imath_\mu$ and $\check\jmath_\mu$ the images of $\imath_\mu$ and $\jmath_\mu$ in the right hand sides of \eqref{dual}, respectively. Using \eqref{imu}, it is straightforward to check that the map
	\[
	\check\imath_\mu\in \Hom_{\GL_k(\K)}(F_\mu\otimes I_{0_{k}}, I_\mu)
	\]
	is given by the multiplication of functions, that is,
	\[
	\check\imath_\mu( \varphi\otimes f)(g) = \varphi(g)\cdot f(g), \quad \varphi\in F_\mu, \ f\in I_{0_k}, \ g\in\GL_k(\K).
	\]
	The equation \eqref{imu-whit} can be rephrased as 
	\[
	\wt{\lambda}_{\mu} \circ \check\imath_\mu\circ(v_\mu \otimes \cdot ) = \wt{\lambda}_{0_{k}} \in \Hom_{\RN_k(\K)}(I_{0_{k}}, \psi_{k,\K}).
	\]
	We also have the commutative diagram
	\be \label{I-pi-dual}
	\begin{CD}
		F_\mu\otimes I_{0_{k}}
		@>   {\rm id}\otimes p_{0_{k}} >> F_\mu\otimes \pi_{0_{k}}   \\
		@ V  \check\imath_\mu VV        @VV  \check\jmath_\mu  V\\
		I_\mu     @>  p_\mu >>   \pi_\mu\\
	\end{CD}
	\ee
	and that
	\[
	\lambda_{\mu}\circ \check\jmath_\mu\circ(v_\mu\otimes\cdot) = \lambda_{0_{k}}\in 
	\Hom_{\RN_k(\K)}(\pi_{0_{k}}, \psi_{k,\K}).
	\]

	\subsection{Theta lifting for $\GL_1\times\GL_n$}
	\label{sec:theta}
	
	For two characters $\eta,\chi:\K^{\times}\to\C^{\times}$, consider the degenerate principal series representation $I_{\eta,\chi}$ as in \eqref{Ietachi}. We discuss its relation with the theta lifting for the type II dual pair $\GL_1(\K)\times\GL_n(\K)$.

	Let $\omega_n:=\CS(\K^{1\times n})$ be the (unnormalized) Weil representation of $\GL_1(\K)\times\GL_n(\K)$ with the action given by
	\[
	(a,g).\Phi(x)=\Phi(a^{-1}xg)
	\]
    for $a\in\K^{\times},\,g\in\GL_n(\K),\,x\in\K^{1\times n}$ and $\Phi\in\CS(\K^{1\times n})$. Put $\CS_0(\K^{1\times n}):=\CS(\K^{1\times n}\backslash\{0\})\hookrightarrow\CS(\K^{1\times n})$ and let $\omega_n^{\circ}:=\CS_0(\K^{1\times n})$ carry the restricted action. Denote by
	\[
	\Theta_n(\eta\chi^n):=(\omega_n\otimes\eta^{-1}\chi^{-n})_{\K^{\times}}
    \]
    and
    \[
    \Theta^{\circ}_n(\eta\chi^n):=(\omega^{\circ}_n\otimes\eta^{-1}\chi^{-n})_{\K^{\times}}
	\]
	the maximal Hausdorff coinvariant spaces. The embedding $\CS_0(\K^n)\hookrightarrow\CS(\K^n)$ induces a natural $\GL_n(\K)$-homomorphism
	\[
	\mathrm{j}_{\eta,\chi}:\Theta_n^{\circ}(\eta\chi^n)\otimes(\chi\circ\det)\to\Theta_n(\eta\chi^n)\otimes(\chi\circ\det).
	\]

       Write 
       \be
       {\rm f}_{\eta,\chi}: \Theta_n(\eta\chi^n)\otimes(\chi\circ\det)\to I_{\eta,\chi}
       \ee
       for the homomorphism defined by evaluating the following normalized Tate integral at $s=0$:
       \[
       ({\rm f}_{\eta,\chi}(\Phi))(g):=\left. \frac{\chi(\det g)}{\oL(s,\eta\chi^n)}\int_{\K^\times}\Phi(a e_n g)\eta(a)\chi^n(a) |a|_\K^s\od\!^\times a\right|_{s=0},
       \]
       where $\Phi\in\CS(\K^{1\times n})$ and $e_n:=[0\  \cdots \ 0 \ 1]\in\K^{1\times n}$. Set
       \be \label{imf}
       \tau_{\eta,\chi}:={\rm im}({\rm f}_{\eta,\chi}) \subset I_{\eta,\chi}.
       \ee
We also have an isomorphism 
\be \label{fcirc}
{\rm f}_{\eta,\chi}^\circ: \Theta_n^\circ(\eta\chi^n)\otimes (\chi\circ\det) \xrightarrow{\cong} I_{\eta,\chi}
\ee
induced by the following absolutely convergent integral
\[
({\rm f}_{\eta,\chi}^\circ(\Phi))(g):=\chi(\det g)\cdot\int_{\K^\times}\Phi(ae_ng)\eta(a)\chi^n(a)\od\!^\times a,
\]
for $\Phi\in \CS_0(\K^{1\times n})$.

	We record the main result of \cite[Theorem 1.1]{X} on the structure of $\Theta_n(\eta\chi^n)$ and $I_{\eta,\chi}$ in the following proposition. 
	
	\begin{prp} \label{thmX}
		Assume that $n\geq 2$. If $s=0$ is not a pole of $\RL(s,\eta\chi^n)$ and $s=n$ is not a pole of $\RL(s,\eta^{-1}\chi^{-n})$, then both $\mathrm{j}_{\eta,\chi}$ and $\mathrm{f}_{\eta,\chi}$ are isomorphisms of irreducible representations. Otherwise, $\Theta_n(\eta\chi^n)$  has length $2$ with an infinite-dimensional irreducible subrepresentation and a finite-dimensional irreducible quotient. 
		\begin{itemize}
			\item[(a)] If $s=n$ is a pole of $\RL(s,\eta^{-1}\chi^{-n})$, then both $\mathrm{j}_{\eta,\chi}$ and $\mathrm{f}_{\eta,\chi}$ are isomorphisms. 
			\item[(b)] If $s=0$ is a pole of $\RL(s,\eta\chi^n)$, then 
			\[
		\mathrm{ker}(\mathrm{j}_{\eta,\chi}) \cong \mathrm{coker}({\rm j}_{\eta,\chi}) \cong {\rm im}({\rm f}_{\eta,\chi})=\tau_{\eta,\chi}.
			\]
            In this case $\tau_{\eta,\chi}$ is the  irreducible quotient of $\Theta_n(\eta\chi^n)\otimes (\chi\circ\det)$ and the irreducible subrepresentation of $I_{\eta,\chi}$.
		\end{itemize}
	\end{prp}

When $n=1$, we understand $I_{\eta,\chi}=\eta^{-1}$ as a character which is clearly irreducible. 

    	Let $\pi_{\mu}$, $\pi_{\nu}$ be as in Section \ref{sec:1.3} and let $\eta$ be the product of their central characters. Write for short
	\[
	\omega_{n,\chi'}:=\omega_n\otimes(\chi'\circ\det).
	\]
    We have the Rankin-Selberg integral 
    \[
\RZ'_{\xi,\chi'}:\pi_{\xi}\,\widehat{\otimes}\,\omega_{n,\chi'}\to\mathfrak{M}_{n,\K}^{\ast}
    \]
defined by
\begin{equation}\label{integraltheta}
    \begin{aligned}
        &\RZ_{\xi,\chi'}'(f,f',\Phi;\mathrm{d}g)\\
            :=&\int_{\RN_n(\K)\backslash\GL_n(\K)}\left\langle\lambda_{\mu},g.f\right\rangle\left\langle\lambda'_{\nu},g.f'\right\rangle\Phi(e_ng)\chi'(\det g)\overline{\mathrm{d}}g
    \end{aligned}
\end{equation}
 for $f\in\pi_{\mu}$, $f'\in\pi_{\nu}$, $\Phi\in\CS(\K^{1\times n})$, $\chi'\in\widehat{\K^{\times}}$ and $\mathrm{d}g\in\mathfrak{M}_{n,\K}$. The integral converges absolutely when $\mathrm{Re}(\chi')$ is sufficiently large and has meromorphic continuation to the complex Lie group $\widehat{\K^{\times}}$. By \cite{J09}, the normalized Rankin-Selberg integral
\begin{equation}
\RZ^{\Theta}_{\xi,\chi'}(f,f',\Phi;\mathrm{d}g):=\frac{1}{\RL(0,\pi_{\mu}\times\pi_{\nu}\times\chi')}\cdot\RZ_{\xi,\chi'}'(f,f',\Phi;\mathrm{d}g)
\end{equation}
is holomorphic and induces a nonzero linear functional
\[
	\RZ_{\xi,\chi'}^{\Theta}\in\mathrm{Hom}_{\GL_n(\K)}\left(\pi_{\xi}\,\widehat{\otimes}\,\omega_{n,\chi'},\mathfrak{M}_{n,\K}^{\ast}\right).
	\]

	\subsection{Translations of $I_{\eta,\chi}$}
	\label{sec:Translation}

	Assume that $\eta,\chi$ are algebraic characters satisfying \eqref{regular} and keep the convention \eqref{iota}. Let $F_{\eta,\chi}$ be the irreducible holomorphic finite-dimensional representation of $\GL_n(\K\otimes_{\R}\C)$ whose infinitesimal character equals that of $I_{\eta,\chi}$. Recall that $F_{\eta,\chi}$ is realized as the algebraic induction in \eqref{Fmu} with a fixed highest weight vector, denoted by $v_{\eta,\chi}\in(F_{\eta,\chi})^{\RN_n(\K\otimes_{\R}\C)}$.
	
	Let $\mathcal{D}(\K^{1\times n})$ denote the algebra of polynomial coefficient differential operators on the real vector space $\K^{1\times n}$. Then $\mathcal{D}(\K^{1\times n})=\otimes_{\iota'\in\mathcal{E}_{\K}}\mathcal{D}_{n,\iota'}$ in view of $\K^{1\times n}\otimes_{\R}\C=\prod_{\iota'\in\mathcal{E}_{\K}}\C^{1\times n}_{\iota'}$, where $\mathcal{D}_{n,\iota'}$ is the algebra of polynomial coefficient differential operators on $\C_{\iota'}^{1\times n}$. For each $r\in\Z$, when $r\geq 0$, let $\mathrm{Sym}_{\iota'}^r\subset\mathcal{D}_{n,\iota'}$ denote the subspace consisting of all homogeneous polynomial functions on $\C_{\iota'}^{1\times n}$ of homogeneous degree $r$; when $r\leq 0$, let $\mathrm{Sym}_{\iota'}^r\subset\mathcal{D}_{n,\iota'}$ denote the subspace consisting of all homogeneous differential operators on $\C_{\iota'}^{1\times n}$ of homogeneous degree $-r$. For all $r\in\Z$, $\mathrm{Sym}_{\iota'}^r$ is an irreducible holomorphic finite-dimensional representation of $\GL_n(\C_{\iota'})$ of extremal weight $(r,0,\dots,0)$. 

    Let
    \[
    F'_{\eta,\chi} = \begin{cases} F_{\chi}\otimes \bigotimes\limits_{\iota'\in\CE_\K}\mathrm{Sym}_{\iota'}^{-\eta_{\iota'}-n\chi_{\iota'}}, & \text{Case $(-)$};\\
    F_{\chi\cdot\prod_{\iota'\in\CE_\K}(\iota'|_{\K^\times})^{-1}}\otimes \bigotimes\limits_{\iota'\in\CE_\K}\mathrm{Sym}_{\iota'}^{n-\eta_{\iota'}-n\chi_{\iota'}}, & \text{Case $(+)$}; \\
     F_{\chi\cdot (\bar \iota|_{\K^\times})^{-1}}\otimes(\mathrm{Sym}_\iota^{-\eta_{\iota}-n\chi_{\iota}}\otimes \mathrm{Sym}_{\bar \iota}^{n-\eta_{\bar \iota}-n\chi_{\bar \iota}}), & \text{Case $(\pm)$}.
     \end{cases}
    \]
    We fix a highest weight vector of $\mathrm{Sym}^r_{\iota'}$ to be ${x_1^{\iota'}}^r$ (resp. $(\varepsilon_{\psi_{\R}}\cdot2\pi\mathrm{i})^r\partial_{x_n^{\iota'}}^{-r}$) when $r\geq 0$ (resp. $r\leq 0$). This provides a highest weight vector $v_{\eta,\chi}'$ of $F_{\eta,\chi}'$. 
    We fix the identification
    \[
    F_{\eta,\chi}\xrightarrow{\cong} F_{\eta,\chi}',\qquad v_{\eta,\chi}\mapsto v_{\eta,\chi}',
    \]
    and denote both representations by $F_{\eta,\chi}$ in the sequel.

	We are going to construct the translation map 
	\[
	\check{\jmath}_{\eta,\chi}\in\mathrm{Hom}_{\GL_n(\K)}\left(F_{\eta,\chi}\otimes F_{\eta_0,\chi_0}^{\vee}\otimes I_{\eta_0,\chi_0}, I_{\eta,\chi}\right)
	\]
	which can be identified with $\jmath_{\eta,\chi}$ in \eqref{jetachi}. To ease the presentation, we omit the discussions for Case $(+)$, which will not be used in this article and can be treated similarly to the other cases.

	\subsubsection{Action of differential operators}
	
	Write $F_{\eta,\chi}\otimes F_{\eta_0,\chi_0}^{\vee}=F_{\chi\chi_0^{-1}}\otimes\mathfrak{D}_{\eta,\chi}$ with $\mathfrak{D}_{\eta,\chi}\subset\mathcal{D}(\K^{1\times n})$. Using the action of differential operators, we first define a map
	\begin{equation}
		\alpha_{\eta,\chi}:\mathfrak{D}_{\eta,\chi}\otimes\omega_n\to\omega_n,\qquad D\otimes\Phi\mapsto D\Phi
	\end{equation}
	for $D\in\mathfrak{D}_{\eta,\chi}$, $\Phi\in\CS(\K^{1\times n})$. Denote the natural projections by
	\[
	\begin{aligned}
		\mathrm{p}_{\eta_0,\chi_0}:\omega_n\twoheadrightarrow\Theta_n(\eta_0\chi_0^n),\qquad\mathrm{p}_{\eta,\chi}:\omega_n\twoheadrightarrow\Theta_n(\eta\chi^n).
	\end{aligned}
	\]
	
	\begin{lemp} \label{tran-S}
		The composition 
		\[
		{\rm p}_{\eta, \chi}\circ \alpha_{\eta, \chi}: \mathfrak{D}_{\eta,\chi} \otimes \omega_n \to \Theta_n(\eta\chi^n)
		\]
		factors through the quotient $\mathrm{id}\otimes\mathrm{p}_{\eta_0,\chi_0}$, hence induces a commutative diagram 
		\[
		\CD
		\mathfrak{D}_{\eta,\chi} \otimes \omega_n  @> \alpha_{\eta, \chi}  >> \omega_n \\
		@V {\rm id}\otimes {\rm p}_{\eta_0, \chi_0} V V @ VV {\rm p}_{\eta, \chi} V \\
		\mathfrak{D}_{\eta,\chi} \otimes \Theta_n(\eta_0\chi_0^n) @>   >> \Theta_n(\eta\chi^n).
		\endCD
		\]
	\end{lemp}
	
	\begin{proof}
		For Case $(\pm)$, since $\Theta_n(\eta\chi^n)\cong\Theta^{\circ}_n(\eta\chi^n)$, it suffices to show that 
		if $\Phi\in \CS_0(\K^{1\times n})$ satisfies that
		\[
		\int_{\K^{\times}}\Phi(ae_ng)\eta_0(a)\chi_0^n(a)\mathrm{d}^{\times} a\equiv 0,
		\]
		then 
		\[
		\int_{\K^{\times}}D\Phi(ae_ng)\eta(a)\chi^n(a)\mathrm{d}^{\times} a\equiv 0
		\]
		for any $D\in \mathfrak{D}_{\eta,\chi}$. Define a smooth function on $\K^n\backslash\{0\}$ by
		\[
		f(x) = \int_{\K^{\times}}\Phi(ax)\eta_0(a)\chi_0(a)^n\mathrm{d}^{\times} a,\quad x\in \K^n\backslash\{0\}.
		\]
		Then it is easy to check that
		\[
		\int_{\K^{\times}}D\Phi(ax)\eta(a)\chi^n(a)\mathrm{d}^{\times} a = Df(x)
		\]
		for any $D\in \mathfrak{D}_{\eta,\chi}$. For Case $(-)$, $\mathfrak{D}_{\eta,\chi}$ consists of homogeneous polynomials and the lemma can be verified similarly.
	\end{proof}
	
	Twisted by $F_{\chi\chi_0^{-1}}$, the map $\alpha_{\eta,\chi}$ induces
	\begin{equation}
		\beta_{\eta,\chi}:F_{\eta,\chi}\otimes F_{\eta_0,\chi_0}^{\vee}\otimes\omega_{n,\chi_0}\to\omega_{n,\chi}.
	\end{equation}

	We make the normalization
\[
\beta_{\eta,\chi}^{\circ}:=C_{\eta,\chi}^{-1}\cdot\beta_{\eta,\chi}
\]
where
	\begin{equation}\label{normalizingfactor}
	C_{\eta,\chi}:=\begin{cases}
 (-1)^{\frac{n(n-1)}{2}\sum_{\iota'\in\CE_{\K}}\chi_{\iota'}}\cdot\left.\frac{\RL(s,\eta_0\chi_0^n)}{\RL(s,\eta\chi^n)}\right|_{s=0}, & \text{Case }(-);\\
(-1)^{\frac{n(n-1)}{2}\chi_{\iota}}\cdot  (\varepsilon_{\psi_{\R}}\cdot\mathrm{i})^{\eta_{\overline{\iota}}+n\chi_{\overline{\iota}}-n}, &\text{Case }(\pm).
	\end{cases}
	\end{equation}
	By Lemma \ref{tran-S}, the composition $\mathrm{p}_{\eta,\chi}\circ\beta^{\circ}_{\eta,\chi}$ factors through the quotient $\mathrm{id}\otimes\mathrm{p}_{\eta_0,\chi_0}$. We will construct the map $\check{\jmath}_{\eta,\chi}$ such that the diagram	
    \begin{equation}
		\begin{CD}\label{RIdiagram}
			F_{\eta,\chi}\otimes F_{\eta_0,\chi_0}^{\vee}\otimes\omega_{n,\chi_0} @>\beta^{\circ}_{\eta,\chi}>>\omega_{n,\chi}\\
			@V\mathrm{f}_{\eta_0,\chi_0}VV @VV\mathrm{f}_{\eta,\chi}V\\
			F_{\eta,\chi}\otimes F_{\eta_0,\chi_0}^{\vee}\otimes I_{\eta_0,\chi_0} @>\check{\jmath}_{\eta,\chi}>> I_{\eta,\chi}
		\end{CD}
	\end{equation}
commutes.

	\subsubsection{\emph{Case $(-)$}}
	
	For Case $(-)$ we can define the map
	\[
	\check{\jmath}_{\eta,\chi}:F_{\eta,\chi}\otimes F_{\eta_0,\chi_0}^{\vee}\otimes I_{\eta_0,\chi_0}\to I_{\eta,\chi}
	\]
	directly by multiplication of polynomials. Take the case $\K\cong\C$ for example (the case $\K\cong\R$ is similar). We have $\mathfrak{D}_{\eta,\chi}=\mathrm{Sym}_{\iota}^{-\eta_{\iota}-n\chi_{\iota}}\otimes\mathrm{Sym}_{\overline{\iota}}^{-\eta_{\overline{\iota}}-n\chi_{\overline{\iota}}}$. We define
	\[
	\mathfrak{D}_{\eta,\chi}\otimes I_{\eta_0,\chi_0}\to I_{\eta,\chi},\qquad (P\otimes Q)\otimes\varphi\mapsto\widetilde{\varphi},
	\]
	where 
	\[
	\widetilde{\varphi}(g):=(-1)^{\frac{n(n-1)}{2}\sum_{\iota'\in\CE_{\K}}\chi_{\iota'}}\cdot\chi(\det g)P(e_n g^\iota)Q(e_n g^{\bar\iota})\cdot\varphi(g), \quad g\in \GL_n(\K),
	\]
    for polynomials $P\in\mathrm{Sym}_\iota^{-\eta_{\iota}-n\chi_{\iota}}$, $Q\in \mathrm{Sym}_{\overline{\iota}}^{-\eta_{\overline{\iota}}-n\chi_{\overline{\iota}}}$ and $\varphi\in I_{\eta_0,\chi_0}$.
     This induces the desired map $\check{\jmath}_{\eta,\chi}$ satisfying \eqref{RIdiagram}. The diagram \eqref{jetachidiagram} can be verified straightforwardly and the induced isomorphism \eqref{jetachicohomology} between cohomology groups can be easily seen from the fact that
    \[
    \begin{aligned}
\RH_{\mathrm{ct}}^0(\R_+^{\times}\backslash\GL_n(\K)^0;F_{\eta,\chi}^{\vee}\otimes I_{\eta,\chi})=&\RH_{\mathrm{ct}}^0(\R_+^{\times}\backslash\GL_n(\K)^0;F_{\eta,\chi}^{\vee}\otimes \tau_{\eta,\chi})\\
=&\left(F_{\eta,\chi}^{\vee}\otimes \tau_{\eta,\chi}\right)^{\GL_{n}(\K)^0}.
\end{aligned}
    \]
    This proves Proposition \ref{prop:translation} for Case $(-)$.

	\subsubsection{\emph{Case $(\pm)$}}

	Since $I_{\eta,\chi}\cong\Theta_n(\eta\chi^n)\otimes(\chi\circ\det)$ in this case, the composition $\mathrm{p}_{\eta,\chi}\circ\beta^{\circ}_{\eta,\chi}$ induces the desired map
	\[
	\check{\jmath}_{\eta,\chi}:F_{\eta,\chi}\otimes F_{\eta_0,\chi_0}^{\vee}\otimes I_{\eta_0,\chi_0}\to I_{\eta,\chi}
	\]
	satisfying \eqref{RIdiagram}. The diagram \eqref{jetachidiagram} can be verified straightforwardly. To justify the induced isomorphism \eqref{jetachicohomology} between cohomology groups, recall from \cite[Theorem 6.1]{HM62} that the continuous cohomology can be calculated by the relative Lie algebra cohomology which is the cohomology of the complex
\[
\RC^i(F_{\eta,\chi}^{\vee}\otimes I_{\eta,\chi}):=\mathrm{Hom}_{\mathrm{SU}(n)}(\wedge^i\mathfrak{p}_n,F_{\eta,\chi}^{\vee}\otimes I_{\eta,\chi}),
\]
where
$
\mathfrak{p}_n:=\left(\mathfrak{gl}_{n,\K}/(\mathfrak{k}_{n,\K}\oplus\R)\right)\otimes_{\R}\C
$
and the differential
\[
d^i_{\eta,\chi}:\RC^{i}(F_{\eta,\chi}^{\vee}\otimes I_{\eta,\chi})\to\RC^{i+1}(F_{\eta,\chi}^{\vee}\otimes I_{\eta,\chi})
\]
is given in \cite[(2.127)]{Kn95}.
The translation map $\jmath_{\eta,\chi}$ induces a map
\[
\jmath_{\eta,\chi}:\RC^{i}(F_{\eta_0,\chi_0}^{\vee}\otimes I_{\eta_0,\chi_0})\to\RC^{i}(F_{\eta,\chi}^{\vee}\otimes I_{\eta,\chi}).
\]
The following lemma  completes the proof of Proposition \ref{prop:translation}.

\begin{lemp}
    $\RH_{\mathrm{ct}}^i\left(\R_+^{\times}\backslash\GL_n(\K)^0;F_{\eta,\chi}^{\vee}\otimes I_{\eta,\chi}\right)=\jmath_{\eta,\chi}\left(\RC^i(F_{\eta_0,\chi_0}^{\vee}\otimes I_{\eta_0,\chi_0})\right)$ for all $i\in\BZ$.
\end{lemp}
    
\begin{proof}
We have the following commutative diagram
\[
\begin{CD}
\RC^{i}(F_{\eta_0,\chi_0}^{\vee}\otimes I_{\eta_0,\chi_0}) @> d_{\eta_0,\chi_0}^{i}>> \RC^{i+1}(F_{\eta_0,\chi_0}^{\vee}\otimes I_{\eta_0,\chi_0}) \\
@V\jmath_{\eta,\chi}VV @V\jmath_{\eta,\chi}VV \\
\RC^{i}(F_{\eta,\chi}^{\vee}\otimes I_{\eta,\chi}) @>d_{\eta,\chi}^{i}>>\RC^{i+1}(F_{\eta,\chi}^{\vee}\otimes I_{\eta,\chi}).
\end{CD}
\]
Since $I_{\eta_0,\chi_0}$ is unitary when restricted to $\SL_n(\K)$, by \cite[Proposition 9.4.3]{Wa88} we have $d_{\eta_0,\chi_0}^i=0$ and
    \begin{equation}\label{1021}
\RH_{\mathrm{ct}}^i\left(\R_+^{\times}\backslash\GL_n(\K)^0;F_{\eta_0,\chi_0}^{\vee}\otimes I_{\eta_0,\chi_0}\right)=\RC^i(F_{\eta_0,\chi_0}^{\vee}\otimes I_{\eta_0,\chi_0}).
    \end{equation}
    
It follows that 
    \[
(d^i_{\eta,\chi}\circ\jmath_{\eta,\chi})\left(\RC^{i}(F_{\eta_0,\chi_0}^{\vee}\otimes I_{\eta_0,\chi_0})\right)=(\jmath_{\eta,\chi}\circ d^i_{\eta_0,\chi_0})\left(\RC^{i}(F_{\eta_0,\chi_0}^{\vee}\otimes I_{\eta_0,\chi_0})\right)=\{0\}.
\]
Put $M=F_{\eta,\chi}^{\vee}\otimes I_{\eta,\chi}$, and let $M_0$ be its generalized infinitesimal character summand corresponding to the trivial representation and write $M=M_0\oplus M_0'$. We have 
\[
M_0=\jmath_{\eta,\chi}\left(F_{\eta_0,\chi_0}^{\vee}\otimes I_{\eta_0,\chi_0}\right)
\]
by \cite[Proposition 7.4.1]{Vo81}, and
\[
\RH_{\mathrm{ct}}^i\left(\R_+^{\times}\backslash\GL_n(\K)^0;M_0'\right)=0
\]
by \cite[Theorem I.5.3]{BW}. In view of \eqref{1021}, the natural map from cocycles to cohomology induces an isomorphism
\[
\RC^i(M_0)=\jmath_{\eta,\chi}\left(\RC^{i}(F_{\eta_0,\chi_0}^{\vee}\otimes I_{\eta_0,\chi_0})\right)\xrightarrow{\sim}\RH_{\mathrm{ct}}^{i}\left(\R_+^{\times}\backslash\GL_n(\K)^0;F_{\eta,\chi}^{\vee}\otimes I_{\eta,\chi}\right),
\]
which completes the proof of the lemma.
\end{proof}

    \subsubsection{Fourier transform}

    We remain in Case $(\pm)$. For later computations, we rewrite $\beta_{\eta,\chi}$ in terms of the Fourier transform. To ease the notation, we set
	\[
	a_{\eta,\chi,\iota}:=-\eta_{\iota}-n\chi_{\iota},\qquad a_{\eta,\chi,\overline{\iota}}:=\eta_{\overline{\iota}}+n\chi_{\overline{\iota}}-n.
	\]
	
	We have the linear dual $\K^{n\times 1}$ of $\K^{1\times n}$ and we define $\check{\omega}_n:=\CS(\K^{n\times 1})$ to be a representation of $\GL_1(\K)\times\GL_n(\K)$ with the action given by
	\[
	(a,g).\Phi'(y)=|a|_{\K}^{n}\cdot|\det(g)|_{\K}^{-1}\cdot\Phi'(ag^{-1}y)
	\]
	for $a\in\K^{\times}$, $g\in\GL_n(\K)$, $y\in\K^{n\times 1}$, $\Phi'\in\CS(\K^{n\times 1})$. For a character $\chi$, we also write for short \[
    \check{\omega}_{n,\chi}:=\check{\omega}_n\otimes(\chi\circ\det).
    \]
	
	\begin{lemp} \label{lem:FS}
		The Fourier transform with respect to $\psi_\K$ defines a $\GL_n(\K)$-isomorphism
		\[
		\begin{aligned}
			\CF_{\psi_\K}:  \omega_n&\to \check{\omega}_n,\\
			\Phi&\mapsto \left(\wh{\Phi}(y): = \CF_{\psi_\K}(\Phi)(y):=\int_{\K^n}\Phi(x) \psi_\K(- x y) \od\! x\right).
		\end{aligned}
		\]
	\end{lemp}
	
	Recall that 
	\[
	\Cd(\K^{1\times n})=\Cd(\C_{\iota}^{1\times n})\otimes\Cd(\C_{\overline{\iota}}^{1\times n})=\C\langle x_i^{\iota},\partial_{x_i^{\iota}}\rangle\otimes\C\langle x_i^{\overline{\iota}},\partial_{x_i^{\overline{\iota}}}\rangle
	\]
	is the algebra of all polynomial coefficient differential operators on $\K^{1\times n}\otimes_{\R}\C\cong\C_{\iota}^{1\times n}\times\C_{\overline{\iota}}^{1\times n}$, where we write $x_1^{\iota},\dots,x_n^{\iota}$ (resp. $x_1^{\overline{\iota}},\dots,x_n^{\overline{\iota}}$) for the coordinates of $\C_{\iota}^{1\times n}$ (resp. $\C_{\overline{\iota}}^{1\times n}$). In a similar vein, we have the algebra
	\[
	\Cd(\K^{n\times 1}) \cong  \BC \left\langle y^\iota_i, \partial_{y^\iota_i} \right\rangle \otimes_\BC \BC \left\langle y^{\bar\iota}_i, \partial_{y^{\bar\iota}_i}\right\rangle, 
	\]
	of all polynomial coefficient differential operators on $\K^{n\times 1}\otimes_\BR\BC\cong\C_{\iota}^{n\times 1}\times \C_{\overline{\iota}}^{n\times 1}$ which acts on $\CS(\K^{n\times 1})$. Here the coordinates of $\C_{\iota}^{n\times 1}$ (resp. $\C_{\overline{\iota}}^{n\times 1}$) are denoted by $y_1^{\iota},\dots,y_n^{\iota}$ (resp. $y_1^{\overline{\iota}},\dots,y_n^{\overline{\iota}}$). We have a $\GL_n(\K\otimes_\BR\BC)$-isomorphism 
	\be \label{FTD}
	\CF_{\psi_\K}: \Cd(\K^{1\times n}) \to \Cd(\K^{n\times 1})
	\ee
	induced by 
	\begin{equation}\label{fourier}
		x^{\iota'}_i \mapsto -\frac{\partial_{y^{\iota'}_i}}{ \varepsilon_{\psi_{\R}}\cdot 2\pi {\rm i}},\quad  \partial_{x^{\iota'}_i} \mapsto  \varepsilon_{\psi_\R}\cdot 2\pi {\rm i}\cdot y^{\iota'}_i, \quad \iota'\in\CE_\K, \ i=1, 2,\ldots, n.
	\end{equation}
	It is straightforward to verify the following result. 
	
	\begin{lemp} \label{lem:F}
		For $D\in \Cd(\K^{1\times n})$ and $\Phi\in \CS(\K^{1\times n})$, it holds that
		\[
		\CF_{\psi_\K}(D\,\Phi) = \CF_{\psi_\K}(D) \CF_{\psi_\K}(\Phi)\in \CS(\K^{n\times 1}).
		\]
	\end{lemp}
	
	Recall that $\Sym_{\overline{\iota}}^{-a_{\eta,\chi, \bar\iota}}$ is the space of homogeneous differential operators on $\C_{\overline{\iota}}^{1\times n}$ of degree $a_{\eta,\chi,\overline{\iota}}$. Denote by $\Sym_{\overline{\iota}}^{*, a_{\eta,\chi, \bar\iota}}$  the space of homogeneous polynomials in $y^{\bar\iota}_1,\ldots, y^{\bar\iota}_n$ of degree $a_{\eta,\chi,\bar\iota}$.
	The map  \eqref{FTD} restricts to a $\GL_n(\BC)$-isomorphism  
	\be \label{partial-F} 
	\CF_{\psi_\K}: \Sym_{\overline{\iota}}^{-a_{\eta,\chi,\bar\iota}} \to \Sym_{\overline{\iota}}^{*, a_{\eta,\chi,\bar\iota}}.
	\ee
	
	We  factorize 
	\[
	\begin{aligned}
		F_{\eta, \chi}\otimes F_{\eta_0, \chi_0}^\vee & = F_{\chi\chi_0^{-1}}\otimes\mathfrak{D}_{\eta,\chi}\\
        & = ({\det}_{\iota}^{\chi_{\iota}} \otimes \Sym_{\iota}^{a_{\eta,\chi, \iota}})\otimes ({\det}_{\overline{\iota}}^{\chi_{\overline{\iota}}-1}\otimes \Sym_{\overline{\iota}}^{-a_{\eta,\chi,\bar\iota}}) \\
        &= : \Cd_{\eta,\chi,\iota} \otimes \Cd_{\eta,\chi, \bar\iota},
	\end{aligned}
	\]
	and put
	\[
	\Cd_{\eta,\chi, \bar\iota}^* := {\det}_{\overline{\iota}}^{\chi_{\overline{\iota}}-1}\otimes  \Sym_{\overline{\iota}}^{*, a_{\eta,\chi,\bar\iota}}.
	\]
	Using Lemma \ref{lem:F} and Fourier inversion, we obtain the factorization of $\beta_{\eta,\chi}$ as
	\begin{eqnarray}
		\beta_{\eta,\chi}:	\label{alpha-fac} F_{\eta, \chi}\otimes F_{\eta_0, \chi_0}^\vee \otimes \omega_{n,\chi_0} & =  &  (\Cd_{\eta,\chi, \iota} \otimes \Cd_{\eta,\chi, \bar\iota})\otimes \omega_{n,\chi_0} \\
		\nonumber & \longrightarrow  & \Cd_{\eta,\chi,\bar\iota} \otimes \omega_{n,\chi_0} \otimes ({\det}_{\iota}^{\chi_{\iota}}\otimes 1) \\
		\nonumber & \xrightarrow{\CF_{\psi_\K}}  & \Cd_{\eta,\chi, \bar\iota}^* \otimes \check{\omega}_{n,\chi_0}\otimes ({\det}_{\iota}^{\chi_{\iota}}\otimes 1)\\
		\nonumber & \longrightarrow &  \check{\omega}_{n,\chi} \\ 
		\nonumber  & \xrightarrow{\CF_{\overline{\psi_\K}}} & \omega_{n,\chi},
	\end{eqnarray}
	where the first and third arrows are multiplications by polynomials on $\CS(\K^{1\times n})$ and $\CS(\K^{n\times 1})$ respectively.

	\subsection{The balanced coefficient systems}
	\label{sec:balanced}
	
	Recall that $n'=n$ or $n'=n-1$. Take $\pi_{\mu}$ and $\pi_{\nu}$ as in the Introduction for dominant weights $\mu\in(\Z^n)^{\mathcal{E}_{\K}}$ and $\nu\in(\Z^{n'})^{\mathcal{E}_{\K}}$. Let $\eta$ be the product of central characters of $\pi_{\mu}$ and $\pi_{\nu}$. Recall  Definition \ref{def:balance} for an algebraic character $\chi$ of $\K^{\times}$ being $\xi$-balanced.
	
	The conditions for $\chi$ being $\xi$-balanced are given in the following lemma. When $n'=n-1$, this is a consequence of the classical branching law \cite[Theorem 8.1.1]{GTM255}. See also the discussions in \cite[Section 2.4.5.2]{Rag2}. When $n'=n$, this is a consequence of Pieri's rule \cite[Corollary 9.2.4]{GTM255}. See also the discussions in \cite[Proposition 3.3]{DX}.

	\begin{lemp} \label{lem-balanced}
		An algebraic character $\chi$ of $\K^{\times}$ is $\xi$-balanced if and only if
		\begin{itemize}
			\item \emph{(Case $n'=n-1$)} For every $\iota'\in\mathcal{E}_{\K}$,
			\[
				\max_{2\leq i\leq n}\{\mu_i^{\iota'}+\nu_{n+1-i}^{\iota'}\}+\chi_{\iota'}\leq 0\leq	\min_{1\leq i\leq n-1}\{\mu_i^{\iota'}+\nu_{n-i}^{\iota'}\}+\chi_{\iota'}.
			\]
			\item \emph{(Case ($-$))} For every $\iota'\in\mathcal{E}_{\K}$,
			\[
			\max_{1\leq i\leq n}\{\mu_i^{\iota'}+\nu_{n+1-i}^{\iota'}\}+\chi_{\iota'}\leq 0\leq \min_{1\leq i\leq n-1}\{\mu_i^{\iota'}+\nu_{n-i}^{\iota'}\}+\chi_{\iota'}.
			\]
			\item \emph{(Case ($+$))} For every $\iota'\in\mathcal{E}_{\K}$
			\[
			\max_{2\leq i\leq n}\{\mu_i^{\iota'}+\nu_{n+2-i}^{\iota'}\}+\chi_{\iota'}\leq 1\leq \min_{1\leq i\leq n}\{\mu_i^{\iota'}+\nu_{n+1-i}^{\iota'}\}+\chi_{\iota'}.
			\]
			\item \emph{(Case ($\pm$))}
			\[
			\begin{aligned}
				& \max_{1\leq i\leq n}\{\mu_i^{\iota}+\nu_{n+1-i}^{\iota}\}+\chi_{\iota}\leq 0\leq \min_{1\leq i\leq n-1}\{\mu_i^{\iota}+\nu_{n-i}^{\iota}\}+\chi_{\iota},\\
				\emph{and}\quad & \max_{2\leq i\leq n}\{\mu_i^{\overline{\iota}}+\nu_{n+2-i}^{\overline{\iota}}\}+\chi_{\overline{\iota}}\leq 1\leq \min_{1\leq i\leq n}\{\mu_i^{\overline{\iota}}+\nu_{n+1-i}^{\overline{\iota}}\}+\chi_{\overline{\iota}}.
			\end{aligned}
			\]
		\end{itemize}
		In all the above cases, we have
		\[
		\begin{aligned}
			&\dim \mathrm{Hom}_{\GL_{n'}(\K\otimes_{\R}\C)}\left(F_{\mu}^{\vee}\otimes F_{\nu}^{\vee}\otimes F_{\chi}^{\vee},\C\right)=1,&\quad&\text{when }n'=n-1,\\
			&\dim \mathrm{Hom}_{\GL_{n'}(\K\otimes_{\R}\C)}\left(F_{\mu}^{\vee}\otimes F_{\nu}^{\vee}\otimes F_{\eta,\chi}^{\vee},\C\right)=1,&\quad&\text{when }n'=n.
		\end{aligned}
		\]
	\end{lemp}

	Though no critical condition will be required in this article, we would like to discuss the relation between balanced and critical characters for completeness. 
	
	We say an algebraic character $\chi$ of $\K^{\times}$ is $\xi$-critical if
	\begin{itemize}
		\item ($n'=n-1$) $s=\frac{1}{2}$ is not a pole of $\RL(s,\pi_{\mu}\times\pi_{\nu}\times\chi)$ or $\RL(s,\pi_{\mu}^{\vee}\times\pi^{\vee}_{\nu}\times\chi^{-1})$.
		\item ($n'=n$) $s=0$ is not a pole of $\RL(s,\pi_{\mu}\times\pi_{\nu}\times\chi)$ or $\RL(1-s,\pi_{\mu}^{\vee}\times\pi^{\vee}_{\nu}\times\chi^{-1})$.
	\end{itemize}
	Denote the set of all $\xi$-critical characters by $\mathrm{Crit}(\xi)$. We further denote
	\[
	\begin{aligned}
		\RB(\xi)^{\natural}&:=\set{\chi\in\RB(\xi) | \chi_{\iota}=\chi_{\overline{\iota}}},\\
		\mathrm{Crit}(\xi)^{\natural}&:=\set{\chi\in\mathrm{Crit}(\xi) | \chi_{\iota}=\chi_{\overline{\iota}}}.
	\end{aligned}
	\]
	The rest of this section is devoted to proving the following lemma.
	
	\begin{lemp} \label{lem-bc}\
\begin{itemize}
    \item[(a)] For \emph{Case $(-)$} and \emph{Case $(+)$} we have $\RB(\xi)\cap\mathrm{Crit}(\xi)=\emptyset$.
    \item[(b)] For \emph{Case ($\pm$)} and the \emph{Case} $n'=n-1$ we have $\RB(\xi)\subset\mathrm{Crit}(\xi)$. Moreover, $\RB(\xi)^{\natural}=\mathrm{Crit}(\xi)^{\natural}$ whenever $\RB(\xi)^{\natural}\neq\emptyset$.
\end{itemize}
	\end{lemp}
	
	\subsubsection{Archimedean L-factors}\label{Lfactor}
	
	Before the proof, we first recall some standard facts about archimedean local L-factors following \cite{K}.  
	
	If $\K\cong\R$, then the following formulas hold.
	\begin{itemize}
		\item For all $t\in\C$ and $\delta\in\{0,1\}$,
		\[
		\RL(s,|\cdot|_{\K}^t\mathrm{sgn}_{\K^{\times}}^{\delta})=\Gamma_{\R}(s+t+\delta).
		\]
		\item For all $a,b\in\C$ with $a-b\in\Z\backslash\{0\}$, and $t,\delta$ as above,
		\[
		\RL(s,D_{a,b}\times|\cdot|_{\K}^t\mathrm{sgn}^{\delta}_{\K^{\times}})=\Gamma_{\C}(s+t+\max\{a,b\}).
		\]
		Here and henceforth, for any $a',b'\in\C$ with $a'-b'\in\Z$,
		\[
		\max\{a',b'\}:=\begin{cases}
			a', & a'-b'\geq 0;\\
			b', & \text{otherwise}.
		\end{cases}
		\]
		\item For all $a,b,a',b'\in\C$ with $a-b,a'-b'\in\Z\backslash\{0\}$,
		\[
		\RL(s,D_{a,b}\times D_{a',b'})=\Gamma_{\C}(s+\max\{a+a',b+b'\})\Gamma_{\C}(s+\max\{a+b',b+a'\}).
		\]
	\end{itemize}
	
	If $\K\cong\C$, then for all $a,b\in\C$ with $a-b\in\Z$,
	\[
	\RL(s,\iota^a\overline{\iota}^b)=\Gamma_{\C}(s+\max\{a,b\}).
	\]
	
	\subsubsection{\emph{Case ($-$) and Case ($+$)}}
	
	We are going to show that $\RB(\xi)\cap\mathrm{Crit}(\xi)=\emptyset$ for Case ($-$) when $\K\cong\C$ and leave aside other cases which can be treated similarly. Take a critical character $\chi\in\mathrm{Crit}(\xi)$. Then $s=0$ is not a pole of the following product of gamma factors
	\[
	\prod_{1\leq i\leq n}\Gamma_{\C}\left(s+\max\{\chi_{\iota}+\widetilde{\mu}_i^{\iota}+\widetilde{\nu}_{n+1-i}^{\iota},\chi_{\overline{\iota}}+\widetilde{\mu}_{n+1-i}^{\overline{\iota}}+\widetilde{\nu}_{i}^{\overline{\iota}}\}\right).
	\]
	This implies that
	\[
		\chi_{\iota}+\max_{1\leq i\leq n}\{\mu_i^{\iota}+\nu_{n+1-i}^{\iota}\}>0\quad \text{or}\quad	\chi_{\overline{\iota}}+\max_{1\leq i\leq n}\{\mu_{i}^{\overline{\iota}}+\nu_{n+1-i}^{\overline{\iota}}\}>0,
	\]
	which contradicts the balanced condition in Lemma \ref{lem-balanced}.

	\subsubsection{\emph{Case ($\pm$)}}

	We calculate the L-function
	\[
	\begin{aligned}
		&\RL(s,\pi_{\mu}\times\pi_{\nu}\times\chi)\\
        =&\prod_{i, k =1,2,\dots, n}\Gamma_{\C}\left(s+\max\{\chi_{\iota}+\widetilde{\mu}_i^{\iota}+\widetilde{\nu}_k^{\iota},\chi_{\overline{\iota}}+\widetilde{\mu}_{n+1-i}^{\overline{\iota}}+\widetilde{\nu}_{n+1-k}^{\overline{\iota}}\}\right),\\
		&\RL(1-s,\pi_{\mu}^{\vee}\times\pi_{\nu}^{\vee}\times\chi^{-1})\\
        =&\prod_{i, k=1,2,\dots, n}\Gamma_{\C}\left(1-s-\min\{\chi_{\iota}+\widetilde{\mu}_i^{\iota}+\widetilde{\nu}_k^{\iota},\chi_{\overline{\iota}}+\widetilde{\mu}_{n+1-i}^{\overline{\iota}}+\widetilde{\nu}_{n+1-k}^{\overline{\iota}}\}\right).
	\end{aligned}
	\]
	Then $\chi\in\mathrm{Crit}(\xi)$ if and only if
	\begin{equation}\label{critical}
		\begin{aligned}
			\max\{\chi_{\iota}+\widetilde{\mu}_i^{\iota}+\widetilde{\nu}_k^{\iota},\chi_{\overline{\iota}}+\widetilde{\mu}_{n+1-i}^{\overline{\iota}}+\widetilde{\nu}_{n+1-k}^{\overline{\iota}}\}&\geq 1,\\
			\min\{\chi_{\iota}+\widetilde{\mu}_i^{\iota}+\widetilde{\nu}_k^{\iota},\chi_{\overline{\iota}}+\widetilde{\mu}_{n+1-i}^{\overline{\iota}}+\widetilde{\nu}_{n+1-k}^{\overline{\iota}}\}&\leq 0
		\end{aligned}
	\end{equation}
	for all $i, k =1,2,\dots, n$. The assertion that $\RB(\xi)\subset\mathrm{Crit}(\xi)$ can be checked straightforwardly by Lemma \ref{lem-balanced}.

	Assume $\RB(\xi)^{\natural}\neq\emptyset$. By Lemma \ref{lem-balanced} this implies
	\begin{equation} \label{lem3.11}
		\begin{aligned}
			\widetilde{\mu}_i^{\iota}+\widetilde{\nu}_k^{\iota}&>\widetilde{\mu}_{n+1-i}^{\overline{\iota}}+\widetilde{\nu}_{n+1-k}^{\overline{\iota}},&\qquad&i+k\leq n,\\
			\widetilde{\mu}_i^{\iota}+\widetilde{\nu}_k^{\iota}&<\widetilde{\mu}_{n+1-i}^{\overline{\iota}}+\widetilde{\nu}_{n+1-k}^{\overline{\iota}},&\qquad&i+k> n.
		\end{aligned}
	\end{equation}
	Then for $\chi\in\mathrm{Crit}(\xi)^{\natural}$, \eqref{critical} is equivalent to the balanced condition in Lemma \ref{lem-balanced} and thus $\RB(\xi)^{\natural}=\mathrm{Crit}(\xi)^{\natural}$. This proves the assertion for Case ($\pm$).

	We  also remark that $\RB(\xi)\neq\mathrm{Crit}(\xi)$ in general. For example, assume $n\geq 2$ and consider
	\[
	\mu:=(0,\dots,0;0,\dots,0),\qquad \nu:=(0,\dots,0;1,\dots,1).
	\]
	Then 
    \[
    \RB(\xi)=\set{\mathbf{1},\iota\overline{\iota}^{-1}},
    \]
    while
	\[
    \begin{aligned}
        \mathrm{Crit}(\xi)=\set{\iota^{\chi_{\iota}}\overline{\iota}^{\chi_{\overline{\iota}}} :\chi_{\iota}+\chi_{\overline{\iota}}=0}&\cup\set{\iota^{\chi_{\iota}}\overline{\iota}^{\chi_{\overline{\iota}}} :\chi_{\iota}\leq 1-n,\chi_{\overline{\iota}}\geq n-1}\\
        &\cup\set{\iota^{\chi_{\iota}}\overline{\iota}^{\chi_{\overline{\iota}}} :\chi_{\iota}\geq n,\chi_{\overline{\iota}}\leq -n}.
    \end{aligned}
	\]

	\subsubsection{$\GL_n\times\GL_{n-1}$}
	When $\K\cong\C$, we calculate the L-function
	\[
	\begin{aligned}
		&\RL(s,\pi_{\mu}\times\pi_{\nu}\times\chi)\\
        =&\prod_{\substack{i=1,2,...,n\\k=1,2,...,n-1}}\Gamma_{\C}\left(s+\max\{\chi_{\iota}+\widetilde{\mu}_i^{\iota}+\widetilde{\nu}_k^{\iota},\chi_{\overline{\iota}}+\widetilde{\mu}_{n+1-i}^{\overline{\iota}}+\widetilde{\nu}_{n-k}^{\overline{\iota}}\}\right),\\
		&\RL(1-s,\pi_{\mu}^{\vee}\times\pi_{\nu}^{\vee}\times\chi^{-1})\\
        =&\prod_{\substack{i=1,2,...,n\\k=1,2,...,n-1}}\Gamma_{\C}\left(1-s-\min\{\chi_{\iota}+\widetilde{\mu}_i^{\iota}+\widetilde{\nu}_k^{\iota},\chi_{\overline{\iota}}+\widetilde{\mu}_{n+1-i}^{\overline{\iota}}+\widetilde{\nu}_{n-k}^{\overline{\iota}}\}\right).
	\end{aligned}
	\]
	Then $\chi\in\mathrm{Crit}(\xi)$ if and only if
	\begin{equation}\label{critical'}
		\begin{aligned}
			\max\{\chi_{\iota}+\widetilde{\mu}_i^{\iota}+\widetilde{\nu}_k^{\iota},\chi_{\overline{\iota}}+\widetilde{\mu}_{n+1-i}^{\overline{\iota}}+\widetilde{\nu}_{n-k}^{\overline{\iota}}\}&\geq \frac{1}{2},\\
			\min\{\chi_{\iota}+\widetilde{\mu}_i^{\iota}+\widetilde{\nu}_k^{\iota},\chi_{\overline{\iota}}+\widetilde{\mu}_{n+1-i}^{\overline{\iota}}+\widetilde{\nu}_{n-k}^{\overline{\iota}}\}&\leq -\frac{1}{2},
		\end{aligned}
	\end{equation}
	for all $1\leq i\leq n$, $1\leq k\leq n-1$. The first assertion can be checked straightforwardly by Lemma \ref{lem-balanced}. The case $\K\cong\R$ is similar.
	
	Assume $\RB(\xi)^{\natural}\neq\emptyset$. Then Lemma \ref{lem-balanced} implies that
	\[
	\widetilde{\mu}_i^{\iota}+\widetilde{\nu}_k^{\iota}-\widetilde{\mu}_{n+1-i}^{\overline{\iota}}-\widetilde{\nu}_{n-k}^{\overline{\iota}}
	\]
	is positive if $i+k\leq n$ and is negative otherwise for every $\iota\in\mathcal{E}_{\K}$. Then for $\chi\in\mathrm{Crit}(\xi)^{\natural}$, \eqref{critical'} is equivalent to the balanced condition in Lemma \ref{lem-balanced} and thus $\RB(\xi)^{\natural}=\mathrm{Crit}(\xi)^{\natural}$. This completes the proof of Lemma \ref{lem-bc}.

	\section{Proof of archimedean period relations}
	
	In this section, we prove the archimedean period relations (Theorem \ref{mainthm}) for $\GL_n\times\GL_n$. Let $\pi_{\mu}\in\Omega(\mu)$, $\pi_{\nu}\in\Omega(\nu)$ for dominant weights $\mu\in(\Z^n)^{\mathcal{E}_{\K}}$, $\nu\in(\Z^n)^{\mathcal{E}_{\K}}$ and $\eta$ the product of central characters of $\pi_{\mu}$, $\pi_{\nu}$. Set $\xi:=(\mu,\nu)$, $\pi_{\xi}:=\pi_{\mu}\,\widehat{\otimes}\,\pi_{\nu}$ and take a $\xi$-balanced character $\chi\in\RB(\xi)$. We write $\pi_{\xi,\chi}:=\pi_{\xi}\,\widehat{\otimes}\, \omega_{n,\chi}$ and the coefficient system $F_{\xi,\chi}:=F_{\mu}\otimes F_{\nu}\otimes F_{\eta,\chi}$ for short. Recall from \eqref{phixichi} and \eqref{jxichi} (see also Section \ref{sec:TP} and Section \ref{sec:Translation}) the balanced map and the translation map. We define a composition map
	\[
	\gamma_{\xi,\chi}:=(\check{\jmath}_{\mu}\otimes\check{\jmath}_{\nu}'\otimes\beta_{\eta,\chi}^{\circ})\circ(\phi_{\xi,\chi}\otimes 1\otimes\cdot):\pi_{\xi_0,\chi_0}\to F_{\xi,\chi}\otimes F_{\xi_0,\chi_0}^{\vee}\otimes\pi_{\xi_0,\chi_0}\to\pi_{\xi,\chi}.
	\]
    The same notation applies to $\xi_0:=(0_n,0_n)$. We will prove the following result.
	
	\begin{thmp}\label{thmap2}
		The diagram 
		\[
		\begin{CD}
			\pi_{\xi_0,\chi_0}
			@>   \oZ_{\xi_0,\chi_0}^{\Theta} >>    \mathfrak{M}^*_{n,\K} \\
			@ V \gamma_{\xi, \chi}  V  V        @ | \\
			\pi_{\xi,\chi} @> \Omega_{\xi,\chi}\cdot \oZ_{\xi,\chi}^{\Theta}  >>   \mathfrak{M}^*_{n,\K}\\
		\end{CD}
		\]
		commutes, where $\Omega_{\xi,\chi}$ is given as in Theorem \ref{mainthm}.
	\end{thmp}

\subsection{Proof of the main theorem}

We explain how to deduce the main Theorem \ref{mainthm} from Theorem \ref{thmap2}. 
    Recall the normalized Rankin-Selberg integrals
\[
\RZ_{\xi,\chi'}^{\circ}:=\frac{\RL(0,\eta\chi'^n)\cdot\RZ_{\xi,\chi'}}{\RL(0,\pi_{\mu}\times\pi_{\nu}\times\chi')}\quad \text{and}\quad \RZ_{\xi,\chi'}^{\Theta}:=\frac{\RZ'_{\xi,\chi'}}{\RL(0,\pi_{\mu}\times\pi_{\nu}\times\chi')},
\]
where  \[
\RZ_{\xi,\chi'}:\pi_{\xi}\,\widehat{\otimes}\,I_{\eta,\chi'}\to\mathfrak{M}_{n,\K}^{\ast}\quad \text{and}\quad \RZ'_{\xi,\chi'}:\pi_{\xi}\,\widehat{\otimes}\,\omega_{n,\chi'}\to\mathfrak{M}_{n,\K}^{\ast}
    \]
are defined by \eqref{rsintegral} and \eqref{integraltheta} respectively.
 Clearly, the diagram 
  \begin{equation}\label{thetadiagram}
		\begin{CD}
			\pi_{\xi}\,\widehat{\otimes}\,\omega_{n,\chi'} @>\RZ^{\Theta}_{\xi,\chi'}>>\mathfrak{M}_{n,\K}^{\ast}\\
			@V\mathrm{f}_{\eta,\chi'}VV @|\\
			\pi_{\xi}\,\widehat{\otimes}\,I_{\eta,\chi'} @>\RZ^{\circ}_{\xi,\chi'}>>\mathfrak{M}_{n,\K}^{\ast}
		\end{CD}
	\end{equation}
commutes by the definitions of $\RZ_{\xi,\chi'}$, $\RZ'_{\xi,\chi'}$ and $\mathrm{f}_{\eta,\chi'}$. Recall from \eqref{imf} that
$
\tau_{\eta,\chi'}={\rm im}({\rm f}_{\eta,\chi'})\subset I_{\eta,\chi'}.
$

\begin{prp} \label{RSnn}
When $\chi\in\RB(\xi)$ is balanced, the normalized Rankin-Selberg integral $\RZ_{\xi,\chi'}^{\circ}$ is holomorphic at $\chi'=\chi$ and  induces a linear functional
\[
\RZ_{\xi,\chi}^{\circ}\in\mathrm{Hom}_{\GL_n(\K)}\left(\pi_{\xi}\,\widehat{\otimes}\,I_{\eta,\chi},\mathfrak{M}_{n,\K}^{\ast}\right)
	\]
    which is nonzero on $\pi_{\xi}\,\widehat{\otimes}\,\tau_{\eta,\chi}$.
\end{prp}

\begin{proof}
Clearly, the integral $\RZ_{\xi,\chi}^{\Theta}$ factors through the quotient
\[
\pi_{\xi}\,\widehat{\otimes}\,\omega_{n,\chi}\twoheadrightarrow\pi_{\xi}\,\widehat{\otimes}\,\Theta_n(\eta\chi^n)\otimes(\chi\circ\det).
\]
 By Proposition \ref{thmX}, when $s=0$ is not a pole of $\RL(s,\eta\chi^n)$,
 \[
 \mathrm{f}_{\eta,\chi}:\Theta_n(\eta\chi^n)\otimes(\chi\circ\det)\to I_{\eta,\chi}
 \]
is an isomorphism and the proposition is obvious in this case.  
   
Suppose that $s=0$ is a pole of $\RL(s,\eta\chi^n)$, so that we are in Case ($-$). Since this pole is simple, to show that 
$\oZ^\circ_{\xi,\chi'}$ is holomorphic at $\chi'=\chi$, it suffices to show that 
\[
\frac{\RZ_{\xi,\chi}}{\RL(0,\pi_{\mu}\times\pi_{\nu}\times\chi)}
\]
vanishes on $\pi_{\xi}\,\widehat{\otimes}\,I_{\eta,\chi}$. In view of the isomorphism ${\rm f}_{\eta,\chi}^\circ$ in \eqref{fcirc}, this is  equivalent to the assertion that $\RZ^{\Theta}_{\xi,\chi}$ vanishes on $\pi_{\xi}\,\widehat{\otimes}\,\omega_{n,\chi}^{\circ}$.
Since  $\mathrm{coker}({\rm j}_{\eta,\chi})\cong \tau_{\eta,\chi}$ by Proposition \ref{thmX}, using  the multiplicity one theorem \cite[Theorem C]{SZ} it suffices to show that
\begin{equation}\label{dim1}
\mathrm{Hom}_{\GL_n(\K)}(\pi_{\mu}\,\widehat{\otimes}\,\pi_{\nu}\otimes\tau_{\eta,\chi},\C)\neq \{0\},
\end{equation}
in which case it has dimension one. 

If $\sigma$ is an irreducible representation of $\GL_n(\K)$ that occurs as a quotient of $\pi_\mu \otimes \tau_{\eta,\chi}$, then there is a nonzero $\GL_n(\K)$-homomorphism $\pi_\mu \to \sigma \otimes \tau_{\eta,\chi}^\vee$, which implies that $\sigma$ has to be generic by considering the Gelfand-Kirillov dimensions \cite{Vo78}. 

Recall from \cite[Lemma 2.2]{LLS24} that $\pi_{\mu}$ is the unique irreducible quotient of 
\[
I_{\mu}=\mathrm{Ind}^{\GL_n(\K)}_{\overline{\RB}_n(\K)}(\chi_{\mu}\cdot\rho_n\cdot \varepsilon_{\pi_{\mu}}').
\]
Thus by the same argument as above, all irreducible quotients of $I_\mu\otimes \tau_{\eta,\chi}$ must be generic as well. 
Note that $\tau_{\eta,\chi}$ is a twist of $F_{\eta,\chi}$ by a quadratic character such that the central character of $\tau_{\eta,\chi}$ is $\eta^{-1}$.
In particular, $I_\mu\otimes \tau_{\eta,\chi}$, its quotient $\pi_\mu\otimes\tau_{\eta,\chi}$ and $\pi_\nu^\vee$ have the same central character. Denote by $\chi_{\pi_\nu^\vee}$ the infinitesimal character of $\pi_\nu^\vee$.  From the balanced condition in Lemma \ref{lem-balanced} and
\[
I_\mu \otimes \tau_{\eta,\chi}\cong \mathrm{Ind}^{\GL_n(\K)}_{\overline{\RB}_n(\K)}\left(\chi_{\mu}\cdot\rho_n\cdot \varepsilon_{\pi_{\mu}}'
\otimes \tau_{\eta,\chi}\right),
\]
it is easy to verify that the direct summand $(I_\mu\otimes \tau_{\eta,\chi})_{\chi_{\pi_\nu^\vee}}$ of $I_\mu\otimes \tau_{\eta,\chi}$ with generalized infinitesimal character 
$\chi_{\pi_\nu^\vee}$ is isomorphic to 
\[
I_{\check\nu}:=\mathrm{Ind}^{\GL_n(\K)}_{\overline{\RB}_n(\K)}(\chi_{\check\nu}\cdot\rho_n\cdot \varepsilon_{\pi_{\nu}}'),
\]
where $\check\nu:=(-\nu^{\iota'}_n, -\nu^{\iota'}_{n-1},\dots, -\nu^{\iota'}_1)_{\iota'\in\CE_\K}$. By \cite[Lemma 2.2]{LLS24} again, $\pi_\nu^\vee$ is the unique generic constituent of $I_{\check\nu}$ and the unique irreducible 
quotient of $I_{\check\nu}$.

Since $\pi_{\mu}\otimes\tau_{\eta,\chi}$ and $I_{\mu}\otimes\tau_{\eta,\chi}$ have the same generic constituents, we see that 
$(\pi_{\mu}\otimes\tau_{\eta,\chi})_{\chi_{\pi_\nu^\vee}}$ is nonzero and has a unique generic irreducible quotient which is isomorphic to $\pi_\nu^\vee$. It follows that
\[
\Hom_{\GL_n(\K)}(\pi_\mu\otimes \tau_{\eta,\chi}, \pi_\nu^\vee)\neq \{0\},
\]
which implies \eqref{dim1}.
\end{proof}

Putting the diagrams in Theorem \ref{thmap2}, \eqref{thetadiagram} and \eqref{RIdiagram} together, we obtain the commutative diagram
\[
\xymatrix{
		\pi_{\xi_0}\,\widehat{\otimes}\, \omega_{n,\chi_0}  \ar[dd]_{ \gamma_{\xi, \chi}}  \ar[rd]^(0.6){\mathrm{f}_{\eta_0,\chi_0}} \ar[rrr]^{  \oZ_{\xi_0,\chi_0}^\Theta}  
		&&&  \frak{M}^*_{n,\K}  \ar@{=}[rd]  \ar@{=}[dd]|!{[dll];[dr]}\hole \\
		&  \pi_{\xi_0}\,\widehat{\otimes}\,I_{\eta_0,\chi_0}  \ar[dd]_(0.4){\gamma_{\xi,\chi}}  \ar[rrr]^(0.4){  \oZ_{\xi_0,\chi_0}^{\circ}} 
		&&&   \frak{M}^*_{n,\K}  \ar@{=}[dd] \\
		\pi_{\xi}\,\widehat{\otimes}\, \omega_{n,\chi} \ar[rd]_{ \mathrm{f}_{\eta,\chi} } \ar[rrr]^(0.7){\Omega_{\xi,\chi}\cdot\oZ_{\xi,\chi}^\Theta} |!{[ur];[dr]}\hole
		&& &    \frak{M}^*_{n,\K} \ar@{=}[rd]  \\
		&  \pi_{\xi}\,\widehat{\otimes}\,I_{\eta,\chi}\ar[rrr]^{  \Omega_{\xi,\chi}\cdot\oZ_{\xi,\chi}^{\circ}} 
		&&& \frak{M}^*_{n,\K} \\
	}
\]
which implies Theorem \ref{mainthm} (a). Here for Case ($-$), note that
	\[
	\RH_{\mathrm{ct}}^0(\R_+^{\times}\backslash\GL_n(\K)^0;F_{\eta,\chi}^{\vee}\otimes I_{\eta,\chi})=\RH_{\mathrm{ct}}^0(\R_+^{\times}\backslash\GL_n(\K)^0;F_{\eta,\chi}^{\vee}\otimes \tau_{\eta,\chi}),
	\]
	where $\tau_{\eta,\chi}=\mathrm{im}(\mathrm{f}_{\eta,\chi})$ is the irreducible subrepresentation of $I_{\eta,\chi}$. To complete the proof of Theorem \ref{mainthm} (b), it remains to prove the following non-vanishing result for Case ($-$).

	\begin{prp}
		\label{prop:nonvanishing} For \emph{Case ($-$)},
		the archimedean modular symbol
		\[
		\wp_{\xi_0,\chi_0}:\RH_{\xi_0,\chi_0}\otimes\mathfrak{O}_{n,\K}\to\C
		\]
		 is non-vanishing when restricted to $\oH_{\xi_0,\chi_0}[\varepsilon]\otimes
            \frak O_{n,\K}$,
           for every character $\varepsilon=\varepsilon_1\otimes \varepsilon_2\otimes \varepsilon_3$ of $\pi_0(\K^\times)^3$ which occurs in 
           $\oH_{\xi_0,\chi_0}$ such that 
           $\varepsilon_1\cdot\varepsilon_2\cdot\varepsilon_3 = \sgn_{\K^\times}^{n-1}$.
	\end{prp}
	
	\begin{proof}
		This is essentially due to \cite[Proposition 5.0.1]{BR17} which can be proved by using the same strategy as \cite{DX, Sun}. We sketch the idea.
		
		Since $\pi_{0_n}$ is unitarizable, we have
		\[
		\begin{aligned}
			\RH_{\xi_0,\chi_0}=&\,\mathrm{Hom}_{\RK_{n,\K}^0}\left({\bigwedge}^{b_{n,\K}}(\mathfrak{gl}_{n,\K}/(\mathfrak{k}_{n,\K}\oplus\R))\otimes_{\R}\C,\pi_{0_n}\right)\\
			& \otimes\,\mathrm{Hom}_{\RK_{n,\K}^0}\left({\bigwedge}^{t_{n,\K}}(\mathfrak{gl}_{n,\K}/(\mathfrak{k}_{n,\K}\oplus\R))\otimes_{\R}\C,\pi_{0_n}\right).
		\end{aligned}
		\] 
		We need to prove the map
		\[
		\wp_{\xi_0,\chi_0}:\mathrm{Hom}_{\RK_{n,\K}^0}\left({\bigwedge}^{d_{n,n,\K}}(\mathfrak{gl}_{n,\K}/(\mathfrak{k}_{n,\K}\oplus\R))\otimes_{\R}\C,\pi_{0_n}\widehat{\otimes}\pi_{0_n}\right)\otimes\mathfrak{O}_{n,\K}\to\C
		\]
		is non-vanishing. It suffices to show that a basis element of 
		\[
		\mathrm{Hom}_{\GL_n(\K)}(\pi_{0_n}\,\widehat{\otimes}\,\pi_{0_n},\mathfrak{M}_{n,\K}^{\ast})
		\]
		does not vanish when restricted to certain minimal $K$-types of $\pi_{0_n}\,\widehat{\otimes}\,\pi_{0_n}$. Since 
		\[
		\dim\mathrm{Hom}_{\GL_n(\K)}(\pi_{0_n}\,\widehat{\otimes}\,\pi_{0_n},\mathfrak{M}_{n,\K}^{\ast})=1,
		\]
		this follows from the proof of \cite[Proposition 5.0.1]{BR17}.
	\end{proof}

	\subsection{Reduction to principal series representations}
	
	Recall that in Section \ref{sec:TP} we have defined a principal series representation  $I_\mu$ with a Whittaker functional $\wt{\lambda}_{\mu} \in \Hom_{\mathrm N_n(\K)}(I_\mu, \psi_{n, \K})$, and a unique 
	$
	p_\mu\in \Hom_{\GL_{n}(\K)}(I_\mu, \pi_\mu)
	$
	such that 
	\[
	\lambda_{\mu} \circ p_\mu =\wt{\lambda}_{\mu}.
	\] 
	We have also defined 
	$
	\check\imath_\mu\in \Hom_{\GL_n(\K)}(F_\mu\otimes I_{0_{n}},   I_\mu)
	$
	such that
	\[
	\wt{\lambda}_{\mu} \circ \check\imath_\mu \circ (v_\mu \otimes \cdot )  =  \wt{\lambda}_{0_{n}}
	\]
	and that the diagram \eqref{I-pi-dual} commutes. We have similar data for $\nu$. 
	
	For every $\chi'\in\widehat{\K^{\times}}$, put
	\[
	\begin{aligned}
		 \CI_{\xi, \chi'} &: =  I_\mu \, \widehat \otimes \, I_\nu  \,\widehat\otimes\,  \omega_{n,\chi'},  \\
		p_\xi&:=p_\mu \otimes p_\nu\otimes {\rm id}\in  \Hom_{\GL_n(\K)^3}( \CI_{\xi, \chi'},   \pi_{\xi,\chi'} ).
	\end{aligned}
	\]
	Define the normalized Rankin-Selberg integral
	\[
	\oZ_{\xi,\chi'}^{\diamond} \in \Hom_{\GL_{n}(\K)}\left( \CI_{\xi, \chi'},  \frak{M}_{n, \K}^*\right)
	\]
	as the composition 
	\[
	\CI_{\xi, \chi'}   \xrightarrow{p_\xi}  \pi_{\xi,\chi'}   \xrightarrow{ \oZ_{\xi,\chi'}^{\Theta}} \frak{M}_{n,\K}^*. 
	\]
	Then 
	\[
	\begin{aligned}
		&    \oZ^\diamond_{\xi,\chi'}( f, f',  \Phi;\mathrm{d}g)   =    \frac{1}{\oL(0, \pi_\mu\times\pi_\nu\times\chi')}    \\
		& \quad  \cdot \int_{\mathrm N_{n}(\K)\backslash \GL_{n}(\K)} \langle\wt{\lambda}_{\mu},g.f\rangle\langle\wt{\lambda}'_{\nu},g.f'\rangle \Phi(e_n g) \chi'(\det g) \overline{\mathrm{d}}g, 
	\end{aligned}
	\]
	for $f\in I_\mu$, $f' \in I_\nu$, $\Phi\in \CS(\K^{1\times n})$, $\mathrm{d}g \in \frak{M}_{n,\K}$ and $\chi'\in\widehat{\K^{\times}}$ with $\mathrm{Re}(\chi')$ sufficiently large. 
	
	Put
	\[
    \begin{aligned}
	\delta_{\xi, \chi} := (\check\imath_\mu \otimes \check\imath_\nu\otimes \beta^{\circ}_{\eta, \chi})\circ (\phi_{\xi, \chi}\otimes 1 \otimes \cdot) :  \CI_{\xi_0, \chi_0} \to F_{\xi,\chi}\otimes F_{\xi_0,\chi_0}^{\vee}\otimes \CI_{\xi_0, \chi_0}  \to  \CI_{\xi,\chi}.
    \end{aligned}
	\]

	In view of all the above, by the multiplicity one theorem \cite[Theorem C]{SZ}, there exists a unique constant $\Omega'_{\xi, \chi}\in \BC$  such that the diagram  
	\be \label{cube}
	\xymatrix{
		\CI_{\xi_0, \chi_0}  \ar[dd]_{ \delta_{\xi, \chi}}  \ar[rd]^(0.6){p_{\xi_0}} \ar[rrr]^{  \oZ_{\xi_0,\chi_0}^\diamond}  
		&&&  \frak{M}^*_{n,\K}  \ar@{=}[rd]  \ar^(0.7){\Omega_{\xi, \chi}' }[dd]|!{[dll];[dr]}\hole \\
		&  \pi_{\xi_0, \chi_0}  \ar[dd]_(0.4){\gamma_{\xi,\chi}}  \ar[rrr]^(0.4){  \oZ_{\xi_0,\chi_0}^{\Theta}} 
		&&&    \frak{M}^*_{n,\K}  \ar[dd]^{\Omega_{\xi, \chi}' } \\
		\CI_{\xi, \chi} \ar[rd]_{ p_\xi } \ar[rrr]^(0.7){\oZ_{\xi,\chi}^\diamond} |!{[ur];[dr]}\hole
		&& &    \frak{M}^*_{n,\K} \ar@{=}[rd]  \\
		&  \pi_{\xi, \chi}\ar[rrr]^{  \oZ_{\xi,\chi}^{\Theta}} 
		&&& \frak{M}^*_{n,\K} \\
	}
	\ee
	commutes.
	
	In the rest of this section we compute the constant $\Omega_{\xi, \chi}'$. The main ingredients of the computation are \cite{LLSS}  and the functional equation for Rankin-Selberg integrals in \cite{J09}.

	\subsection{Integral over the open orbit}
	
	\label{sec:lss}
	
	In this subsection we recall the main result of \cite{LLSS} for $\GL_n\times\GL_n$. 
	
	Let $\varrho=(\varrho_1,\ldots, \varrho_n) \in (\widehat{\K^\times})^n$, viewed as a character of $\bar{\mathrm B}_n(\K)$ as usual, and let
	\[
	I(\varrho) :=\Ind^{\GL_n(\K)}_{\bar{\rm B}_n(\K)} \varrho
	\]
	be the corresponding principal series representation of $\GL_n(\K)$. Let $\varrho' = (\varrho_1',\ldots, \varrho_{n}')\in (\widehat{\K^\times})^{n}$. We have a meromorphic family of unnormalized Rankin-Selberg integrals
	\[
	\oZ_{\chi'}
	\in \Hom_{\GL_{n}(\K)}\left(I(\varrho) \, \widehat{\otimes}\, I(\varrho') \, \widehat{\otimes}\, \omega_{n,\chi'}, \frak{M}_{n,\K}^*\right),
	\]
	such that
	\be \label{Zvarrho}
	\begin{aligned}
		&  \oZ_{\chi'}( f, f' ,\Phi;\mathrm{d}g)   \\
		= \ & \int_{\mathrm N_{n}(\K)\backslash \GL_{n}(\K)} \langle\lambda_{\varrho},g.f\rangle\langle \lambda'_{\varrho'},g.f' \rangle\Phi(e_n g) \chi'(\det g)\overline{\mathrm{d}}g, 
	\end{aligned}
	\ee
	for $f\in I(\varrho)$, $f' \in I(\varrho')$, $\Phi\in \CS(\K^{1\times n})$, $\mathrm{d}g \in \frak{M}_{n,\K}$ and $\chi'\in\widehat{\K^{\times}}$ with $\mathrm{Re}(\chi')$ sufficiently large, where $\lambda_{\varrho}\in \Hom_{\RN_n(\K)}(I(\varrho), \psi_{n,\K})$ and 
	$\lambda'_{\varrho'}\in \Hom_{\RN_{n}(\K)}(I(\varrho'), \overline{\psi_{n,\K}})$
	are defined in the same way as \eqref{la'}.
	
	Recall that
	\[
	z=\left(z_n,\left[\begin{array}{cc}
		z_{n-1} & 0\\
		0 & 1
	\end{array}\right],e_n\right)\in\GL_n(\Z)\times\GL_n(\Z)\times\Z^{1\times n},
	\]
	where $z_n\in\GL_n(\Z)$ is defined inductively in \eqref{zn}, and that the right action of $\GL_{n}(\K)$ on  $\CB_n(\K)\times \CB_n(\K)  \times \K^{1\times n} $ has a unique  open orbit 
	\be \label{openorbit}
	z.\GL_n(\K) = \big(\bar{\mathrm B}_n(\K) z_n\times \bar{\mathrm B}_{n}(\K)z_{n-1}\times \{e_n\}\big). \GL_{n}(\K),
	\ee
	where we have written $z$ for its image in $\CB_n(\K)\times\CB_n(\K)\times \K^{1\times n}$ by abuse of notation. 
	Note that  
	\[
	I(\varrho)\,\widehat{\otimes}\, I(\varrho') = \Ind^{\GL_n(\K)\times \GL_{n}(\K)}_{\bar{\rm B}_n(\K)\times \bar{\rm B}_{n}(\K)}(\varrho\otimes \varrho').
	\]
	Following \cite{LLSS}, we first formally define
	\[ 
	\Lambda_{\chi'}\in   \Hom_{\GL_{n}(\K)}\left(I(\varrho)\, \widehat \otimes \, I(\varrho')\, \widehat\otimes \, \omega_{n,\chi'},  \frak{M}_{n,\K}^*\right) 
	\]
	as the integral over the above open orbit. That is, 
	\be \label{lambda}
	\quad  \Lambda_{\chi'}(\Psi;\mathrm{d}g) 
	:=    \int_{\GL_{n}(\K)} \Psi(z.g)  \chi'(\det g) \od\! g
	\ee
	for $ \Psi\in  I(\varrho) \, \widehat\otimes \, I(\varrho') \, \widehat\otimes\, \omega_n$ and $\mathrm{d}g\in \frak{M}_{n,\K}$.
	
	Define
	\[
	\sgn(\varrho, \varrho' ,\chi'):=\prod_{i>k,  \, i+k\leq n} (\varrho_i \cdot \varrho'_k \cdot\chi')(-1),
	\]
	and a meromorphic function 
	\be \label{LSS:gamma}
	\gamma_{\psi_\K}(s, \varrho, \varrho',\chi'):=\prod_{i+k\leq n} \gamma(s, \varrho_i \cdot \varrho'_k\cdot\chi' , \psi_\K),
	\ee
	where 
	\[
	\gamma(s, \omega, \psi_\K) = \varepsilon(s,\omega, \psi_\K) \cdot \frac{\oL(1-s, \omega^{-1})}{\oL(s, \omega)}
	\]
	is the local gamma factor of a character $\omega \in \widehat{\K^\times}$, and $\varepsilon(s, \omega, \psi_\K)$ is the local epsilon factor\footnote{In 
		the convention of \cite{J09}, this would be $\varepsilon(s, \omega, \overline{\psi_\K}) = \omega(-1) \varepsilon(s, \omega, \psi_\K)$.}, defined following \cite{T, J1, Ku}.
	For later use, we also define 
	\[
	\varepsilon_{\psi_\K}(s, \varrho, \varrho',\chi'):=\prod_{i+k\leq n} \varepsilon(s, \varrho_i \cdot \varrho'_k\cdot\chi', \psi_\K).
	\]
	Finally define the meromorphic function 
	\be \label{LSS:Gamma}
	\Gamma_{\psi_\K}(\varrho, \varrho' ,\chi'):=\sgn(\varrho, \varrho',\chi' )  \cdot \gamma_{\psi_\K}(0, \varrho, \varrho',\chi' ).
	\ee
	
	Consider the complex manifold 
	\[
	\mathcal{M}:=(\widehat{\K^\times})^n \times (\widehat{\K^\times})^{n}\times\widehat{\K^\times}
	\]
	and its nonempty open subset 
	\[
	\Omega:=\Set{ (\varrho, \varrho',\chi') \in \mathcal{M} | \begin{aligned} & \mathrm{Re}(\varrho_i)+\mathrm{Re}(\varrho'_k)+\mathrm{Re}(\chi')<1  \textrm{ whenever } i+k\leq n,\\
		& \mathrm{Re}(\varrho_i)+\mathrm{Re}(\varrho'_k)+\mathrm{Re}(\chi')>0 \textrm{ whenever } i+k> n
	\end{aligned}}.
	\]

	\begin{thmp} \cite[Theorem 1.5 (a)]{LLSS} \label{thm:LSS}
		Assume that $(\varrho,\varrho',\chi')\in\Omega$. Then the integral \eqref{lambda} converges absolutely, and 
		\be \label{eq:LSS}
		\Lambda_{\chi'} ( \Psi;\mathrm{d}g)   = \Gamma_{\psi_\K}(\varrho, \varrho',\chi' )  \cdot \oZ_{\chi'}( \Psi;\mathrm{d}g). 
		\ee
	\end{thmp}

	Put
	\[
	\left(I(\varrho) \,\widehat\otimes\, I(\varrho') \, \widehat\otimes \, \omega_n\right)^\sharp:=\Set{\Psi\in  I(\varrho)\, \widehat\otimes \, I(\varrho') \, \widehat\otimes \, \omega_n | \Psi |_{z. \GL_{n}(\K)}\in
	\mathcal{S}(z. \GL_{n}(\K)) }.
	\]
	Then for every $\varrho\in (\widehat{\K^\times})^n$, $\varrho'\in (\widehat{\K^\times})^{n}$ and $\Psi \in  \left(I(\varrho) \,\widehat\otimes\, I(\varrho') \, \widehat\otimes \, \omega_n\right)^\sharp$, the integral \eqref{lambda} converges absolutely and is an entire function of $\chi'\in\widehat{\K^\times}$. We have the following consequence of Theorem \ref{thm:LSS}, which can be proved in the same way as \cite[Corollary 3.3]{LLS24}.
	
	\begin{corp} \label{cor:LSS}
		For every $\varrho\in (\widehat{\K^\times})^n$, $\varrho'\in (\widehat{\K^\times})^{n}$ and $\Psi \in  \left(I(\varrho) \widehat\otimes I(\varrho')\widehat\otimes\omega_n\right)^\sharp$, the equality \eqref{eq:LSS} holds as entire functions of $\chi'\in\widehat{\K^\times}$.
	\end{corp}
	
	\begin{proof} 
		Let $\CC$ and $\CC'$ be the connected components of $(\widehat{\K^\times})^n$ containing $\varrho$ and $\varrho'$ respectively.
		Write $\RK_\K:=\RK_{n,\K}\times \RK_{n,\K}$.  Define
		\[
		C^\infty_{\CC, \CC'}(\RK_\K):=\Set{ f \in C^\infty(\RK_\K) | \begin{array}{l} f(b\cdot k) = ( \varrho \otimes \varrho')(b) \cdot f(k), \\
			\textrm{for all } b\in \RK_\K\cap (\bar{\mathrm B}_n(\K) \times \bar{\mathrm B}_{n}(\K)), \ k\in \RK_\K
		\end{array}},
		\] 
		which only depends on $\CC$ and $\CC'$, not on the particular elements $\varrho \in \CC$ and $\varrho'\in \CC'$.

		Consider the natural  map 
		\[
		\RK_\K\times \K^{1\times n} \to (\bar{\mathrm B}_n(\K) \times \bar{\mathrm B}_{n}(\K) )\backslash (\GL_n(\K) \times \GL_{n}(\K))\times \K^{1\times n},
		\]
		which is surjective by the Iwasawa decomposition. Let $\CK_\K^\sharp \subset \RK_\K\times \K^{1\times n}$ be the pre-image of the open orbit \eqref{openorbit} under the above map. Fix $f \in C^\infty_{\CC, \CC'}(\RK_\K)\,\widehat\otimes \, \omega_n$ such that 
		$f|_{\CK_\K^\sharp}\in \CS(\CK_\K^\sharp)$. Then there is a unique  $\Psi_{\varrho,\varrho'}\in   \left(I(\varrho) \,\widehat\otimes\, I(\varrho') \, \widehat\otimes \, \omega_n\right)^\sharp$ such that 
		\[
		\Psi_{\varrho, \varrho'}|_{\RK_\K\times\K^{1\times n}} = f.
		\]
		
		Let $\CM^\circ$ be a connected component of $\CM$ contained in  $\CC\times \CC'\times\widehat{\K^\times}$.  The integral  
		$\Lambda_{\chi'}(\Psi_{\varrho,\varrho'};\mathrm{d}g)$ is  holomorphic when $(\varrho, \varrho',\chi')$ varies in $\CM^\circ$. 
		By \cite[Section 8.1]{J09}, we also have that
		\[
		\Gamma_{\psi_\K}( \varrho, \varrho',\chi' )  \cdot \oZ_{\chi'}( \Psi_{\varrho,\varrho'} ;\mathrm{d}g)  
		\]
		is meromorphic on $\CM^{\circ}$.  Since the equality 
		\[
		\Lambda_{\chi'}(\Psi_{\varrho,\varrho'};\mathrm{d}g)  =  \Gamma_{\psi_\K} ( \varrho, \varrho',\chi')  \cdot \oZ_{\chi'}( \Psi_{\varrho,\varrho'};\mathrm{d}g) 
		\]
		holds on $\Omega\cap \CM^\circ$, which is nonempty and open, it holds over all $\CM^\circ$ by the uniqueness of holomorphic continuation. 
	\end{proof}

	\subsection{Functional equation}
	
	We recall the functional equation for Rankin-Selberg integrals following \cite{J09}, and we will give a representation-theoretic interpretation.  For $\varrho=(\varrho_1, \ldots, \varrho_n) \in (\wh{\K^\times})^n$ as above, put
	\[
	\wh{\varrho}: = (\varrho_n^{-1},\ldots, \varrho_1^{-1}) \in (\wh{\K^\times})^n. 
	\]
	Recall that $g^\tau = {}^t g^{-1}$ for $g\in \GL_n(\K)$. For $f\in C^\infty(\GL_n(\K))$, define $\wt{f}\in C^\infty(\GL_n(\K))$ as
	\[
	\wt{f}(g):= f(w_n g^\tau). 
	\] 
	Then $\wt{f}\in I(\wh{\varrho})$ for any $f\in I(\varrho)$, and it is easy to check that
	\be \label{ftilde}
	\wt{g. f}= g^\tau.\wt{f},\quad g\in \GL_n(\K).
	\ee
	
	\begin{lemp} \label{lem:whit}
		For $f\in I(\varrho)$, define $W_f\in C^\infty(\GL_n(\K))$ by
		$
		W_f(g):= \lambda_{\varrho}(g.f). 
		$
		Then
		$
		\wt{W_f}(g) = \lambda'_{\wh\varrho}(g.\wt{f}).
		$
	\end{lemp}
	
	\begin{proof}
		We first prove the following  identity 
		\be \label{whit-id}
		\lambda'_{\wh\varrho}(w_n.\wt{f}) = \lambda_{\varrho}(f). 
		\ee
		Clearly it suffices to prove \eqref{whit-id} for $f\in I(\varrho)$ satisfying that $f\vert_{\RN_n(\K)}\in \CS(\RN_n(\K))$, in which case (see \eqref{la'})
		\[
		\lambda_{\varrho}(f) = \int_{\RN_n(\K)} f(u) \overline{\psi_{n,\K}}(u)\od\!u.
		\]
		Then $w_n.\wt{f} \vert_{\RN_n(\K)}\in \CS(\RN_n(\K))$, hence we find that 
		\[
        \begin{aligned}
		\lambda'_{\wh\varrho}(w_n.\wt{f}) &= \int_{\RN_n(\K)} f(w_n u^\tau  w_n) \psi_{n,\K}(u)\od\!u\\
        &= \int_{\RN_n(\K)} f(u) \overline{\psi_{n,\K}}(u)\od\!u = \lambda_{\varrho}(f).
        \end{aligned}
		\]
		This proves \eqref{whit-id}. Using this and \eqref{ftilde}, and noting that $w_n^\tau=w_n=w_n^{-1}$, we obtain that 
		\[
		\wt{W_f}(g) = \lambda_{\varrho}( w_n g^\tau.f) = \lambda'_{\wh\varrho}(g.\wt{f}).
		\]
	\end{proof}
	
	Recall that for $\Phi\in\CS(\K^{1\times n})$ we have the Fourier transform 
	\[
	\wh\Phi = \CF_{\psi_\K}(\Phi) \in \CS(\K^{n\times 1}),
	\]
	and $\CF_{\psi_\K}: \CS(\K^{1\times n})\to \CS(\K^{n\times 1})$ is a $\GL_n(\K)$-isomorphism between $\omega_n$ and $\check{\omega}_n$. Define 
	\[
	\wt\oZ_{\chi'} 
	\in \Hom_{\GL_{n}(\K)}\left(I(\varrho) \, \widehat{\otimes}\, I(\varrho') \, \widehat{\otimes}\, \check{\omega}_{n,\chi'}, \frak{M}_{n,\K}^{\ast} \right)
	\]
	such that
	\[
	\begin{aligned}
		&  \wt{\oZ}( f, f' ,\wh\Phi;\chi',\mathrm{d}g)   \\
		= \ & \int_{\mathrm N_{n}(\K)\backslash \GL_{n}(\K)}\langle \lambda'_{\wh\varrho},g.\wt{f} \rangle\langle \lambda_{\wh{\varrho'}} ,g. \wt{f'} \rangle\wh\Phi({}^t g\cdot  {}^te_n ) \chi'(\det(g))^{-1}|\det(g)|_{\K}  \overline{\mathrm{d}}g
	\end{aligned}
	\]
	when the integral is convergent. By \cite[Theorem 2.1]{J09} and Lemma \ref{lem:whit}, we have the functional equation 
	\be \label{FE}
	\begin{aligned}
		&\wt{\oZ}_{\chi'}( f, f', \wh\Phi;\mathrm{d}g) \\
		= \, &\omega_{\varrho}(-1) \omega_{\varrho'}(-1)^n\chi'(-1)^{\frac{n(n+1)}{2}} \cdot \gamma(0, \varrho\times\varrho'\times\chi', \psi_\K)\cdot \oZ_{\chi'}(f, f', \Phi;\mathrm{d}g),
	\end{aligned}
	\ee
	where $\omega_\varrho$, $\omega_{\varrho'}$ denote the central characters of $I(\varrho)$, $I(\varrho')$ respectively, and
	\[
	\gamma(s, \varrho\times\varrho'\times\chi', \psi_\K): = \prod_{i, k =1,2,\dots, n} \gamma(s, \varrho_i \cdot \varrho'_k\cdot\chi', \psi_\K).
	\]
	
	It is clear that if we replace $\psi_\K$ by $\overline{\psi_\K}$ and define 
	\[
	\oZ_{{\chi'}^{-1}|\cdot|_{\K}}\in \Hom_{\GL_{n}(\K)}\left(I(\wh\varrho) \, \widehat{\otimes}\, I(\wh{\varrho'}) \, \widehat{\otimes}\, \omega_{n,\chi'^{-1}|\cdot|_{\K}} , \frak{M}_{n,\K}^{\ast} \right)
	\]
	following \eqref{Zvarrho}, and define ${}^t\wh\Phi\in \CS(\K^{1\times n})$ by 
	\[
	{}^t\wh\Phi(x) := \widehat{\Phi}({}^tx),\quad x\in \K^{1\times n},
	\]
	then 
	\be \label{wtZ}
	\wt\oZ_{\chi'}( f,f',\wh\Phi;\mathrm{d}g) = \oZ_{{\chi'}^{-1}|\cdot|_{\K}}(\wt{f}, \wt{f'}, {}^t\wh\Phi ;\mathrm{d}g).
	\ee

	\subsection{Four commutative diagrams}
	
	We consider Case ($\pm$) in this and the next subsection. We now specialize the above discussion  to the principal series representations $I_\mu$ and $I_\nu$.  
	Define $\varrho^\mu = (\varrho^\mu_1, \ldots, \varrho^\mu_n)\in (\widehat{\K^\times})^n$, where
	\[
	\varrho^\mu_i  :=  \iota^{\tilde{\mu}_i^{\iota}}\bar{\iota}^{\tilde{\mu}_i^{\bar{\iota}}}\in \widehat{\K^\times}, \quad i =1,2,\ldots, n,
	\]
	so that
	$I_\mu = I(\varrho^\mu)$ and $\wt{\lambda}_{\mu} = \lambda_{\varrho^\mu}$
	in the notation of Section \ref{sec:lss}. Recall that the half-integers $\tilde{\mu}^{\iota'}_i$, $i=1,2,\ldots, n$, $\iota'\in \CE_\K$, are determined by $\mu$ as in \eqref{lmu}.  Put
	\[
	\wh\mu := (-\mu_n^\iota,\ldots, -\mu_1^\iota; -\mu_n^{\bar\iota}, \ldots, -\mu_1^{\bar\iota}) \in (\BZ^n)^{\CE_\K}.
	\]
	Then it holds that
	\[
	\wh{\varrho^\mu} = \varrho^{\wh\mu}.
	\]
	Likewise, define $\varrho^\nu = (\varrho^\nu_1,\ldots, \varrho^\nu_{n})\in (\widehat{\K^\times})^{n}$ so that $I_\nu = I(\varrho^\nu)$, $\wt\lambda'_{\nu} = \lambda'_{\varrho^\nu}$.
	
	We will factorize the map $\delta_{\xi, \chi}$ in \eqref{cube} using \eqref{alpha-fac} and Proposition \ref{prop:phixichi}. To this end, we introduce the following elements of $(\BZ^n)^{\CE_\K}$:
	\[
	\varsigma := (\mu^\iota_1+\chi_{\iota},\ldots, \mu^\iota_n+\chi_{\iota}; 0, \ldots, 0),\quad \upsilon := (\nu^\iota_1,\ldots, \nu^\iota_n; 0,\ldots, 0),
	\]
	which depend on $\mu$, $\nu$ and $\chi$. 
	Define half-integers $\tilde\varsigma^{\iota'}_i$ and  $\tilde{\upsilon}^{\iota'}_i$, $i=1, 2,\ldots, n$, $\iota'\in\CE_\K$ using $\varsigma$ and $\upsilon$ respectively as in \eqref{lmu}, which further give the elements
	\[
	\varrho^\varsigma=(\varrho^\varsigma_1,\ldots, \varrho^\varsigma_n), \ \varrho^\upsilon = (\varrho^\upsilon_1,\ldots, \varrho^\upsilon_n) \in (\wh{\K^\times})^n
	\]
	as above. Let $\phi_{\xi,\chi}^{\iota}$ and $\phi_{\xi,\chi}^{\overline{\iota}}$ be functions on
    \[
    \GL_n(\C_{\iota})\times\GL_n(\C_{\iota})\times\C_{\iota}^{1\times n}\qquad\text{and}\qquad\GL_n(\C_{\overline{\iota}})\times\GL_n(\C_{\overline{\iota}})\times\C_{\overline{\iota}}^{n\times 1}
    \]
    respectively such that $\phi_{\xi,\chi}=\phi_{\xi,\chi}^{\iota}\otimes\phi_{\xi,\chi}^{\overline{\iota}}$ normalized such that
    \[
    \phi_{\xi,\chi}^{\iota}(z^{\iota})=\phi_{\xi,\chi}^{\overline{\iota}}(\check{z}^{\overline{\iota}})=1.
    \]
    Put
    \[
    \begin{aligned}
     \phi_{\xi,\chi}^{\iota\prime}(g_1,g_2,x)&:=\iota(\det g_1)^{\chi_{\iota}}\cdot\phi_{\xi,\chi}^{\iota}(g_1,g_2,x),&\quad &g_1,g_2\in\GL_n(\C_{\iota}),x \in\C_{\iota}^{1\times n}\\
         \phi_{\xi,\chi}^{\overline{\iota}\prime}(g_1,g_2,y)&:=\iota(\det g_1)^{-\chi_{\iota}}\cdot\phi_{\xi,\chi}^{\overline{\iota}}(g_1,g_2,y),&\quad &g_1,g_2\in\GL_n(\C_{\overline{\iota}}),y\in\C_{\overline{\iota}}^{n\times 1}.
    \end{aligned}
    \]
  We can realize $C_{\eta,\chi}\cdot\delta_{\xi, \chi}$ as the following composition map evaluated at the character $\chi'=\chi_0$:
	\begin{equation} \label{delta-fac}
		\begin{aligned}
			\CI_{\xi_0, \chi'}  & \xrightarrow{\phi_{\xi,\chi}^{\iota\prime}\cdot} \CI_{\varsigma, \upsilon, \chi'}:= I(\varrho^\varsigma)\,\wh\otimes \, I(\varrho^\upsilon) \, \wh\otimes\, \omega_{n,\chi'} \\ 
			& \xrightarrow{\CF_{\psi_\K}} \check{\CI}_{\varsigma, \upsilon, \chi'}:= I(\varrho^\varsigma)\,\wh\otimes \, I(\varrho^\upsilon) \, \wh\otimes\, \check{\omega}_{n,\chi'} \\
			& \xrightarrow{\phi_{\xi, \chi}^{\bar\iota\prime}\cdot} \check{\CI}_{\xi, \chi'\chi_0^{-1}\chi}:=  I(\varrho^\mu)\,\wh\otimes \, I(\varrho^\nu) \, \wh\otimes\, \check{\omega}_{n,\chi'\chi_0^{-1}\chi} \\
			& \xrightarrow{\CF_{\overline{\psi_\K}}} \CI_{\xi, \chi'\chi_0^{-1}\chi} =  I(\varrho^\mu)\,\wh\otimes \, I(\varrho^\nu) \, \wh\otimes\, \omega_{n,\chi'\chi_0^{-1}\chi},
		\end{aligned}
	\end{equation}
	where the first and third arrows are multiplications by the functions $\phi_{\xi, \chi}^{\iota\prime}$ and $\phi_{\xi, \chi}^{\bar\iota\prime}$  respectively. Here we recall that the normalizing constant $C_{\eta,\chi}$ is given by \eqref{normalizingfactor}.
	
	In view of \eqref{delta-fac}, to evaluate the constant $\Omega_{\xi, \chi}'$ in \eqref{cube} , we only need to study the commutative diagrams in the next four lemmas. 
	
	\begin{lemp} \label{lem:CD1}
		The diagram 
		\[
		\begin{CD}
			\CI_{\xi_0, \chi'}^\sharp 
			@> \oZ_{\chi'} >> \mathfrak{M}^*_{n,\K} \\
			@ V  \phi_{\xi, \chi}^{\iota\prime}\cdot  V V          @  V V C_1(\chi')  V  \\
			\CI_{\varsigma, \upsilon, \chi'}^\sharp @> \oZ_{\chi'} >>  \mathfrak{M}^*_{n,\K}
		\end{CD}
		\]
		commutes, where
		\[
		C_1(\chi') = \frac{\Gamma_{\psi_\K}( \varrho^{0_n}, \varrho^{0_n},\chi' ) }{\Gamma_{\psi_\K}(\varrho^\varsigma, \varrho^\upsilon,\chi') }.
		\]
	\end{lemp}
	
	\begin{proof}
		It is clear that $\phi_{\xi, \chi}^{\iota\prime}\cdot  \CI_{\xi_0, \chi'}^\sharp  \subset \CI_{\varsigma, \upsilon, \chi'}^\sharp$. In view of Corollary \ref{cor:LSS}, it suffices to show that the diagram 
		\[
		\begin{CD}
			\CI_{\xi_0, \chi'}^\sharp 
			@> \Lambda_{\chi'} >> \mathfrak{M}^*_{n,\K} \\
			@ V  \phi_{\xi, \chi}^{\iota\prime}\cdot  V V          @  |  \\
			\CI_{\varsigma, \upsilon, \chi'}^\sharp @> \Lambda_{\chi'} >>  \mathfrak{M}^*_{n,\K}\\
		\end{CD}
		\]
		commutes. By  Proposition \ref{prop:phixichi},  for $\Psi\in \CI_{\xi_0, \chi'}^\sharp$ and $\mathrm{d}g\in \frak{M}_{n,\K}$ we have that
		\[
		\begin{aligned}
			&  \Lambda_{\chi'}(\phi_{\xi, \chi}^{\iota\prime} \cdot \Psi;\mathrm{d}g) \\ 
			= \ &  \int_{\GL_{n}(\K)} \phi_{\xi, \chi}^{\iota\prime}(z.g) \Psi(zg)\chi'(\det g)  \mathrm{d}g \\
			= \ &  \int_{\GL_{n}(\K)}  \Psi(z.g)\chi'(\det g)  \mathrm{d}g\\
			= \ & \Lambda_{\chi'}(\Psi;\mathrm{d}g).
		\end{aligned}
		\]
		This proves the lemma.
	\end{proof}
	
	\begin{lemp} \label{lem:CD2} The diagram 
		\[
		\begin{CD}
			\CI_{\varsigma, \upsilon, \chi'} 
			@> \oZ_{\chi'} >> \mathfrak{M}^*_{n,\K} \\
			@ V  \CF_{\psi_\K}  V V          @  V V C_2(\chi')  V  \\
			\check\CI_{\varsigma, \upsilon, \chi'} @> \wt\oZ_{\chi'} >>  \mathfrak{M}^*_{n,\K}
		\end{CD}
		\]
		commutes, where
		\[
		C_2(\chi') = \omega_{\varrho^\varsigma}(-1) \omega_{\varrho^\upsilon}(-1)^n\chi'(-1)^{\frac{n(n+1)}{2}} \cdot \gamma(0, \varrho^\varsigma\times \varrho^\upsilon\times\chi', \psi_\K).
		\]
	\end{lemp}
	
	\begin{proof}
		This is the functional equation \eqref{FE}.
	\end{proof}
	
	Put
	\[
	\check\CI_{\varsigma, \upsilon, \chi'}^\sharp:=\Set{ \Psi\in \check\CI_{\varsigma, \upsilon, \chi'} |  \Psi\vert_{\check z. \GL_{n}(\K)}\in
	\mathcal{S}( \check z. \GL_{n}(\K)) },
	\]
	and define $\check\CI_{\xi, \chi'\chi_0^{-1}\chi}^\sharp$ in the same way. 
	
	\begin{lemp} \label{lem:CD3}
		The diagram 
		\[
		\begin{CD}
			\check \CI_{\varsigma, \upsilon, \chi'}^\sharp
			@> \wt\oZ_{\chi'} >> \mathfrak{M}^*_{n,\K} \\
			@ V  \phi_{\xi, \chi}^{\bar\iota\prime}\cdot V V          @  V V C_3(\chi')  V  \\
			\check\CI_{\xi, \chi'\chi_0^{-1}\chi}^\sharp @> \wt\oZ_{\chi'\chi_0^{-1}\chi} >>  \mathfrak{M}^*_{n,\K}
		\end{CD}
		\]
		commutes, where
		\[
		C_3(\chi') =(-1)^{\frac{n(n-1)}{2}\chi_{\iota}}\cdot\frac{\Gamma_{\overline{\psi_\K}}( \varrho^{\wh\varsigma}, \varrho^{\wh\upsilon} ,\chi'^{-1}|\cdot|_{\K}) }{\Gamma_{\overline{\psi_\K}}( \varrho^{\wh\mu}, \varrho^{\wh\nu},\chi'^{-1}\chi_0\chi^{-1}|\cdot|_{\K}) }.
		\]
	\end{lemp}
	
	\begin{proof}
		For  $\Psi \in C^\infty(\GL_n(\K)\times \GL_n(\K)\times \K^{n\times 1})$, define $\wt\Psi \in C^\infty(\GL_n(\K)\times\GL_n(\K)\times \K^{1\times n})$ by
		\[
		\wt\Psi(g, h, x):= \Psi(w_ng^\tau, w_n h^\tau, {}^t x).
		\] 
		If $\Psi\in  \check\CI_{\varsigma, \upsilon, \chi'}$, then $\wt\Psi\in \CI_{\wh\varsigma, \wh\upsilon, \chi'^{-1}|\cdot|_{\K}}$  and by \eqref{wtZ} we have that  
		\[
		\wt\oZ_{\chi'}( \Psi;\mathrm{d}g) = \oZ_{\chi'^{-1}|\cdot|_{\K}}(\wt\Psi;\mathrm{d}g).
		\]
		Moreover, it is easy to verify that $\wt\Psi(z.g) = \Psi(\check z.g^\tau)$, $g\in \GL_n(\K)$,  hence $\wt\Psi\in \CI_{\wh\varsigma, \wh\upsilon, \chi'^{-1}|\cdot|_{\K}}^\sharp$ if and only if $\Psi\in \check\CI_{\varsigma, \upsilon, \chi'}^\sharp$.  
		A similar discussion applies to $\check\CI_{\xi, \chi'\chi_0^{-1}\chi}$ and $\CI_{\wh\xi, \chi'^{-1}\chi_0\chi^{-1}|\cdot|_{\K}}:=\CI_{\wh\mu, \wh\nu, \chi'^{-1}\chi_0\chi^{-1}|\cdot|_{\K}}$. Thus we obtain  a commutative diagram 
		\[
		\begin{CD}
			\check \CI_{\varsigma, \upsilon, \chi'}^\sharp
			@>   >>  \CI_{\wh\varsigma, \wh\upsilon, \chi'^{-1}|\cdot|_{\K}}^\sharp   \\
			@ V \phi_{\xi, \chi}^{\bar\iota\prime}\cdot  V V          @  VV \wt\phi_{\xi, \chi}^{\bar\iota\prime}\cdot  V  \\
			\check\CI_{\xi, \chi'\chi_0^{-1}\chi}^\sharp @>   >>  \CI_{\wh\xi, \chi'^{-1}\chi_0\chi^{-1}|\cdot|_{\K}}^\sharp
		\end{CD}
		\]
		where the horizontal arrows are given by $\Psi\mapsto \wt\Psi$. 
		
		Recall  that  $\oZ_{\chi'^{-1}|\cdot|_{\K}}(\wt\Psi;\mathrm{d}g)$ is defined following \eqref{Zvarrho} with $\psi_\K$ replaced by $\overline{\psi_\K}$. In view of Corollary \ref{cor:LSS} and the above discussions, it suffices to show that the diagram
		\[
		\begin{CD}
			\CI_{\wh\varsigma, \wh\upsilon, \chi'^{-1}|\cdot|_{\K}}^\sharp 
			@> \Lambda_{\chi'^{-1}|\cdot|_{\K}} >> \mathfrak{M}^*_{n,\K} \\
			@ V  \wt\phi_{\xi, \chi}^{\bar\iota\prime}\cdot  V V          @  VV\cdot (-1)^{\frac{n(n-1)}{2}\chi_{\iota}}V  \\
			\CI_{\wh\xi, \chi'^{-1}\chi_0\chi^{-1}|\cdot|_{\K}}^\sharp @> \Lambda_{\chi'^{-1}\chi_0\chi^{-1}|\cdot|_{\K}} >>  \mathfrak{M}^*_{n,\K}
		\end{CD}
		\]
		commutes. This is similar to the proof of Lemma \ref{lem:CD1}, noting that 
		\[
		\wt\phi_{\xi, \chi}^{\bar\iota\prime}( z.g) = \phi_{\xi, \chi}^{\bar\iota\prime}(\check z.g^\tau) = (-1)^{\frac{n(n-1)}{2}\chi_{\iota}}\cdot(\chi\chi^{-1}_0)(\det g)
		\]
		by Proposition \ref{prop:phixichi}. This proves the lemma.
	\end{proof}
	
	\begin{lemp} \label{lem:CD4}
		The diagram 
		\[
		\begin{CD}
			\check \CI_{\xi, \chi'\chi_0^{-1}\chi}
			@> \wt\oZ_{\chi'\chi_0^{-1}\chi} >> \mathfrak{M}^*_{n,\K} \\
			@ V  \CF_{\overline{\psi_\K}} V V          @  V V C_4(\chi')  V  \\
			\CI_{\xi, \chi'\chi_0^{-1}\chi} @> \oZ_{\chi'\chi_0^{-1}\chi}>>  \mathfrak{M}^*_{n,\K}
		\end{CD}
		\]
		commutes, where
		\[
		C_4(\chi') =\frac{1}{\omega_{\varrho^\mu}(-1) \omega_{\varrho^\nu}(-1)^n(\chi'\chi_0^{-1}\chi)(-1)^{\frac{n(n+1)}{2}}\cdot \gamma(0, \varrho^\mu \times \varrho^\nu\times\chi'\chi_0^{-1}\chi,\psi_\K)}.
		\]
	\end{lemp}
	
	\begin{proof}
		This is again the functional equation \eqref{FE}.
	\end{proof}
	
	\begin{corp}
		Let the notations be as above. Then
		\begin{equation} \label{const}
			\begin{aligned}
				\Omega_{\xi, \chi}' & =   C_{\eta,\chi}^{-1}\cdot\left.  
				\frac{\oL(0, \pi_{0_{n}}\times \pi_{0_{n}}\times\chi')}{\oL(0, \pi_\mu\times \pi_\nu\times\chi)} \cdot \prod^4_{i=1}C_i(\chi') \right\vert_{\chi'=\chi_0}.
			\end{aligned}
		\end{equation}
	\end{corp}
	
	\begin{proof}
		This follows from \eqref{cube} and Lemmas \ref{lem:CD1}--\ref{lem:CD4}.  
	\end{proof}

	\subsection{Case ($\pm$)}
	
	We are going to prove Theorem \ref{thmap2} for Case ($\pm$). It remains to evaluate  $\Omega_{\xi, \chi}'$ using \eqref{const}.
	
	We first make some preparations. Recall the archimedean L-factors given in Section \ref{Lfactor} and that (see \cite{K}) for  $a, b\in \BC$ with $a-b\in \BZ$, 
	\[
	\begin{aligned}
		\varepsilon(s, \iota^a \bar{\iota}^b, \psi_\K) &= (\varepsilon_{\psi_{\R}}\cdot\mathrm{i})^{\abs{a-b}}. 
	\end{aligned}
	\]
	Note that
	\[
	\gamma(s, \iota^a\bar\iota^b, \psi_\K) =  (\varepsilon_{\psi_{\R}}\cdot\mathrm{i})^{\abs{a-b}} \cdot \frac{\Gamma_\BC(1-s-\min\{a,b\})}{\Gamma_\BC(s+\max\{a,b\})}.
	\]
	From this,  we deduce the following useful lemma, which is also an easy consequence of the functional equation for Tate integrals. 
	
	\begin{lemp} \label{lem:gamma}
		For $\omega\in\wh{\K^\times}$, it holds that 
		\[
		\gamma(s, \omega, \psi_\K)\cdot \gamma(1-s, \omega^{-1}, \overline{\psi_\K}) =1.
		\]
	\end{lemp}
	
	For $i=1, 2,\ldots, n$, write 
	\[
	\check i := n+1-i,
	\]
	so that for $\varrho=(\varrho_1,\ldots, \varrho_n)\in (\wh{\K^\times})^n$ we have that 
	\[
	\wh\varrho = (\varrho^{-1}_{\check 1}, \ldots, \varrho^{-1}_{\check n}).
	\]
	With this notation, we recall from Section \ref{sec:TP} that
	\[
	\pi_\mu \cong I(\iota^{\tilde\mu^\iota_1}\bar\iota^{\tilde\mu^{\bar\iota}_{\check 1}}, \ldots, \iota^{\tilde\mu^\iota_n}\bar\iota^{\tilde\mu^{\bar\iota}_{\check n}}).
	\]
	The following result is an immediate consequence of the balanced condition in Lemma \ref{lem-balanced}.
	
	\begin{lemp} \label{lem:Rag}
		Assume that $\chi \in \RB(\xi)$. Then the following hold.
		\begin{itemize}
			\item[(a)]
			$
			\tilde{\mu}^\iota_i + \tilde{\nu}^\iota_k+\chi_{\iota} - \tilde{\mu}^{\bar{\iota}}_{\check i}  - \tilde{\nu}^{\bar{\iota}}_{\check k}-\chi_{\overline{\iota}}
			$
			is positive if $i+k\leq n$, and is negative otherwise.
			
			\item[(b)]
			$\tilde{\mu}^\iota_i + \tilde{\nu}^\iota_{k}+\chi_{\iota} \leq 0< \tilde{\mu}^{\bar\iota}_i + \tilde{\nu}^{\bar\iota}_k+\chi_{\overline{\iota}}$ if $i+k=n+1$. 
		\end{itemize} 
	\end{lemp}
	
	The rest of this subsection is devoted to the proof of  Proposition \ref{prop:Omega} below, which thereby finishes the proof of Theorem \ref{thmap2}.
	
	\begin{prp} \label{prop:Omega}
		The constant $\Omega_{\xi, \chi}'$ is nonzero, with inverse equal to the constant $c'_{\xi,\chi}\cdot\varepsilon'_{\xi,\chi}$ given in Theorem \ref{mainthm}.
	\end{prp}
	
	\begin{proof}
		We first  collect the contributions to $\Omega_{\xi, \chi}'^{-1}$ from various kinds of factors. By \eqref{LSS:gamma} and Lemma \ref{lem:gamma}, we find that 
		\[
		\gamma_{\overline{\psi_\K}}(1, \varrho^{\wh\mu}, \varrho^{\wh\nu},\chi'^{-1}\chi_0\chi^{-1})\cdot \gamma(0, \varrho^\mu\times \varrho^\nu\times\chi'\chi_0^{-1}\chi, \psi_\K) = \prod_{i+k\leq n+1}\gamma(0, \varrho^\mu_i \cdot\varrho^\nu_k\cdot\chi'\chi_0^{-1}\chi, \psi_\K)
		\]
		and that
		\[
		\frac{\gamma_{\overline{\psi_\K}}(1, \varrho^{\wh\varsigma}, \varrho^{\wh\upsilon},\chi'^{-1} )}{\gamma_{\psi_\K}(0, \varrho^\varsigma, \varrho^\upsilon,\chi')}   \cdot \gamma(0, \varrho^\varsigma\times \varrho^\upsilon\times\chi', \psi_\K)
		= \prod_{i+k=n+1}\gamma(0, \varrho^\varsigma_i \cdot\varrho^\upsilon_k\cdot\chi', \psi_\K).
		\]
		Unfolding \eqref{LSS:Gamma}, it follows that 
		\begin{equation} \label{omega'-1}
			\begin{aligned}
				\Omega_{\xi, \chi}'^{-1} = \varepsilon\cdot C_{\eta,\chi} &\cdot  \frac{ \prod_{i+k\leq n+1}\gamma(0, \varrho^\mu_i \varrho^\nu_k\chi'\chi_0^{-1}\chi, \psi_\K)}{\left(\prod_{i+k\leq n}\gamma(0, \varrho^{0_n}_i \varrho^{0_n}_k\chi', \psi_\K)\right)\cdot \left(\prod_{i+k=n+1}\gamma(0, \varrho^\varsigma_i \varrho^\upsilon_k\chi', \psi_\K)\right)}\\
				&\cdot\left.\frac{\oL(0, \pi_\mu\times \pi_\nu\times\chi'\chi_0^{-1}\chi)}{\oL(0, \pi_{0_n}\times \pi_{0_n}\times\chi')}\right\vert_{\chi'=\chi_0},
			\end{aligned}
		\end{equation}
		where $\varepsilon=\pm1$ is given by
		\be \label{sign}
        \begin{aligned}
		\varepsilon :=&(-1)^{\frac{n(n-1)}{2}\chi_{\iota}}(\chi\chi_0^{-1})(-1)^{\frac{n(n+1)}{2}}\\
        &\cdot\frac{\sgn(\varrho^{0_n}, \varrho^{0_n},\chi_0)\sgn(\varrho^{\wh\varsigma}, \varrho^{\wh\upsilon},\chi_0)}{\sgn(\varrho^\varsigma, \varrho^\upsilon,\chi_0)\sgn(\varrho^{\wh\mu},\varrho^{\wh\nu},\chi)}
		\cdot \frac{\omega_{\varrho^\varsigma}(-1) \omega_{\varrho^\upsilon}(-1)^n}{\omega_{\varrho^\mu}(-1)\omega_{\varrho^\nu}(-1)^n}.
        \end{aligned}
		\ee
		We evaluate the different kinds of  factors in \eqref{omega'-1}  as follows. 
		
		Using Lemma \ref{lem:Rag} (a), we find that
		\[
		\begin{aligned}
			\oL(0, \pi_\mu\times \pi_\nu\times\chi'\chi_0^{-1}\chi) =& \prod_{i+k\leq n} \Gamma_\BC(\tilde{\mu}^{\iota}_i +\tilde{\nu}^{\iota}_k+\chi'_{\iota}+\chi_{\iota})\\
        &\cdot\prod_{i+k\leq n} \Gamma_\BC(\tilde{\mu}^{\overline{\iota}}_i +\tilde{\nu}^{\overline{\iota}}_k+\chi'_{\overline{\iota}}+\chi_{\overline{\iota}}-1)\\
		&\cdot \prod_{i+k=n+1}\Gamma_\BC(\tilde{\mu}_i^{\bar{\iota}}+\tilde{\nu}_k^{\bar{\iota}}+\chi'_{\overline{\iota}}+\chi_{\overline{\iota}}-1).
		\end{aligned}
		\]
		We have that
		\[
        \begin{aligned}
        &\frac{\oL(1, (\varrho^\mu_i \varrho^\nu_k\chi'\chi_0^{-1}\chi)^{-1})}{\oL(0, \varrho^\mu_i\varrho^\nu_k\chi'\chi_0^{-1}\chi)} \\
        = &\frac{\Gamma_\BC(1-\min\{\tilde{\mu}^{\iota}_i + \tilde{\nu}^{\iota}_k+\chi'_{\iota}+\chi_{\iota},\tilde{\mu}^{\overline{\iota}}_i + \tilde{\nu}^{\overline{\iota}}_k+\chi'_{\overline{\iota}}+\chi_{\overline{\iota}}-1\})}{\Gamma_\BC(\max\{\tilde{\mu}^{\iota}_i + \tilde{\nu}^{\iota}_k+\chi'_{\iota}+\chi_{\iota},\tilde{\mu}^{\overline{\iota}}_i + \tilde{\nu}^{\overline{\iota}}_k+\chi'_{\overline{\iota}}+\chi_{\overline{\iota}}-1\})}.
        \end{aligned}
		\]
		From this and Lemma \ref{lem:Rag} (b),  
		\be \label{C1}
		\begin{aligned}
			& \left( \prod_{i+k\leq n+1} \frac{\oL(1, (\varrho^\mu_i \varrho^\nu_k\chi'\chi_0^{-1}\chi)^{-1})}{\oL(0, \varrho^\mu_i\varrho^\nu_k\chi'\chi_0^{-1}\chi)} \right) \cdot  \oL(0, \pi_\mu\times \pi_\nu\times\chi'\chi_0^{-1}\chi)\\
			= \ & \prod_{i+k\leq n}\Gamma_{\C}\left(\min\left\{
            \tilde{\mu}^{\iota}_i + \tilde{\nu}^{\iota}_k+\chi'_{\iota}+\chi_{\iota},\tilde{\mu}^{\overline{\iota}}_i + \tilde{\nu}^{\overline{\iota}}_k+\chi'_{\overline{\iota}}+\chi_{\overline{\iota}}-1
            \right\}\right)\\
            &\cdot\prod_{i+k\leq n}\Gamma_{\C}\left(1-\min\left\{
            \tilde{\mu}^{\iota}_i + \tilde{\nu}^{\iota}_k+\chi'_{\iota}+\chi_{\iota},\tilde{\mu}^{\overline{\iota}}_i + \tilde{\nu}^{\overline{\iota}}_k+\chi'_{\overline{\iota}}+\chi_{\overline{\iota}}-1
            \right\}\right)\\
			&  \cdot \prod_{i+k=n+1}\Gamma_\BC(1- \tilde{\mu}_i^{\iota} - \tilde{\nu}_k^{\iota}-\chi'_{\iota}-\chi_{\iota}).
		\end{aligned}
		\ee
		In a similar manner, we find that
		\begin{equation} \label{C1'}
			\begin{aligned}
				& \left( \prod_{i+k\leq n} \frac{\oL(1, (\varrho^{0_n}_i \varrho^{0_n}_k\chi')^{-1})}{\oL(0, \varrho^{0_n}_i\varrho^{0_n}_k\chi')} \right) \cdot 
				\left(\prod_{i+k=n+1}\frac{\oL(1,  (\varrho^\varsigma_i \varrho^\upsilon_k\chi')^{-1})}{\oL(0, \varrho^\varsigma_i \varrho^\upsilon_k\chi')}\right)\\
                &\cdot \oL(0, \pi_{0_n}\times \pi_{0_n}\times\chi')\\
				= \, & \prod_{i+k\leq n}\left( \Gamma_\BC(n+1-i-k+\min_{\iota'\in\mathcal{E}_{\K}}\{\chi'_{\iota'}\})
				\cdot\Gamma_\BC(-n+i+k-\min_{\iota'\in\mathcal{E}_{\K}}\{\chi'_{\iota'}\})\right) \\
				&  \cdot \prod_{i+k=n+1}\Gamma_\BC(1- \tilde{\mu}_i^{\iota} - \tilde{\nu}_k^{\iota}-\chi'_{\iota}-\chi_{\iota}).
			\end{aligned}
		\end{equation}
		
		It is easy to verify the following identity
		\be \label{gamma-id}
		\Gamma_\BC(s+ \ell)\cdot\Gamma_\BC(1-s-\ell) = (-1)^\ell \cdot \Gamma_\BC(s)\cdot \Gamma_\BC(1-s),\quad \ell \in \BZ.
		\ee
		Then the ratio of \eqref{C1}  and \eqref{C1'} equals 
		\be \label{LC}
		\prod_{i+k\leq n}(-1)^{\min_{\iota'\in\CE_\K} \{ \mu^{\iota'}_i + \nu^{\iota'}_k+\chi_{\iota'}\}} = \prod_{i+k\leq n}{\rm i}^{2\min_{\iota'\in\CE_\K} \{ \mu^{\iota'}_i + \nu^{\iota'}_k+\chi_{\iota'}\}}.
		\ee
		Specializing to $\chi'=\chi_0$, we have the local epsilon factor
		\[
		\varepsilon(0, \varrho^\mu_i \varrho^\nu_k\chi, \psi_\K) =  (\varepsilon_{\psi_{\R}}\cdot\mathrm{i})^{\max_{\iota'\in\CE_\K}\{\mu^{\iota'}_i + \nu^{\iota'}_k+\chi_{\iota'} \} - \min_{\iota'\in \CE_\K} \{ \mu^{ \iota'}_i + \nu^{ \iota'}_k+\chi_{\iota'}\} }.
		\]
In particular, $\varepsilon(0,\varrho_i^{0_n}\varrho_k^{0_n}\chi_0,\psi_{\K})=\varepsilon_{\psi_{\R}}\cdot\mathrm{i}$.  By Lemma \ref{lem:Rag} (b) again, for $i+k=n+1$ it holds that 
		\[
		\begin{aligned}
			\varepsilon(0, \varrho^\mu_i \varrho^\nu_k\chi, \psi_\K)  &= (\varepsilon_{\psi_{\R}}\cdot\mathrm{i})^{\mu^{\bar\iota}_i + \nu^{\bar\iota}_k +\chi_{\overline{\iota}}- \mu^\iota_i - \nu^\iota_k-\chi_{\iota}},\\
			\varepsilon(0, \varrho^\varsigma_i \varrho^\upsilon_k\chi_0, \psi_\K)  &=  (\varepsilon_{\psi_{\R}}\cdot\mathrm{i})^{1-\chi_{\iota} - \mu^\iota_i - \nu^\iota_k}.
		\end{aligned}
		\]
		
		In view of  \eqref{LC} and the above results  for local epsilon factors, it can be verified that the total contribution of all the local gamma factors and L-factors in \eqref{omega'-1} equals 
		\be \label{gamma-L}
		\prod_{i=1}^n(\varepsilon_{\psi_{\R}}\cdot\mathrm{i})^{\mu_i^{\overline{\iota}}+\nu_i^{\overline{\iota}}+\chi_{\overline{\iota}}-1}\cdot  \prod_{i+k\leq n}  (\varepsilon_{\psi_{\R}}\cdot\mathrm{i})^{\mu^{\iota}_i+\mu_i^{\overline{\iota}}+\nu^{\iota}_k+\nu_k^{\overline{\iota}}+\chi_{\iota}+\chi_{\overline{\iota}}-1}.
		\ee
		
		It is straightforward to  evaluate \eqref{sign} and find that
		\be  \label{sign-value}
        \begin{aligned}
		\varepsilon  = (-1)^{\frac{n(n-1)}{2}(\chi_{\overline{\iota}}-1)}&\cdot\prod_{i=1}^n(-1)^{\mu_i^{\overline{\iota}}+n(\nu_i^{\overline{\iota}}-1)+\chi_{\overline{\iota}}}
		\\
        &\cdot\prod_{i>k,\, i+k\leq n}(-1)^{\mu_i^{\iota}+\nu_k^{\iota}+\mu_{n+1-i}^{\overline{\iota}}+\nu_{n+1-k}^{\overline{\iota}}+\chi_{\iota}+\chi_{\overline{\iota}}-1}.
        \end{aligned}
		\ee
		
		Taking the product of $C_{\eta,\chi}$, \eqref{gamma-L} and \eqref{sign-value}, we have
		\[
\Omega_{\xi,\chi}'^{-1}=c'_{\xi,\chi}\cdot\varepsilon'_{\xi,\chi}
		\]
		which finishes the proof of the proposition. 
	\end{proof}

	\subsection{Case ($-$)}
	
	We now prove Theorem \ref{thmap2} for Case ($-$). Replacing
    \[
\left(\pi_{\mu},\pi_{\nu},\chi\right)\qquad\text{by}\qquad\left(\pi_{\mu}\otimes\varepsilon_{\pi_{\mu}},\pi_{\nu}\otimes\varepsilon_{\pi_{\nu}},\chi\varepsilon_{\pi_{\mu}}\varepsilon_{\pi_{\nu}}\right),
    \]
     we may assume that $\varepsilon_{\pi_{\mu}}=\varepsilon_{\pi_{\nu}}=\mathbf{1}_{\K^{\times}}$ is trivial. We apply the discussion in Section \ref{sec:lss} to principal series representations $I_{\mu}=I(\varrho^{\mu})$, $I_{\nu}=I(\varrho^{\nu})$, where $\varrho^{\mu}=(\varrho_1^{\mu},\dots,\varrho_{n}^{\mu})\in(\widehat{\K^{\times}})^n$ and $\varrho^{\nu}=(\varrho_1^{\nu},\dots,\varrho_{n}^{\nu})\in(\widehat{\K^{\times}})^n$ are the corresponding characters as in Section \ref{sec2.1.1}.
	
	As in the proofs of Lemma \ref{lem:CD1} and Lemma \ref{lem:CD3}, we have the following commutative diagram:
	\[
	\begin{CD}
		\CI_{\xi_0,\chi'}^{\sharp} @>\Lambda_{\chi'}>> \mathfrak{M}_{n,\K}^{\ast}\\
		@V\phi_{\xi,\chi}\cdot VV @|\\
		\CI_{\xi,\chi'\chi_0^{-1}\chi}^{\sharp} @>\Lambda_{\chi'\chi_0^{-1}\chi}>>\mathfrak{M}_{n,\K}^{\ast},
	\end{CD}
	\]
    where the left vertical arrow is the multiplication by $\phi_{\xi,\chi}$.
	Then
	\begin{equation} \label{const'}
		\Omega_{\xi,\chi}'=C_{\eta,\chi}^{-1}\cdot\left.\frac{\RL(0,\pi_{0_n}\times\pi_{0_n}\times\chi')}{\RL(0,\pi_{\mu}\times\pi_{\nu}\times\chi'\chi_0^{-1}\chi)}\cdot\frac{\Gamma_{\psi_{\K}}(\varrho^{0_n},\varrho^{0_n},\chi')}{\Gamma_{\psi_\K}(\varrho^{\mu}, \varrho^{\nu},\chi'\chi_0^{-1}\chi)}\right|_{\chi'=\chi_0}.
	\end{equation}
	It remains to evaluate  $\Omega_{\xi, \chi}'$ using \eqref{const'} and show that its inverse is equal to the constant $c_{\xi,\chi}\cdot\varepsilon''_{\xi,\chi}\cdot g_{\xi,\chi}(0)$ given in Theorem \ref{mainthm}.
	
	We need to calculate
	\[
		\Omega_{\xi,\chi}'^{-1}=C_{\eta,\chi}\cdot\left.\frac{\mathrm{sgn}(\varrho^{\mu},\varrho^{\nu},\chi)}{\mathrm{sgn}(\varrho^{0_n},\varrho^{0_n},\chi_0)}\cdot\frac{\gamma_{\psi_{\K}}(0,\varrho^{\mu},\varrho^{\nu},\chi'\chi_0^{-1}\chi)}{\gamma_{\psi_{\K}}(0,\varrho^{0_n},\varrho^{0_n},\chi')}\cdot\frac{\RL(0,\pi_{\mu}\times\pi_{\nu}\times\chi'\chi_0^{-1}\chi)}{\RL(0,\pi_{0_n}\times\pi_{0_n}\times\chi')}\right|_{\chi'=\chi_0}.
	\]
	One easily computes that
	\[
	\frac{\mathrm{sgn}(\varrho^{\mu},\varrho^{\nu},\chi)}{\mathrm{sgn}(\varrho^{0_n},\varrho^{0_n},\chi_0)}=\prod_{i>k,\,i+k\leq n}(-1)^{\sum_{\iota'\in\mathcal{E}_{\K}}(\mu_i^{\iota'}+\nu_k^{\iota'}+\chi_{\iota'})}=\varepsilon_{\xi,\chi}.
	\]
	
	Since $\chi\in\RB(\xi)$ is balanced, as a consequence of the balanced condition in Lemma \ref{lem-balanced}, we have that
	\[
	\widetilde{\mu}_i^{\iota}+\widetilde{\nu}_k^{\iota}+\chi_{\iota}-\widetilde{\mu}^{\overline{\iota}}_{\check{i}}-\widetilde{\nu}^{\overline{\iota}}_{\check k}-\chi_{\overline{\iota}}
	\]
	is positive if $i+k\leq n$, and is negative if $i+k\geq n+2$.

	\subsubsection{\emph{$\K\cong\C$}}
	
	 We find that
	\[
	\begin{aligned}
		&\oL(0, \pi_\mu\times \pi_\nu\times\chi'\chi_0^{-1}\chi) =\prod_{i+k\leq n, \,\iota'\in \CE_\K} \Gamma_\BC(\tilde{\mu}^{\iota'}_i +\tilde{\nu}^{\iota'}_k+\chi'_{\iota'}+\chi_{\iota'})\\
		&\qquad \cdot \prod_{i=1}^n\Gamma_\BC\left(\max\{\mu_i^{\iota}+\nu_{n+1-i}^{\iota}+\chi_{\iota}'+\chi_{\iota},\mu_{n+1-i}^{\overline{\iota}}+\nu_{i}^{\overline{\iota}}+\chi_{\overline{\iota}}'+\chi_{\overline{\iota}}\}\right).
	\end{aligned}
	\]
	We have that
	\[
	\frac{\oL(1, (\varrho^\mu_i \varrho^\nu_k\chi'\chi_0^{-1}\chi)^{-1})}{\oL(0, \varrho^\mu_i\varrho^\nu_k\chi'\chi_0^{-1}\chi)} = \frac{\Gamma_\BC(1-\min_{\iota' \in\CE_\K}\{\tilde{\mu}^{\iota'}_i + \tilde{\nu}^{\iota'}_k+\chi'_{\iota'}+\chi_{\iota'}\})}{\Gamma_\BC(\max_{\iota'\in \CE_\K}\{\tilde{\mu}^{\iota'}_i + \tilde{\nu}^{\iota'}_k+\chi'_{\iota'}+\chi_{\iota'}\})}.
	\]
	Hence
	\[
	\begin{aligned}
		&\left(\prod_{i+k\leq n}\frac{\RL(1,(\varrho_i^{\mu}\varrho_k^{\nu}\chi'\chi_0^{-1}\chi)^{-1})}{\RL(0,\varrho_i^{\mu}\varrho_k^{\nu}\chi'\chi_0^{-1}\chi)}\right)\cdot\RL(0,\pi_{\mu}\times\pi_{\nu}\times\chi'\chi_0^{-1}\chi)\\
		=&\prod_{i+k\leq n}\left(\Gamma_{\C}(\min_{\iota'\in\mathcal{E}_{\K}}\{\widetilde{\mu}_i^{\iota'}+\widetilde{\nu}_k^{\iota'}+\chi'_{\iota'}+\chi_{\iota'}\})\Gamma_{\C}(1-\min_{\iota'\in\mathcal{E}_{\K}}\{\widetilde{\mu}_i^{\iota'}+\widetilde{\nu}_k^{\iota'}+\chi'_{\iota'}+\chi_{\iota'}\})\right)\\
		& \cdot  \prod_{i=1}^n\Gamma_\BC\left(\max\{\mu_i^{\iota}+\nu_{n+1-i}^{\iota}+\chi_{\iota}'+\chi_{\iota},\mu_{n+1-i}^{\overline{\iota}}+\nu_{i}^{\overline{\iota}}+\chi_{\overline{\iota}}'+\chi_{\overline{\iota}}\}\right).
	\end{aligned}
	\]	
	
	Specializing to $\chi'=\chi_0$, using the identity \eqref{gamma-id}, we calculate that
	\[
	\begin{aligned}
		&\left.\frac{\gamma_{\psi_{\K}}(0,\varrho^{\mu},\varrho^{\nu},\chi'\chi_0^{-1}\chi)}{\gamma_{\psi_{\K}}(0,\varrho^{0_n},\varrho^{0_n},\chi')}\cdot\frac{\RL(0,\pi_{\mu}\times\pi_{\nu}\times\chi'\chi_0^{-1}\chi)}{\RL(0,\pi_{0_n}\times\pi_{0_n}\times\chi')}\right|_{\chi'=\chi_0}\\
        =&\prod_{i+k\leq n,\,\iota'\in\mathcal{E}_{\K}}(\varepsilon_{\psi_\R}\cdot\mathrm{i})^{\mu_i^{\iota'}+\nu_k^{\iota'}+\chi_{\iota'}}\\
        &\cdot\prod_{i=1}^n\left.\frac{\Gamma_\BC\left(s+\max\{\mu_i^{\iota}+\nu_{n+1-i}^{\iota}+\chi_{\iota},\mu_{n+1-i}^{\overline{\iota}}+\nu_{i}^{\overline{\iota}}+\chi_{\overline{\iota}}\}\right)}{\Gamma_{\C}(s)}\right|_{s=0},
	\end{aligned}
	\]
	which completes the proof for $\K\cong\C$.

	\subsubsection{\emph{$\K\cong\R$}}
	
	When $n$ is even, we have
	\[
	\begin{aligned}
		& \RL(0,\pi_{\mu}\times\pi_{\nu}\times\chi'\chi_0^{-1}\chi) \\
         = &  \prod_{i+k\leq n} \Gamma_\BC(\tilde{\mu}^{\iota}_i +\tilde{\nu}^{\iota}_k+\chi'_{\iota}+\chi_{\iota})\\
		&\cdot \prod_{i=1}^{\frac{n}{2}}\Gamma_\BC\left(\max\{\widetilde{\mu}_i^{\iota}+\widetilde{\nu}_{n+1-i}^{\iota}+\chi_{\iota}'+\chi_{\iota},\widetilde{\mu}_{n+1-i}^{\iota}+\widetilde{\nu}_{i}^{\iota}+\chi_{\iota}'+\chi_{\iota}\}\right).
	\end{aligned}
	\]
For a character $\omega\in\widehat{\K^{\times}}$, we write $\delta(\omega)\in\{0,1\}$ such that $\omega(-1)=(-1)^{\delta(\omega)}$. Then
	\[
	\varepsilon(s,\omega,\psi_{\K})=(\varepsilon_{\psi_{\R}}\cdot\mathrm{i})^{\delta(\omega)}.
	\]
	We have that
	\[
	\frac{\oL(1, (\varrho^\mu_i \varrho^\nu_k\chi'\chi_0^{-1}\chi)^{-1})}{\oL(0, \varrho^\mu_i\varrho^\nu_k\chi'\chi_0^{-1}\chi)}=\frac{\Gamma_{\R}\left(1-\widetilde{\mu}_i^{\iota}-\widetilde{\nu}_k^{\iota}-\chi_{\iota}'-\chi_{\iota}+\delta(\varrho_i^{\mu}\varrho_k^{\nu}\chi'\chi_0^{-1}\chi)\right)}{\Gamma_{\R}\left(\widetilde{\mu}_i^{\iota}+\widetilde{\nu}_k^{\iota}+\chi_{\iota}'+\chi_{\iota}+\delta(\varrho_i^{\mu}\varrho_k^{\nu}\chi'\chi_0^{-1}\chi)\right)}.
	\]
	Hence
	\[
	\begin{aligned}
		&\left(\prod_{i+k\leq n}\frac{\RL(1,(\varrho_i^{\mu}\varrho_k^{\nu}\chi'\chi_0^{-1}\chi)^{-1})}{\RL(0,\varrho_i^{\mu}\varrho_k^{\nu}\chi'\chi_0^{-1}\chi)}\right)\cdot\RL(0,\pi_{\mu}\times\pi_{\nu}\times\chi'\chi_0^{-1}\chi)\\
		=&\prod_{i+k\leq n}\left(\Gamma_{\R}(\widetilde{\mu}_i^{\iota}+\widetilde{\nu}_k^{\iota}+\chi_{\iota}'+\chi_{\iota}+1-\delta(\varrho_i^{\mu}\varrho_k^{\nu}\chi'\chi_0^{-1}\chi))\right.\\
        &\qquad\quad\left.\cdot\Gamma_{\R}(1-\widetilde{\mu}_i^{\iota}-\widetilde{\nu}_k^{\iota}-\chi_{\iota}'-\chi_{\iota}+\delta(\varrho_i^{\mu}\varrho_k^{\nu}\chi'\chi_0^{-1}\chi))\right)\\
		& \cdot  \prod_{i=1}^{\frac{n}{2}}\Gamma_\BC\left(\max\{\widetilde{\mu}_i^{\iota}+\widetilde{\nu}_{n+1-i}^{\iota}+\chi_{\iota}'+\chi_{\iota},\widetilde{\mu}_{n+1-i}^{\iota}+\widetilde{\nu}_{i}^{\iota}+\chi_{\iota}'+\chi_{\iota}\}\right).
	\end{aligned}
	\]
	Specializing to $\chi'=\chi_0$, using the identity
	\[
	\Gamma_{\R}(s+l)\cdot\Gamma_{\R}(2-s-l)=\mathrm{i}^l\cdot\Gamma_{\R}(s)\cdot\Gamma_{\R}(2-s),\qquad l\in2\Z,
	\]
	we calculate that
	\[
	\begin{aligned}
		&\left.\frac{\gamma_{\psi_{\K}}(0,\varrho^{\mu},\varrho^{\nu},\chi'\chi_0^{-1}\chi)}{\gamma_{\psi_{\K}}(0,\varrho^{0_n},\varrho^{0_n},\chi')}\cdot\frac{\RL(0,\pi_{\mu}\times\pi_{\nu}\times\chi'\chi_0^{-1}\chi)}{\RL(0,\pi_{0_n}\times\pi_{0_n}\times\chi')}\right|_{\chi'=\chi_0}\\
		=&\prod_{i+k\leq n}(\varepsilon_{\psi_\R}\cdot\mathrm{i})^{\mu_i^{\iota}+\nu_k^{\iota}+\chi_{\iota}}\cdot\prod_{i=1}^{\frac{n}{2}}\left.\frac{\Gamma_\BC\left(s+\max\{\mu_i^{\iota}+\nu_{n+1-i}^{\iota}+\chi_{\iota},\mu_{n+1-i}^{\iota}+\nu_{i}^{\iota}+\chi_{\iota}\}\right)}{\Gamma_{\C}(s)}\right|_{s=0},
	\end{aligned}
	\]
    as desired.
	
	When $n$ is odd, we have
	\[
	\begin{aligned}
		& \RL(0,\pi_{\mu}\times\pi_{\nu}\times\chi'\chi_0^{-1}\chi)\\
         = & \prod_{i+k\leq n} \Gamma_\BC(\tilde{\mu}^{\iota}_i +\tilde{\nu}^{\iota}_k+\chi'_{\iota}+\chi_{\iota})\\
		& \cdot \prod_{i=1}^{\frac{n-1}{2}}\Gamma_\BC\left(\max\{\widetilde{\mu}_i^{\iota}+\widetilde{\nu}_{n+1-i}^{\iota}+\chi_{\iota}'+\chi_{\iota},\widetilde{\mu}_{n+1-i}^{\iota}+\widetilde{\nu}_{i}^{\iota}+\chi_{\iota}'+\chi_{\iota}\}\right)\\
		& \cdot \Gamma_{\R}\left(\widetilde{\mu}_{\frac{n+1}{2}}^{\iota}+\widetilde{\nu}_{\frac{n+1}{2}}^{\iota}+\chi'_{\iota}+\chi_{\iota}+\delta(\varrho^{\mu}_{\frac{n+1}{2}}\varrho^{\nu}_{\frac{n+1}{2}}\chi'\chi_0^{-1}\chi)\right).
	\end{aligned}
	\]
	The same computation as above shows that
	\[
	\begin{aligned}
		&\left.\frac{\gamma_{\psi_{\K}}(0,\varrho^{\mu},\varrho^{\nu},\chi'\chi_0^{-1}\chi)}{\gamma_{\psi_{\K}}(0,\varrho^{0_n},\varrho^{0_n},\chi')}\cdot\frac{\RL(0,\pi_{\mu}\times\pi_{\nu}\times\chi'\chi_0^{-1}\chi)}{\RL(0,\pi_{0_n}\times\pi_{0_n}\times\chi')}\right|_{\chi'=\chi_0}\\
		=&\prod_{i+k\leq n}(\varepsilon_{\psi_\R}\cdot\mathrm{i})^{\mu_i^{\iota}+\nu_k^{\iota}+\chi_{\iota}}\cdot\prod_{i=1}^{\frac{n-1}{2}}\frac{\Gamma_\BC\left(s+\max\{\mu_i^{\iota}+\nu_{n+1-i}^{\iota}+\chi_{\iota},\mu_{n+1-i}^{\iota}+\nu_{i}^{\iota}+\chi_{\iota}\}\right)}{\Gamma_{\C}(s)}\\
		&\cdot \left.\frac{\Gamma_{\R}\left(s+\mu_{\frac{n+1}{2}}^{\iota}+\nu_{\frac{n+1}{2}}^{\iota}+\chi_{\iota}+\delta(\varrho_{\frac{n+1}{2}}^{\mu}\varrho^{\nu}_{\frac{n+1}{2}}\chi)\right)}{\Gamma_{\R}(s+\delta(\chi_0))}\right|_{s=0}.
	\end{aligned}
	\]
	This completes the proof for $\K\cong\R$.

	\section*{Acknowledgments}
	The authors thank Michael Harris for raising the question on non-vanishing. D. Liu was supported in part by National Key R \& D Program of China No. 2022YFA1005300 and  Zhejiang Provincial Natural Science Foundation of China under Grant No. LZ22A010006. B. Sun was supported in part by  National Key R \& D Program of China No. 2022YFA1005300, and New Cornerstone Science Foundation.

    After the paper was written, ChatGPT-6 Astra was used to proofread the manuscript and verify computations, helping to identify and correct various errors in earlier versions. 

\begin{thebibliography}{}
		
		\bibitem[AG08]{AG1}
		Avraham Aizenbud and Dmitry Gourevitch.
		\newblock Schwartz functions on {N}ash manifolds.
		\newblock {\em Int. Math. Res. Not. IMRN}, (5):Art. ID rnm 155, 37, 2008.
		
		\bibitem[BR17]{BR17}
		Baskar Balasubramanyam and A.~Raghuram.
		\newblock Special values of adjoint {$L$}-functions and congruences for
		automorphic forms on {${\rm GL}(n)$} over a number field.
		\newblock {\em Amer. J. Math.}, 139(3):641--679, 2017.
		
		\bibitem[BW00]{BW}
		A.~Borel and N.~Wallach.
		\newblock {\em Continuous cohomology, discrete subgroups, and representations
			of reductive groups}, volume~67 of {\em Mathematical Surveys and Monographs}.
		\newblock American Mathematical Society, Providence, RI, second edition, 2000.
		
		\bibitem[Che22]{Chen}
		Shih-Yu Chen.
		\newblock Algebraicity of ratios of {R}ankin-{S}elberg {L}-functions and
		applications to {D}eligne's conjecture. Ann. of Math. (2), to appear; arXiv:2205.15382.
		
		\bibitem[Clo90]{Clo}
		Laurent Clozel.
		\newblock Motifs et formes automorphes: applications du principe de
		fonctorialit\'e.
		\newblock In {\em Automorphic forms, {S}himura varieties, and {$L$}-functions,
			{V}ol. {I} ({A}nn {A}rbor, {MI}, 1988)}, volume~10 of {\em Perspect. Math.},
		pages 77--159. Academic Press, Boston, MA, 1990.
		
		\bibitem[DC91]{Du}
		Fokko Du~Cloux.
		\newblock Sur les repr\'{e}sentations diff\'{e}rentiables des groupes de {L}ie
		alg\'{e}briques.
		\newblock {\em Ann. Sci. \'{E}cole Norm. Sup. (4)}, 24(3):257--318, 1991.
		
		\bibitem[DX17]{DX}
		Chao-Ping Dong and Huajian Xue.
		\newblock On the nonvanishing hypothesis for {R}ankin-{S}elberg convolutions
		for {${\rm GL}_n(\Bbb C)\times{\rm GL}_n(\Bbb C)$}.
		\newblock {\em Represent. Theory}, 21:151--171, 2017.
		
		\bibitem[GW09]{GTM255}
		Roe Goodman and Nolan~R. Wallach.
		\newblock {\em Symmetry, representations, and invariants}, volume 255 of {\em
			Graduate Texts in Mathematics}.
		\newblock Springer, Dordrecht, 2009.
		
		\bibitem[HM62]{HM62}
		G.~Hochschild and G.~D. Mostow.
		\newblock Cohomology of {L}ie groups.
		\newblock {\em Illinois J. Math.}, 6:367--401, 1962.
		
		\bibitem[Jac79]{J1}
		Herv\'e Jacquet.
		\newblock Principal {$L$}-functions of the linear group.
		\newblock In {\em Automorphic forms, representations and {$L$}-functions
			({P}roc. {S}ympos. {P}ure {M}ath., {O}regon {S}tate {U}niv., {C}orvallis,
			{O}re., 1977), {P}art 2}, volume XXXIII of {\em Proc. Sympos. Pure Math.},
		pages 63--86. Amer. Math. Soc., Providence, RI, 1979.
		
		\bibitem[Jac09]{J09}
		Herv\'e Jacquet.
		\newblock Archimedean {R}ankin-{S}elberg integrals.
		\newblock In {\em Automorphic forms and {$L$}-functions {II}. {L}ocal aspects},
		volume 489 of {\em Contemp. Math.}, pages 57--172. Amer. Math. Soc.,
		Providence, RI, 2009.
		
		
		
		\bibitem[Kna94]{K}
		Anthony~W. Knapp.
		\newblock Local {L}anglands correspondence: the {A}rchimedean case.
		\newblock In {\em Motives ({S}eattle, {WA}, 1991)}, volume 55, Part 2 of {\em
			Proc. Sympos. Pure Math.}, pages 393--410. Amer. Math. Soc., Providence, RI,
		1994.

\bibitem[KV95]{Kn95}
		Anthony~W. Knapp and David~A. Vogan, Jr.
		\newblock {\em Cohomological induction and unitary representations}, volume~45
		of {\em Princeton Mathematical Series}.
		\newblock Princeton University Press, Princeton, NJ, 1995.
        
		\bibitem[Kud03]{Ku}
		Stephen~S. Kudla.
		\newblock Tate's thesis.
		\newblock In {\em An introduction to the {L}anglands program ({J}erusalem,
			2001)}, pages 109--131. Birkh\"auser Boston, Boston, MA, 2003.
		
		
		\bibitem[LLS24]{LLS24}
		Jian-Shu Li, Dongwen Liu, and Binyong Sun.
		\newblock Period relations for {R}ankin-{S}elberg convolutions for
		$\mathrm{GL}(n)\times \mathrm{GL}(n-1)$.
		\newblock {\em Compositio Math.}, 160:1871--1915, 2024.
		
		\bibitem[LLSS23]{LLSS}
		Jian-Shu Li, Dongwen Liu, Feng Su, and Binyong Sun.
		\newblock Rankin-{S}elberg convolutions for {${\rm GL}(n)\times {\rm GL}(n)$}
		and {${\rm GL}(n)\times {\rm GL}(n-1)$} for principal series representations.
		\newblock {\em Sci. China Math.}, 66(10):2203--2218, 2023.
		
		\bibitem[Rag16]{Rag2}
		A.~Raghuram.
		\newblock Critical values of {R}ankin-{S}elberg {$L$}-functions for
		{$\text{GL}_n\times\text{GL}_{n-1}$} and the symmetric cube {$L$}-functions
		for {$\text{GL}_2$}.
		\newblock {\em Forum Math.}, 28(3):457--489, 2016.
		
		\bibitem[Sun17]{Sun}
		Binyong Sun.
		\newblock The nonvanishing hypothesis at infinity for {R}ankin-{S}elberg
		convolutions.
		\newblock {\em J. Amer. Math. Soc.}, 30(1):1--25, 2017.
		
		\bibitem[SZ12]{SZ}
		Binyong Sun and Chen-Bo Zhu.
		\newblock Multiplicity one theorems: the {A}rchimedean case.
		\newblock {\em Ann. of Math. (2)}, 175(1):23--44, 2012.
		
		\bibitem[Tat79]{T}
		J.~Tate.
		\newblock Number theoretic background.
		\newblock In {\em Automorphic forms, representations and {$L$}-functions
			({P}roc. {S}ympos. {P}ure {M}ath., {O}regon {S}tate {U}niv., {C}orvallis,
			{O}re., 1977), {P}art 2}, volume XXXIII of {\em Proc. Sympos. Pure Math.},
		pages 3--26. Amer. Math. Soc., Providence, RI, 1979.
		
		\bibitem[Vog78]{Vo78}
		David~A. Vogan, Jr.
		\newblock Gelfand-{K}irillov {D}imension for {H}arish-{C}handra {M}odules.
		\newblock {\em Invent. Math.}, 48:75--98, 1978.
		
		\bibitem[Vog81]{Vo81}
		David~A. Vogan, Jr.
		\newblock {\em Representations of real reductive {L}ie groups}, volume~15 of
		{\em Progress in Mathematics}.
		\newblock Birkh\"auser, Boston, MA, 1981.
		
		\bibitem[VZ84]{VZ}
		David~A. Vogan, Jr. and Gregg~J. Zuckerman.
		\newblock Unitary representations with nonzero cohomology.
		\newblock {\em Compositio Math.}, 53(1):51--90, 1984.
		
		\bibitem[Wal88]{Wa88}
		Nolan~R. Wallach.
		\newblock {\em Real reductive groups. {I}}, volume 132 of {\em Pure and Applied
			Mathematics}.
		\newblock Academic Press, Inc., Boston, MA, 1988.
		
		\bibitem[Wal92]{Wa2}
		Nolan~R. Wallach.
		\newblock {\em Real reductive groups. {II}}, volume 132 of {\em Pure and
			Applied Mathematics}.
		\newblock Academic Press, Inc., Boston, MA, 1992.
		
		\bibitem[Xue18]{X}
		Huajian Xue.
		\newblock Homogeneous distributions on finite dimensional vector spaces.
		\newblock {\em J. Lie Theory}, 28(1):33--41, 2018.
		
		
	\end{thebibliography}

\end{document}